\documentclass[a4paper,11pt]{amsart}

\usepackage{amsmath} 
\usepackage{amssymb} 
\usepackage{amscd} 
\usepackage{amsxtra}
\newcommand{\C}{\mathbb C}
\newcommand{\R}{\mathbb R}
\newcommand{\Z}{\mathbb Z}
\newcommand{\Q}{\mathbb Q}
\newcommand{\T}{\mathbb T}
\newcommand{\F}{\mathbb F}
\newcommand{\Flim}{\F_{\rm lim}} 
\newcommand{\tbF}{\widetilde{\mathbb{F}}}  
\newcommand{\hbF}{\hat{\mathbb F}} 
\newcommand{\bL}{\mathbb L}
 
\newcommand{\Proj}{\mathbb P}

\newcommand{\G}{\mathbb G} 
\newcommand{\D}{\mathbb D} 
\newcommand{\K}{\mathbb K} 
\newcommand{\HH}{\mathbb H} 
\newcommand{\U}{\mathbb U} 

\newcommand{\seminf}{$\frac{\infty}{2}$} 

\newcommand{\ev}{\operatorname{ev}}
\newcommand{\age}{\operatorname{age}}
\newcommand{\Hom}{\operatorname{Hom}}
\newcommand{\End}{\operatorname{End}}

\newcommand{\Pic}{\operatorname{Pic}}
\newcommand{\Ker}{\operatorname{Ker}}

\newcommand{\Image}{\operatorname{Im}}
\newcommand{\Boxop}{\operatorname{Box}}

\newcommand{\id}{\operatorname{id}}
\newcommand{\Grading}{\operatorname{\mathsf{Gr}}}
\newcommand{\Res}{\operatorname{Res}}
\newcommand{\rank}{\operatorname{rank}}

\newcommand{\unit}{\operatorname{\boldsymbol{1}}}
\newcommand{\Span}{\operatorname{Span}}
\newcommand{\Tr}{\operatorname{Tr}}
\newcommand{\Eff}{\operatorname{Eff}}
  
\newcommand{\Mir}{\operatorname{Mir}} 
\newcommand{\ch}{\operatorname{ch}} 
\newcommand{\tch}{\widetilde{\operatorname{ch}}} 
\newcommand{\Td}{\operatorname{Td}} 
\newcommand{\tTd}{\widetilde{\operatorname{Td}}}
\newcommand{\codim}{\operatorname{codim}} 
\newcommand{\Tor}{\operatorname{Tor}} 
\newcommand{\pr}{\operatorname{pr}} 
\newcommand{\Vol}{\operatorname{Vol}} 
\newcommand{\crit}{\operatorname{cr}} 
\newcommand{\Hess}{\operatorname{Hess}} 
\newcommand{\cl}{\operatorname{cl}}

\newcommand{\Aut}{\operatorname{Aut}} 
\newcommand{\Gr}{\operatorname{Gr}} 
\newcommand{\inv}{\operatorname{inv}}
\newcommand{\Log}{\operatorname{Log}} 
\newcommand{\pt}{\operatorname{pt}} 
 
\newcommand{\Lie}{\operatorname{Lie}} 

\newcommand{\bN}{\boldsymbol{N}} 

\newcommand{\bs}{\boldsymbol{s}} 

\newcommand{\sfT}{\mathsf{T}} 
\newcommand{\sfH}{\mathsf{H}} 

\newcommand{\cF}{\mathcal{F}}

\newcommand{\cA}{\mathcal{A}}
\newcommand{\cU}{\mathcal{U}}
\newcommand{\cO}{\mathcal{O}}

\newcommand{\cX}{\mathcal{X}}
\newcommand{\cH}{\mathcal{H}}
\newcommand{\tcH}{\widetilde{\mathcal{H}}}
\newcommand{\hcH}{\widehat{\mathcal{H}}}
\newcommand{\cM}{\mathcal{M}}

\newcommand{\cR}{\mathcal{R}}

\newcommand{\cV}{\mathcal{V}}
\newcommand{\cQ}{\mathcal{Q}}  
\newcommand{\cI}{\mathcal{I}}
\newcommand{\cJ}{\mathcal{J}} 
\newcommand{\cC}{\mathcal{C}}  
 
\newcommand{\cK}{\mathcal{K}} 
\newcommand{\tcF}{\widetilde{\mathcal{F}}}
\newcommand{\cMo}{\mathcal{M}^{\rm o}}
\newcommand{\cRz}{\mathcal{R}^{(0)}}

\newcommand{\hatnabla}{\widehat{\nabla}}
\newcommand{\hD}{\hat{D}}

\newcommand{\hK}{\widehat{K}}
\newcommand{\tC}{\widetilde{C}}
\newcommand{\hrho}{\hat{\rho}} 
\newcommand{\hGamma}{\widehat{\Gamma}}

\newcommand{\frD}{\mathfrak{D}}
\newcommand{\fra}{\mathfrak{a}}

\newcommand{\gl}{\mathfrak{gl}}

\newcommand{\ov}{\overline}
\newcommand{\oi}{{\overline{\imath}}}
\newcommand{\oj}{{\overline{\jmath}}}

\newcommand{\iu}{\sqrt{-1}} 
\newcommand{\Yo}{Y^{\rm o}} 

\newcommand{\ceil}[1]{\lceil #1\rceil}
\newcommand{\floor}[1]{\lfloor #1 \rfloor}

\newtheorem{theorem}{Theorem}[section]
\newtheorem{lemma}[theorem]{Lemma}
\newtheorem{proposition}[theorem]{Proposition}
\newtheorem{proposition-conjecture}[theorem]{Proposition-Conjecture}
\newtheorem{definition-proposition}[theorem]{Definition-Proposition} 
\newtheorem{conjecture}[theorem]{Conjecture}

\theoremstyle{definition}
\newtheorem{definition}[theorem]{Definition}
\newtheorem{remark}[theorem]{Remark}
\newtheorem{example}[theorem]{Example}

\def\pair#1#2{\langle #1,#2\rangle}
\def\parfrac#1#2{\frac{\partial{#1}}{\partial #2}}
\def\corr#1{\left\langle #1 \right\rangle}

\def\del#1{#1} 

\begin{document}

\title{Real and integral structures 
in quantum cohomology I: toric orbifolds} 
\author{Hiroshi Iritani}
\address{Faculty of Mathematics, Kyushu University, 6-10-1, 
Hakozaki, Higashiku, Fukuoka, 812-8581, Japan.}
\email{iritani@math.kyushu-u.ac.jp}
\address{Department of Mathematics, Imperial College London, 
Huxley Building, 180, Queen's Gate, London, 
SW7 2AZ, United Kingdom.}
\email{h.iritani@imperial.ac.uk}  
\begin{abstract}
We study real and integral structures in the space of 
solutions to the quantum differential equations. 
First we show that, under mild conditions, any real structure in 
orbifold quantum cohomology yields a pure and polarized $tt^*$-geometry 
near the large radius limit. 
Secondly, we use mirror symmetry to calculate the ``most natural" 
integral structure in quantum cohomology of toric orbifolds.  
We show that the integral structure 
pulled back from the singularity B-model 
is described only in terms of topological data in the A-model; 
$K$-group and a characteristic class.  
Using integral structures, 
we give a natural explanation why the quantum parameter 
should specialize to a root of unity 
in Ruan's crepant resolution conjecture. 
\end{abstract} 
\maketitle 
\tableofcontents

\section{Introduction} 
Quantum cohomology is a family of 
commutative algebra structures  
$(H^*(X,\C),\circ_\tau)$ on the cohomology 
parametrized by $\tau\in H^*(X,\C)$. 
The structure constants of 
the quantum product $\circ_\tau$ are 
given by power series\footnote
{More precisely, Fourier series 
in the $H^2$-part of $\tau$ and 
power series in the other part of $\tau$. 
} in $\tau$ whose coefficients are genus zero 
Gromov-Witten invariants. 
The real or integral structures on quantum cohomology 
in the usual sense --- the subspaces $H^*(X,\R)$ or 
$H^*(X,\Z)$ of $H^*(X,\C)$ --- 
are not the subject of the present paper. 
We will study hidden real or integral structures which 
lie in the space of solutions to 
\emph{quantum differential equations}.  

Our study of real or integral structures in quantum 
cohomology is motivated by mirror symmetry. 
Classical mirror symmetry for Calabi-Yau manifolds 
states that the Gromov-Witten theory (A-model) 
of a Calabi-Yau manifold $X$ is equivalent to the 
Hodge theory (B-model) of the mirror dual Calabi-Yau $X^\vee$. 
Small quantum cohomology of $X$ 
defines the \emph{A-model variation 
of Hodge structures} (henceforth A-model VHS) on 
$\bigoplus_{p} H^{p,p}(X)$ \cite{morrison-mathaspects,cox-katz}. 
On the other hand, 
the deformation of complex structures of $X^\vee$ 
also defines a variation of Hodge structures 
(B-model VHS) on $H^n(X^\vee)$. 
Mathematically, mirror symmetry can be formulated as 
an isomorphism between the A-model VHS of $X$ and 
the B-model VHS of $X^\vee$. 
While the B-model VHS is naturally equipped with 
the integral local system $H^n(X^\vee,\Z)$, 
the A-model VHS seems to lack such integral 
structures. Then we are led to the question: 
what is the natural integral structure 
on the A-model VHS?  
Our calculation for the toric orbifolds suggests  
that the $K$-group of $X$ should give the 
integral local system in the A-model VHS. 

\vspace{5pt} 
\noindent
{\bf Quantum cohomology as a VHS.}
In this paper, we use the language of 
\emph{semi-infinite variation of Hodge structures} 
due to Barannikov \cite{barannikov-qpI,barannikov-proj} 
to include non Calabi-Yau case in our theory.
We will briefly explain 
how this arises from quantum cohomology.  
It is well-known that 
quantum cohomology associates a one parameter 
family of flat connections, called Dubrovin connection,  
on the trivial vector bundle $H^*(X)\times H^*(X)
\to H^*(X)$: 
\begin{align*}
\nabla_i = \parfrac{}{t^i} + \frac{1}{z} \phi_i \circ_\tau.  
\end{align*} 
Here, $\{\phi_i\}$ is a basis of $H^*(X)$, 
$\{t^i\}$ is a linear co-ordinate system 
on $H^*(X)$ dual to $\{\phi_i\}$ 
and $z\in \C^*$ is a parameter. 
Let $L(\tau,z)$ be the fundamental 
solution to the quantum differential equation 
$\nabla_i L(\tau,z) =0$ 
given by the gravitational descendants 
(see (\ref{eq:fundamentalsol_L})).   
Here, $L(\tau,z)$ is an $\End(H^*(X))$-valued function 
in $(\tau,z)\in H^*(X)\times \C^*$. 
Following Coates-Givental \cite{coates-givental}, 
we introduce an infinite dimensional vector space $\cH^X$ by 
\[
\cH^X := H^*(X) \otimes \C\{z,z^{-1}\}, 
\]
where $\C\{z,z^{-1}\}$ denotes the space of 
holomorphic functions on $\C^*$. 
Via the correspondence $\cH^X \ni v(z) 
\mapsto L(\tau,z)v(z)$, we can think of $\cH^X$ 
as the space of flat sections of the Dubrovin connection. 
The fundamental solution $L(\tau,z)$ 
defines the family of ``semi-infinite" subspaces 
of $\cH^X$: 
\[
\F_\tau := L(\tau,z)^{-1}(H^*(X)\otimes \C\{z\}) \subset \cH^X, \quad 
\tau \in H^*(X), 
\]
where $\C\{z\}$ denotes the space of holomorphic functions on $\C$. 
The semi-infinite flag $\cdots\subset 
z^{-1} \F_\tau \subset \F_\tau \subset z \F_\tau \subset \cdots$ 
satisfies properties 
analogous to the usual finite dimensional VHS: 
\begin{alignat}{2}
\label{eq:introd_Griffiths}
&\parfrac{}{t^i} \F_\tau \subset z^{-1} \F_\tau 
\quad  
&& \text{(Griffiths Transversality)}   \\
\label{eq:introd_bilinearrel} 
&(\F_\tau,\F_\tau)_{\cH^X} \subset \C\{z\} 
\quad 
&& \text{(Bilinear Relations)} 
\end{alignat} 
where $(\alpha,\beta)_{\cH^X} = 
\int_X \alpha(-z) \cup \beta(z)$ for $\alpha,\beta\in \cH^X$.  
We call this family of subspaces the \emph{semi-infinite
variation of Hodge structures} or \seminf VHS. 

\vspace{5pt} 
\noindent
{\bf Real and integral structures.}
The free $\C\{z,z^{-1}\}$-module 
$\cH^X$ can be regarded as 
the space of global sections 
of the trivial vector bundle $\sfH^X$ 
over the space of parameters $z$: 
\[
\cH^X = \Gamma(\C^*, \sfH^X), \quad 
\sfH^X = H^*(X)\times \C^* \to \C^*. 
\] 
Due to the existence of the grading in quantum cohomology, 
Dubrovin connection $\nabla$ can 
be extended in the direction of the parameter $z$. 
The extended flat connection induces, 
via the fundamental solution $L(\tau,z)$, 
the following flat connection $\hatnabla_{z\partial_z}$ 
on $\sfH^X$:   
\[
\hatnabla_{z\partial_z} = z\partial_z + \mu - \frac{\rho}{z} 
\in \End_\C(\cH^X),  
\quad \rho = c_1(X), \quad 
\mu \text{ is given in (\ref{eq:def_mu}).} 
\]
A \emph{real or integral structure} 
(Definition \ref{def:realintstr}) 
on the \seminf VHS is given by the choice of 
a real or integral local system 
underlying the flat bundle  
$(\sfH^X,\hatnabla_{z\partial_z})$. 
A real structure defines a real subbundle 
of $\sfH^X|_{S^1}$ and an involution 
$\kappa_\cH \colon \cH^X \to \cH^X$ satisfying 
$\kappa_\cH(z \alpha) = z^{-1} \kappa_\cH(\alpha)$. 
The involution $\kappa_\cH$ coincides, along $|z|=1$,  
with the complex conjugation of sections 
with respect to the real subbundle. 
For a nice choice of real structures, 
we expect the following properties:  
\begin{alignat}{2}
\label{eq:introd_hodgedecomp} 
&\F_\tau \oplus z^{-1} \kappa_\cH(\F_\tau) = \cH^X, 
\quad
&&\text{(Hodge Decomposition)} \\
\label{eq:introd_bilinearineq}  
&(\kappa_\cH(\alpha),\alpha)_{\cH^X}>0, \quad 
\alpha\in \F_\tau \cap \kappa_{\cH}(\F_\tau)
\setminus\{0\}  
\quad 
&&\text{(Bilinear Inequality)}  
\end{alignat} 
These properties 
(\ref{eq:introd_Griffiths}), (\ref{eq:introd_bilinearrel}), 
(\ref{eq:introd_hodgedecomp}), (\ref{eq:introd_bilinearineq}) 
of \seminf VHS actually reduce to 
the corresponding properties of a finite dimensional 
VHS in the conformal limit 
(see Section \ref{subsec:periods_conformallimit}). 
We call the properties (\ref{eq:introd_hodgedecomp}) and 
(\ref{eq:introd_bilinearineq}) 
\emph{pure} and \emph{polarized} respectively.
First we show that (\ref{eq:introd_hodgedecomp}), 
(\ref{eq:introd_bilinearineq}) indeed hold  
near the ``large radius limit" 
\emph{i.e.} $\tau= -x \omega$, $\Re(x)\to\infty$ 
for some K\"{a}hler class $\omega$, 
under reasonable assumptions on the real 
structures: 

\begin{theorem}[Theorem \ref{thm:pure_polarized}] 
\label{thm:introd_pure_polarized} 
Assume that a real structure is invariant 
under the monodromy (Galois) transformations given by  
$G^{\cH}(\xi), \xi\in H^2(X,\Z)$ 
(see Equation (\ref{eq:GaloisH}) and 
Proposition \ref{prop:char_A_real_int_str}). 
If the condition (\ref{eq:kappa_induces_Inv}) 
(which is empty when $X$ is a manifold) holds,  
$\F_\tau$ is pure (\ref{eq:introd_hodgedecomp}) 
near the large radius limit. 
If moreover the condition 
(\ref{eq:leadingterm_kappaV}) holds and 
$H^*(X) = \bigoplus_p H^{p,p}(X)$, 
$\F_\tau$ is polarized (\ref{eq:introd_bilinearineq}) 
near the large radius limit.  
\end{theorem} 

In the theorem above, 
we allow $X$ to be an orbifold or 
a smooth Deligne-Mumford stack  
(see Theorem \ref{thm:pure_polarized} 
for a more precise statement). 
Given a nice real structure 
satisfying (\ref{eq:introd_hodgedecomp}) 
and (\ref{eq:introd_bilinearineq}),  
quantum cohomology will be endowed with \emph{$tt^*$-geometry}  
due to Cecotti-Vafa 
\cite{cecotti-vafa-top-antitop, 
cecotti-vafa-classification},  
which has also been developed by 
Dubrovin \cite{dubrovin-fusion} and 
Hertling \cite{hertling-tt*}. 
The family 
$\F_\tau \cap \kappa_\cH(\F_\tau)$ of finite 
dimensional Hermitian vector spaces 
(canonically identified with quantum cohomology) 
is equipped with a rich geometric structure, 
called Cecotti-Vafa structure (Proposition \ref{prop:CV-str}). 
$tt^*$-geometry also gives an example of 
a harmonic bundle or a twistor structure 
of Simpson \cite{simpson-mixedtwistor}.  
Closely related results 
have been shown in a more abstract setting for 
TERP structures in \cite{hertling-tt*, hertling-sevenheck} 
and the proof of 
Theorem \ref{thm:introd_pure_polarized} 
looks similar to them. 
In fact, when $X$ is Fano and 
the K\"{a}hler class $\omega$ is $c_1(X)$, 
the conclusions of Theorem \ref{thm:introd_pure_polarized} 
can be deduced from \cite[Theorem 7.3]{hertling-sevenheck}. 


\vspace{5pt}
\noindent
{\bf Integral structures for toric orbifolds.} 
In the case of toric orbifolds, 
we concretely calculate the integral structures 
in quantum cohomology  
corresponding to those in the mirrors. 
A mirror partner of a toric orbifold is given 
by the Landau-Ginzburg model, which consists 
of a family $\{Y_q\}_{q\in \cM}$ of algebraic tori  
and Laurent polynomials $W_q\colon Y_q \to \C$. 
In Section \ref{sec:intstr_via_toricmirrors}, 
we construct B-model \seminf VHS 
from the singularity defined by the Landau-Ginzburg model. 
This is underlain by a canonical integral local 
system formed by the relative cohomology groups 
$R_{\Z,(q,z)} = H^n(Y_q,\{\Re(W_q/z)\ll 0\}, \Z)$. 
Assuming mirror symmetry for toric orbifolds --- 
which will be shown in a 
forthcoming paper \cite{CCIT:toric} --- 
we show the following:  

\begin{theorem}[Theorem \ref{thm:pulledbackintstr}]
\label{thm:introd_pulledbackintstr} 
Let $\cX$ be a weak Fano toric orbifold
given by initial data satisfying 
the condition $\hrho \in \cl(\tilde{C}_\cX)$ 
in Section \ref{subsubsec:weakFano}. 
Assume that mirror symmetry conjecture 
in Section \ref{subsec:mirrorsym_toric} 
and the condition (A3) in Section 
\ref{subsec:Amodel_intstr} hold for $\cX$. 
The integral structure of quantum cohomology of $\cX$ 
pulled back from the Landau-Ginzburg mirror is given 
by the image of the $K$-group of topological 
orbifold vector bundles under the map 
$K(\cX)\to \Gamma(\widetilde{\C^*},\sfH^\cX)$ 
(denoted by $z^{-\mu}z^\rho\Psi$ in the main body of the text):  
\[
[V]\longmapsto 
z^{-\mu}z^\rho 
\frac{1}{(2\pi)^{n/2}} 
\hGamma_{\cX} \cup (2\pi\iu)^{\deg/2} \inv^*\tch([V]).  
\] 
Here the image lies 
in the space of (multi-valued) flat sections of 
$(\sfH^\cX,\hatnabla_{z\partial_z})$ 
and $\hGamma_\cX$ is a universal characteristic class 
of $T\cX$. 
(See Section \ref{subsec:Amodel_intstr} for the notation.)
Under this map, 
Mukai pairing on $K(\cX)$ induces  
the pairing $(\cdot,\cdot)_{\cH^\cX}$ on $\cH^\cX$.  
\end{theorem} 

The integral structures given in 
Theorem \ref{thm:introd_pulledbackintstr} 
make sense for general symplectic orbifolds. 
Furthermore, the real structure induced from this integral structure 
satisfies the conditions in 
Theorem \ref{thm:introd_pure_polarized},  
so in particular yields positive definite $tt^*$ geometry 
on $\bigoplus_p H^{p,p}(X)$ near the large radius limit 
for arbitrary $X$ 
(Definition-Proposition \ref{def-prop:A-model_int}).

The connection between $K$-theory and 
quantum cohomology is compatible 
with the picture of 
\emph{homological mirror symmetry}. 
In string theory, 
there are two types of D-branes --- A-type and B-type ---  
and homological mirror symmetry 
predicts that the category of A-type D-branes 
on $X$ is equivalent to the category of B-type D-branes 
on the mirror $X^\vee$. 
In our case, vector bundles 
on a toric orbifold $\cX$ 
give B-type D-branes and 
Lefschetz thimbles in the Landau-Ginzburg mirror 
give A-type D-branes. 
Via oscillatory integrals, 
a Lefschetz thimble gives a flat section  
of the B-model \seminf VHS of the Landau-Ginzburg model. 
Thus by homological mirror symmetry,   
a vector bundle on $\cX$ 
should also give a flat section 
of the A-model \seminf VHS or quantum cohomology.  
In the context of toric varieties and GKZ system 
associated to it, 
these viewpoints have been emphasized by 
Borisov-Horja \cite{borisov-horja-FM} and 
Hosono \cite{hosono}. 
Borisov-Horja \cite{borisov-horja-FM} identified 
the space of solutions 
to the GKZ system with the complexified 
$K$-group of a toric orbifold;  
A conjecture \cite[Conjecture 6.3]{hosono} 
of Hosono (stated in terms of hypergeometric functions) 
is compatible with the integral structure 
in Theorem \ref{thm:introd_pulledbackintstr}. 
The key step in the proof of 
Theorem \ref{thm:introd_pulledbackintstr} 
is a calculation of an oscillatory integral 
over the special Lefschetz thimble $\Gamma_0$ 
formed by real points. 
It turns out that the Lefschetz thimble 
$\Gamma_0$ corresponds to 
the structure sheaf $\cO_\cX$ 
on the toric orbifold $\cX$.

\vspace{5pt}
\noindent
{\bf Global study of quantum cohomology --- Ruan's conjecture.}
The study of integral structures 
will also be useful to understand the 
\emph{global K\"{a}hler moduli space} 
where quantum cohomology is analytically continued.  
The present project was greatly 
motivated by the joint work \cite{CIT:I} 
with Coates and Tseng, where 
we studied the crepant resolution conjecture 
for some toric examples. 
Ruan's crepant resolution conjecture states that 
when we have a crepant resolution $Y$ of 
a Gorenstein orbifold $\cX$, 
quantum cohomology of $\cX$ and $Y$ 
are related by analytic continuations. 
In the analytic continuation, 
some of the quantum parameters of $Y$ 
are conjectured to specialize to a root of unity 
at the large radius limit point of $\cX$. 
In \cite{CIT:I}, we found 
in some examples that 
the quantum cohomology \seminf VHS's of $\cX$ 
and $Y$ are related by an analytic continuation 
followed by a certain symplectic 
transformation $\U\colon \cH^\cX \to \cH^Y$. 
Incorporating integral structures into 
this picture, we suggest the picture that 
the symplectic transformation $\U$ is induced from 
a (conjectural) isomorphism of $K$-groups 
(McKay correspondence)
so that the following commutes: 
\[
\begin{CD}
K(\cX) @>{\cong}>> K(Y) \\
@V{z^{-\mu}z^{\rho}\Psi}VV  @V{z^{-\mu}z^\rho\Psi} VV  \\
\Gamma(\widetilde{\C^*},\sfH^\cX) @>{\U}>> 
\Gamma(\widetilde{\C^*},\sfH^Y).   
\end{CD} 
\]
Here the vertical maps relate  
the $K$-groups with the space of flat sections in 
quantum cohomology; for toric orbifolds, 
they should be the same as what is given 
in Theorem \ref{thm:introd_pulledbackintstr}.  
The bottom map is induced from 
$\U\colon \cH^\cX = \Gamma(\C^*,\sfH^\cX) \to 
\Gamma(\C^*,\sfH^Y) = \cH^Y$. 
This picture, under certain assumptions, 
gives us a natural explanation 
for the specialization to a root of unity. 
We will use ``integral periods" in the A-model 
to predict specialization values of some 
quantum parameters. 

This paper is organized as follows. 
In Section \ref{sec:generalVHS}, we introduce 
real and integral structures for 
a general graded \seminf VHS. This section 
owes much to Hertling \cite{hertling-tt*}. 
In Section \ref{sec:A-model}, we study 
integral structures in (orbifold) 
quantum cohomology and prove 
Theorem \ref{thm:introd_pure_polarized}. 
In Section \ref{sec:intstr_via_toricmirrors}, 
we calculate the integral structures pulled back 
from the mirror for toric orbifolds. 
In Section \ref{sec:exampleP1tt*}, we calculate 
the $tt^*$-geometry of $\Proj^1$. 
Our calculation recovers the physicists' result 
\cite{cecotti-vafa-exactsigma}. 
The aim here is to demonstrate that 
the Birkhoff factorization calculates 
$tt^*$-geometry perturbatively. 
In Section \ref{sec:integralperiods}, 
we discuss the role of integral structures 
in Ruan's conjecture. 

We remark that the convergence of 
the quantum cohomology is assumed throughout the paper. 
Also we consider only the even parity 
part of the cohomology, \emph{i.e.} 
$H^*(X)$ means  $\bigoplus_k H^{2k}(X)$. 
Note that the orbifold cohomology $H_{\rm orb}^*(\cX)$ 
(equipped with the Chen-Ruan's orbifold cup product)   
is denoted also by $H_{\rm CR}^*(\cX)$ 
in the literature. 

\vspace{5pt} 
\noindent 
{\bf Notes added in v3:} 
Since the first version of the paper was written, 
Katzarkov-Kontsevich-Pantev \cite{KKP} proposed 
a rational structure on a 
$\boldsymbol{\mathsf{nc}}$-Hodge structure 
defined by the same $\hGamma$-class independently, 
based on the calculation on 
quantum cohomology of $\Proj^n$. 
Here a $\boldsymbol{\mathsf{nc}}$-Hodge structure 
corresponds to a semi-infinite Hodge structure 
in this paper. 
They also imposed the condition that 
a rational structure is compatible with 
the Stokes structure \cite[Definition 2.5]{KKP}. 

The results on the $\hGamma$-integral structure 
and mirror symmetry for toric orbifolds in this paper 
were revised in the paper \cite{iritani-Int}.  
This revision does not contain the results on real structures, 
but contains more details on toric mirror symmetry. 

\vspace{5pt}  
\noindent  
{\bf Acknowledgments} 
Thanks are due to 
Tom Coates, 
Alessio Corti, 
Hsian-Hua Tseng 
for useful discussions and 
their encouragement. 
This project is motivated by the joint work 
\cite{CIT:I, CCIT:An, CCIT:toric} with them. 
The author is grateful to 
Martin Guest 
for useful discussions 
on the loop group and Iwasawa factorization. 
Guest also had the idea \cite{guest_durham, guest-qc_int} 
to incorporate real structures in quantum cohomology.  
The author would like to thank 
Jim Bryan, Claus Hertling, Yongbin Ruan 
and the referees of the paper 
for very valuable comments. 
This research was supported by 
Inoue Research Award for Young Scientists, 
Grant-in-Aid for Young Scientists (B), 
19740039, 2007 
and EPSRC(EP/E022162/1). 
\vspace{5pt} 
 
\noindent{\bf Notation}
\vspace{5pt}
 
\begin{tabular}{ll}
$\iu$ & imaginary unit $\iu^2=-1$ \\ 
$\cM$ & complex analytic space \\ 
$\D_0 \subset \C$ & disc $\{z\in \C\;;\; |z|\le 1\}$\\
$\D_\infty \subset \Proj^1\setminus \{0\}$ & 
disc $\{z\in \C\cup \{\infty\}=\Proj^1 \; ; \; |z| \ge 1\} $  \\
$(-)\colon \cM\times \C\to \cM\times \C$ & 
map defined by $(\tau,z)\mapsto (\tau,-z)$ \\
$\cX$ & smooth Deligne-Mumford stack  \\
$I\cX$ & inertia stack of $\cX$ \\
$\sfT=\{0\}\cup \sfT'$ & index set of inertia components; \\
$\inv \colon I\cX \to I\cX, \ \sfT \to \sfT$ 
& involution $(x,g)\mapsto (x,g^{-1})$ \\ 
$\iota_v$ & age of inertia component $v\in \sfT$ \\
$n$, $n_v$ & $\dim_\C \cX$, $\dim_\C \cX_v$ \\ 
$\C\{z,z^{-1}\}$, $\C\{z\}$, $\C\{z^{-1}\}$ 
& the space of holomorphic functions on $\C^*$, $\C$, 
$\Proj^1\setminus\{0\}$. 
\end{tabular}

\section{Real and integral structures on \seminf VHS} 
\label{sec:generalVHS}
We introduce real and integral structures for 
a semi-infinite variation of Hodge structures or \seminf VHS. 
We explain that a \seminf VHS with a real structure 
produces a Cecotti-Vafa structure if it is pure. 
A \seminf VHS was originally introduced by 
Barannikov \cite{barannikov-qpI, barannikov-proj}.  
A \seminf VHS with a real structure considered here corresponds to 
the TERP structure due to Hertling \cite{hertling-tt*} 
(see Remark \ref{rem:relation_hertling}). 
The exposition here largely follows 
the line of \cite{hertling-tt*, CIT:I}. 

\subsection{Definition}  
\label{subsec:realintstr_VHS_def} 
Let $\cM$ be a smooth complex analytic space. 
Let $\cO_\cM$ be the analytic structure sheaf on $\cM$. 
Let $\cO_{\cM}\{z\}:=\pi_*(\cO_{\cM\times \C})$ be the push-forward 
of $\cO_{\cM\times \C}$ by the projection 
$\pi\colon \cM\times \C\to \cM$. 
Here $z$ is a co-ordinate on the $\C$ factor. 
Similarly, we set $\cO_{\cM}\{z^{-1}\}:=
\pi_*(\cO_{\cM\times (\Proj^1\setminus\{0\})})$, 
$\cO_{\cM}\{z,z^{-1}\} := \pi_*(\cO_{\cM\times \C^*})$. 
Let $\C\{z\}$, $\C\{z,z^{-1}\}$ and 
$\C\{z^{-1}\}$ be the space of 
holomorphic functions on $\C$, $\C^*$ 
and $\Proj^1\setminus\{0\}$ respectively. 
Let $\Omega_\cM^1$ be the sheaf 
of holomorphic 1-forms on $\cM$ 
and $\Theta_\cM$ be the sheaf 
of holomorphic tangent vector fields on $\cM$.  

\begin{definition}[\cite{CIT:I}]  
A \emph{semi-infinite variation of Hodge structures}, or \seminf VHS 
is a locally free $\cO_\cM\{z\}$-module $\cF$ of rank $N$ 
endowed with a holomorphic flat connection 
\[ 
\nabla\colon \cF\rightarrow z^{-1} \cF \otimes \Omega^1_{\cM}
\]
and a perfect pairing 
\[
(\cdot,\cdot)_\cF\colon 
\cF\times \cF \to \cO_{\cM}\{z\} 
\]
satisfying 
\begin{align*}
\nabla_X (f s) &= (Xf)s + f \nabla_X s, \\
[\nabla_X, \nabla_Y] s & = \nabla_{[X,Y]} s, \\  
(s_1, f(z)s_2)_{\cF} &= (f(-z)s_1,s_2)_{\cF} = f(z)(s_1,s_2)_{\cF},  \\
(s_1, s_2)_{\cF} & = (s_2, s_1)_{\cF}|_{z\to -z}, \\
X (s_1,s_2)_\cF &= (\nabla_X s_1,s_2)_\cF + (s_1,\nabla_X s_2)_\cF 
\end{align*} 
for sections $s,s_1,s_2$ of $\cF$, 
$f\in \cO_\cM \{z\}$ and $X\in \Theta_\cM$. 
Here, $\nabla_X$ is a map from $\cF$ to $z^{-1}\cF$ 
and $z^{-1} \cF$ is regarded as a submodule 
of $\cF \otimes_{\cO_\cM\{z\}} \cO_\cM\{z,z^{-1}\}$. 
The first two properties are part of the 
definition of a flat connection. 
The pairing $(\cdot,\cdot)_\cF$ 
is perfect in the sense that 
it induces an isomorphism of the fiber 
$\cF_\tau$ at $\tau\in \cM$ with 
$\Hom_{\C\{z\}}(\cF_\tau,\C\{z\})$. 

A {\it graded \seminf VHS} is a \seminf VHS $\cF$ endowed with a 
$\C$-endomorphism $\Grading \colon \cF \to \cF$ and 
an Euler vector field $E\in H^0(\cM,\Theta_\cM)$ satisfying 
\begin{align*}
\Grading(f s_1) &= (2 (z\partial_z + E) f ) s_1 +  f \Grading(s_1), \\
[\Grading, \nabla_X] &= \nabla_{2[E,X]}, \quad X\in \Theta_\cM, \\
2(z\partial_z+E) (s_1,s_2)_{\cF} 
&= (\Grading(s_1),s_2)_{\cF} + 
(s_1,\Grading(s_2))_{\cF} - 2n(s_1,s_2)_{\cF} 
\end{align*} 
where $n\in \C$. 
\end{definition} 

A \seminf VHS is a semi-infinite analogue of the usual 
finite dimensional VHS without a real structure. 
The ``semi-infinite" flag
$\cdots \subset z \cF \subset \cF 
\subset z^{-1} \cF \subset z^{-2}\cF \subset \cdots$ of 
$\cF \otimes_{\cO_\cM\{z\}} \cO_{\cM}\{z,z^{-1}\}$ 
plays the role of the Hodge filtration. 
The flat connection $\nabla_X$ shifts this filtration by one --- 
this is an analogue of the Griffiths transversality. 

The structure of a graded \seminf VHS $\cF$ 
can be rephrased in terms of a locally free sheaf 
$\cRz$ over $\cM\times \C$ with 
a flat connection $\hatnabla$. 
Here $\cRz$ is a locally free 
$\cO_{\cM\times \C}$-module of rank $N$  
such that $\cF = \pi_* \cRz$. 
We define the meromorphic connection $\hatnabla$ 
on $\cRz$ 
\[
\hatnabla\colon 
\cRz \longrightarrow \frac{1}{z}\cRz \otimes \pi^* \Omega^1_{\cM} 
\oplus \cRz \frac{dz}{z^2}  
\]
by  
\begin{equation}
\label{eq:hatnabla}
\hatnabla s := \nabla s 
+ (\frac{1}{2} \Grading(s) - \nabla_E s - \frac{n}{2} s) 
\frac{dz}{z}    
\end{equation} 
for a section $s$ of $\cF = \pi_* \cRz$. 
It is easy to see that the conditions on 
$\Grading$ and $\nabla$ above imply that 
$\hatnabla$ is also flat.  
The pairing $(\cdot,\cdot)_{\cF}$ on $\cF$ 
induces a non-degenerate pairing on $\cRz$: 
\[
(\cdot,\cdot)_{\cRz} \colon 
(-)^*\cRz \otimes \cRz \rightarrow \cO_{\cM\times \C}. 
\]
where $(-)\colon \cM\times \C \to \cM\times \C$ is a map 
$(\tau,z) \mapsto (\tau,-z)$. 
This pairing is flat with respect to 
$\hatnabla$ on $\cRz$ and $(-)^*\hatnabla$ on $(-)^*\cRz$. 
Denote by $\cR$ the restriction of $\cRz$ to $\cM\times \C^*$. 
Since $\hatnabla_{z\partial_z}$ is regular outside $z=0$, 
$\cR$ gives a flat vector bundle on $\cM\times \C^*$.  
Let $R \to \cM\times \C^*$ be the 
$\C$-local system underlying the flat bundle $\cR$.  
This has a pairing $(\cdot,\cdot)_R \colon (-)^*R \otimes_\C R \to \C$ 
induced from $(\cdot,\cdot)_{\cRz}$.  
\begin{definition} 
\label{def:realintstr} 
Let $\cF$ be a graded \seminf VHS with $n\in \Z$. 
A \emph{real structure} on \seminf VHS 
is a sub $\R$-local system $R_\R\to \cM\times \C^*$ of $R$
such that $R = R_\R \oplus \iu R_\R$ 
and the pairing takes values in $\R$ on $R_\R$
\[
(\cdot,\cdot)_R \colon 
(-)^*R_\R \otimes_\R R_\R \rightarrow \R.  
\]
An \emph{integral structure} on \seminf VHS is 
a sub $\Z$-local system $R_\Z \to \cM\times \C^*$ of $R$   
such that $R= R_\Z \otimes_\Z \C$  
and the pairing takes values in $\Z$ on $R_\Z$
\[
(\cdot,\cdot)_R \colon 
(-)^* R_\Z \otimes R_\Z \rightarrow \Z 
\]
and is unimodular \emph{i.e.} induces an isomorphism 
$R_{\Z,(\tau,-z)} \cong \Hom(R_{\Z,(\tau,z)},\Z)$ for 
$(\tau,z)\in \cM\times \C^*$. 
\end{definition} 
\begin{remark}
\label{rem:relation_hertling} 
A graded \seminf VHS with a real structure defined here 
is almost equivalent to a TERP($n$) structure 
introduced by Hertling \cite{hertling-tt*}. 
The only difference is that the flat connection 
$\hatnabla$ in TERP($n$) structure is not assumed 
to arise from a grading operator $\Grading$ 
and an Euler vector field $E$.  
Therefore, a graded \seminf VHS gives a TERP structure, 
but the converse is not true in general. 
For the convenience of the reader, we give  
differences in convention between \cite{hertling-tt*} and us. 
Let $\tilde{\nabla}$, $\tilde{R}$, $\tilde{R}_\R$, 
$\tilde{P}\colon \tilde{R} \otimes (-)^*\tilde{R} \to \C$ 
denote the flat connection, $\C$-local system, sub $\R$-local system  
and a pairing appearing in \cite{hertling-tt*}. 
They are related to our $\hatnabla$, 
$R$, $R_\R$, $(\cdot,\cdot)_{\cRz}$ as 
\begin{align*}
&\tilde\nabla = \hatnabla + \frac{n}{2} \frac{dz}{z}, \\ 
&\tilde{R} = (-z)^{-\frac{n}{2}} R, \quad 
\tilde{R}_{\R} = (-z)^{-\frac{n}{2}} R_\R,    \\ 
&\tilde{P}(s_1, s_2) =  z^n (s_2,s_1)_{\cRz}. 
\end{align*} 
Then $\tilde{R}$ is the local system defined by $\tilde{\nabla}$, 
$\tilde P$ is $\tilde{\nabla}$-flat and 
\[
\tilde{P}(\tilde{R}_{\R,(\tau,z)}\times \tilde{R}_{\R,(\tau,-z)}) 
= z^n \left(
z^{-n/2}R_{\R,(\tau,-z)}, (-z)^{-n/2} R_{\R,(\tau,z)}
\right)_R  
\subset  \iu^n \R.  
\]
\end{remark} 

\subsection{Semi-infinite period map}
\label{subsec:semi-inf_period} 

\begin{definition}
For a graded \seminf VHS $\cF$,  
the \emph{spaces $\cH$, $\cV$ of multi-valued flat sections}  
are defined to be 
\begin{align*}
\cH &:= \{ s\in \Gamma(\widetilde{\cM} \times \C^*, \cR) \;;\; 
\nabla s =0 \}, \\
\cV &:= \{ s\in \Gamma((\cM\times \C^*)\sptilde,\cR) 
\;;\; \hatnabla s =0 \},  
\end{align*} 
where $\widetilde{\cM}$ and $(\cM\times \C^*)\sptilde$ 
are the universal 
covers of $\cM$ and $\cM\times \C^*$ respectively. 
The space $\cH$ is a free $\C\{z,z^{-1}\}$-module, where $\C\{z,z^{-1}\}$
is the space of entire functions on $\C^*$.  
The space $\cV$ is a finite dimensional $\C$-vector space
identified with the fiber of the local system $R$. 
The flat connection $\hatnabla$ and the 
pairing $(\cdot,\cdot)_{\cRz}$ on $\cRz$ 
induce an operator 
\[\hatnabla_{z\partial_z}\colon \cH\to \cH\] 
and a pairing 
\[
(\cdot,\cdot)_{\cH}\colon \cH\times \cH\rightarrow \C\{z,z^{-1}\}  
\]
satisfying 
\begin{align*}
(f(-z) s_1, s_2)_{\cH} &= (s_1, f(z) s_2)_{\cH} =
f(z)(s_1,s_2)_{\cH} \quad 
f(z) \in \C\{z,z^{-1}\}, \\
(s_1,s_2)_{\cH} &= (s_2,s_1)_{\cH}|_{z\mapsto -z} \\
z\partial_z(s_1,s_2)_{\cH}& = 
(\hatnabla_{z\partial_z} s_1, s_2)_{\cH} + 
(s_1,\hatnabla_{z\partial_z} s_2)_{\cH}. 
\end{align*} 
The free $\C\{z,z^{-1}\}$-module 
$\cH$ can be regarded as the space of 
global sections of a flat vector bundle $\sfH \to \C^*$. 
Then $\cV$ can be identified with 
the space of multi-valued flat sections of $\sfH$. 
A pairing $(\cdot,\cdot)_{\cV}\colon \cV\otimes_\C \cV \to \C$ 
is defined by 
\begin{equation}
\label{eq:pairing_def_V}
(s_1,s_2)_{\cV} := (s_1(\tau,e^{\pi \iu} z), s_2(\tau,z))_{R}
\end{equation} 
where $s_1(\tau,e^{\pi \iu}z)\in \cR_{(\tau,-z)}$ 
denote the parallel translation of 
$s_1(\tau,z)\in \cR_{(\tau,z)}$ along the counterclockwise path 
$[0,1]\ni \theta \mapsto e^{\pi \iu \theta}z$.  
\end{definition} 

A \seminf VHS $\cF$ on $\cM$ defines a map 
from $\widetilde{\cM}$ to the 
\emph{Segal-Wilson Grassmannian} of $\cH$. 
For $u \in \cF_\tau$ at $\tau\in \widetilde{\cM}$, 
there exists a unique flat section $s_u\in \cH$ such that 
$s_u(\tau) =u$. 
This defines an embedding of a fiber $\cF_\tau$ into $\cH$:   
\begin{equation}
\label{eq:embedding_J}
\cJ_\tau \colon \cF_\tau \longrightarrow \cH, \quad 
u\longmapsto s_u, \quad 
\tau \in \widetilde{\cM}.  
\end{equation} 
We call the image $\F_\tau\subset \cH$ of this embedding 
the \emph{semi-infinite Hodge structure}. 
This is a free $\C\{z\}$-module of rank $N$. 
The family $\{\F_\tau\subset \cH\}_{\tau\in \widetilde{\cM}}$ 
of subspaces gives the 
\emph{moving subspace realization of \seminf VHS}. 
Fix a $\C\{z,z^{-1}\}$-basis $e_1,\dots, e_N$ of $\cH$. 
Then the image of a local frame 
$s_1,\dots,s_N$ of $\cF$ over $\cO_\cM\{z\}$ 
under $\cJ_\tau$ can be written as 
$\cJ_\tau(s_j) = \sum_{i=1}^N e_i J_{ij}(\tau,z)$. 
When $z$ is restricted to $S^1=\{|z|=1\}$, 
the $N\times N$ matrix $(J_{ij}(\tau,z))$
defines an element of the smooth loop group 
$LGL_N(\C)=C^\infty(S^1,GL_N(\C))$. 
Another choice of a local basis of $\cF$ changes 
the matrix $(J_{ij}(\tau,z))$ by right multiplication by 
a matrix with entries in $\C\{z\}$. 
Thus the Hodge structure $\F_\tau$ gives a point 
$(J_{ij}(\tau,z))_{ij}$ in the smooth Segal-Wilson Grassmannian 
$\Gr_{\frac{\infty}{2}}(\cH) := LGL_N(\C)/L^+GL_N(\C)$ 
\cite{pressley-segal}. Here $L^+GL_N(\C)$ consists 
of smooth loops which are the boundary values 
of holomorphic maps $\{z\in \C\;;\; |z|<1\}\to GL_N(\C)$.  
The map 
\[
\widetilde{\cM} \ni \tau \longmapsto \F_\tau 
\in \Gr_{\frac{\infty}{2}}(\cH)
\]
is called the \emph{semi-infinite period map}. 
\begin{proposition}[{\cite[Proposition 2.9]{CIT:I}}]
\label{prop:property_VHS}
The semi-infinite period map $\tau\mapsto \F_\tau$ satisfies: 
\begin{itemize}
\item[(i)] $X \F_\tau \subset z^{-1} \F_\tau$ for $X\in \Theta_\cM$;   
\item[(ii)] $(\F_\tau,\F_\tau)_{\cH} \subset \C\{z\}$;  
\item[(iii)] $(\hatnabla_{z\partial_z}+E) \F_\tau 
\subset \F_\tau$. In particular, 
$\hatnabla_{z\partial_z}\F_\tau
\subset z^{-1}\F_\tau$.  
\end{itemize} 
The first property {\rm (ii)} 
is an analogue of Griffiths transversality 
and the second {\rm (iii)} is the Hodge-Riemann bilinear relation.
\end{proposition} 

In terms of the flat vector bundle $\sfH\to\C^*$ above 
(such that $\cH = \Gamma(\C^*,\sfH)$), 
the Hodge structure $\F_\tau\subset \cH$ 
is considered to be an extension 
of $\sfH$ to $\C$ such that the flat connection has 
a pole of Poincar\'{e} rank 1 at $z=0$. 

Real and integral structures on \seminf VHS define the following  
subspaces $\cH_\R$, $\cV_\R$, $\cV_\Z$: 
\begin{align*} 
\cH_\R &:= \{ s\in \cH \; ;\; s(\tau,z) \in R_{\R,(\tau,z)}, \ 
(\tau,z)\in \widetilde{\cM}\times S^1\} \\  
\cV_\R &:= \{ s\in \cV \; ;\; s(\tau,z) \in R_{\R,(\tau,z)}, \ 
(\tau,z)\in (\cM\times \C^*)\sptilde \} \\
\cV_\Z &:= \{ s\in \cV \; ;\; s(\tau,z) \in R_{\Z,(\tau,z)}, \ 
(\tau,z) \in (\cM\times \C^*)\sptilde \}   
\end{align*} 
Then $\cH_\R$ becomes a (not necessarily free) module 
over a ring $C^h(S^1,\R)$:
\[
C^h(S^1,\R):=\left\{
f(z) \in \C\{z,z^{-1}\} \;;\; f(z)\in \R \text{ if } |z|=1 \right \}.   
\] 
Clearly, we have $\C\{z,z^{-1}\} = C^h(S^1,\R)\oplus \iu C^h(S^1,\R)$. 
The involution $\kappa$ on $\C\{z,z^{-1}\}$ corresponding to the 
real form $C^h(S^1,\R)$ is given by  
\[
\kappa(f)(z) =\ov{f(\gamma(z))}, \quad f(z) \in \C\{z,z^{-1}\},  
\]  
where $\gamma(z) = 1/\ov{z}$ and the $\ov{\phantom{A}}$ 
in the right-hand side 
is the complex conjugate. 
We also have $\cH \cong \cH_\R \oplus \iu \cH_\R$.   
This real form $\cH_\R\subset \cH$ defines 
an involution $\kappa_\cH\colon \cH\rightarrow \cH$ such that 
\begin{align}
\label{eq:property_kappaH}
\begin{split}
\kappa_\cH (f s) &= \kappa(f)\kappa_{\cH}(s),  \\   
\kappa_\cH \hatnabla_{z\partial_z} &= - 
\hatnabla_{z\partial_z} \kappa_{\cH}, \\ 
\kappa((s_1,s_2)_{\cH}) &= (\kappa_{\cH}(s_1), 
\kappa_{\cH}(s_2))_{\cH}.   
\end{split}
\end{align} 
Similarly, we have $\cV= \cV_\R \oplus \iu \cV_\R$; 
we denote by $\kappa_{\cV}\colon \cV\to\cV$ 
the involution defined by the real structure $\cV_\R$.  
\del{
\begin{remark}
\label{rem:Cinfty} 
In the context of the smooth Grassmannian, 
it is more natural to work over $C^\infty(S^1,\C)$  
instead of $\C\{z,z^{-1}\}$. We put 
\[
\tcH := \cH \otimes_{\C\{z,z^{-1}\}} C^\infty(S^1,\C), \quad 
\tbF_\tau := \F_\tau \otimes_{\C\{z\}} \cO(\D_0),  
\]
where $\cO(\D_0)$ is a subspace of $C^\infty(S^1,\C)$ consisting of 
functions which are the boundary values of holomorphic functions 
on the interior of the disc $\D_0=\{z\in \C\;;\; |z|\le 1\}$. 
The involution $\kappa_\cH\colon \tcH\to \tcH$ and 
the real form $\tcH_\R$ is defined similarly 
and the same properties hold. 
Conversely, using the flat connection 
$\hatnabla_{z\partial_z}$ in the $z$-direction, 
one can recover $\F_\tau$ from $\tbF_\tau$ 
since flat sections of $\hatnabla_{z\partial_z}$ 
determine an extension of the bundle on $\D_0$ to $\C$. 
\end{remark}
}

\subsection{Pure and polarized \seminf VHS}
Following Hertling \cite{hertling-tt*}, 
we define an extension $\hK$ of $\cRz$ across $z=\infty$.  
The properties ``pure and polarized" for $\cF$ 
are defined in terms of this extension. 


\begin{definition}[Extension of $\cRz$ across $z=\infty$] 
\label{def:extension_infty}  
Let $\gamma\colon \cM\times \Proj^1 \rightarrow \cM\times \Proj^1$ 
be the map defined by $\gamma(\tau,z) = (\tau, 1/\ov{z})$. 
Let $\ov{\cM}$ denote the complex conjugate of $\cM$, \emph{i.e.} 
$\ov{\cM}$ is the same as $\cM$ as a real-analytic manifold 
but holomorphic functions on $\ov{\cM}$ are  
anti-holomorphic functions on $\cM$. 
The pull-back $\gamma^*\cRz$ of $\cRz$ 
has the structure of an 
$\cO_{\cM\times \ov{(\Proj^1\setminus\{0\})}}$-module. 
Thus its complex conjugate $\ov{\gamma^*\cRz}$ 
has the structure of 
an $\cO_{\ov{\cM}\times (\Proj^1\setminus \{0\})}$-module.  
Regarding $\cRz$ and $\ov{\gamma^*\cRz}$ 
as real-analytic vector bundles over $\cM\times \C$ 
and $\cM\times (\Proj^1\setminus\{0\})$, 
we glue them along $\cM\times \C^*$ by the fiberwise map  
\begin{equation}
\label{eq:gluingmap}
\begin{CD}
\cRz_{(\tau,z)} @>{\kappa}>> \ov{\cRz}_{(\tau,z)} 
@>{P(\gamma(z),z)}>> \ov{\cRz}_{(\tau,\gamma(z))} 
= \ov{\gamma^*\cRz}_{(\tau,z)},  
\end{CD} 
\quad z\in \C^*. 
\end{equation} 
Here the first map $\kappa$ is the real involution on 
$\cRz_{(\tau,z)}$ with respect to 
the real form $R_{\R,(\tau,z)}$ 
and the second map $P(\gamma(z),z)$ 
is the parallel translation for 
the flat connection $\hatnabla$ along the path 
$[0,1] \ni t \mapsto (1-t) z + t \gamma(z)$. 
Define $\hK\to \cM\times \Proj^1$ to 
be the real-analytic complex vector bundle 
obtained by gluing $\cRz$ and $\ov{\gamma^*\cRz}$ in this way. 
Notice that $\hK|_{\tau\times \Proj^1}$ 
has the structure of a holomorphic vector bundle 
since the gluing map (\ref{eq:gluingmap}) 
preserves the holomorphic structure 
in the $\Proj^1$-direction. 
\end{definition} 

\begin{definition} 
\label{def:pure} 
A graded \seminf VHS $\cF$ with a real structure 
is called \emph{pure} at $\tau\in \cM$ if 
$\hK|_{\{\tau\}\times \Proj^1}$ 
is trivial as a holomorphic vector bundle on $\Proj^1$. 
\end{definition} 

A pure graded \seminf VHS with a real structure here 
corresponds to the (trTERP) structure in \cite{hertling-tt*}. 
Here we follow the terminology in \cite{hertling-sevenheck}. 

We rephrase the purity in terms of the moving 
subspace realization
$\{\F_\tau \subset \cH\}$. 
When we identify $\cH$ with the space of global 
sections of $\hK|_{\{\tau\}\times \C^*} = 
\cR|_{\{\tau\}\times \C^*}$, 
it is easy to see that the involution 
$\kappa_\cH\colon \cH\to \cH$ is induced 
by the gluing map (\ref{eq:gluingmap}). 
Then $\F_\tau$ is identified with 
the space of holomorphic sections of $\hK|_{\{\tau\}\times \C^*}$ 
which can extend to $\{\tau\}\times \C$; 
$\kappa_\cH(\F_\tau)$ is identified with the space of holomorphic 
sections of $\hK|_{\{\tau\}\times \C^*}$ which can extend to 
$\{\tau\}\times (\Proj^1\setminus\{0\})$. 
\del{Similarly, 
$\tbF_\tau$ (resp. $\kappa_{\cH}(\tbF_\tau)$) is 
identified with the space of smooth sections 
of $\hK|_{\{\tau\}\times S^1}$ which can extend 
to holomorphic sections on $\D_0$ (resp. $\D_\infty$), 
where $\D_0 = \{z\in \C\;;\; |z|\le 1\}$, 
$\D_\infty = \{z\in \C\cup \{\infty\} 
= \Proj^1 \;;\; |z| \ge 1\}$ 
and $\tbF_\tau$ 
is the space in Remark \ref{rem:Cinfty}.} 


\begin{proposition}
\label{prop:purity_condition} 
A graded \seminf VHS $\cF$ with a real structure 
is pure at $\tau\in \cM$ if and only if 
one of the following natural maps is an isomorphism: 
\begin{align} 
\label{eq:FcapkappaF}
\F_\tau \cap \kappa_{\cH}(\F_\tau) &\longrightarrow   \F_\tau/z\F_\tau, \\
\label{eq:FcapHreal} 
(\F_\tau \cap \cH_\R)\otimes \C &\longrightarrow  \F_\tau/z\F_\tau, \\ 
\label{eq:FpluskappaF} 
\F_\tau \oplus z^{-1} \kappa_{\cH}(\F_\tau) & \longrightarrow  \cH. 
\end{align} 
\del{This holds also true when $\F_\tau$, $\cH$, $\cH_\R$ are replaced with 
$\tbF_\tau$, $\tcH$, $\tcH_\R$ in Remark \ref{rem:Cinfty}. } 
When $\cF$ is pure at some $\tau$, 
$\cH_\R$ is a free module over $C^h(S^1,\R)$.  
\end{proposition} 
\begin{proof}
Under the identifications we explained above, 
$\F_\tau\cap \kappa_{\cH}(\F_\tau)$ is 
identified with the space of global sections of 
$\hK|_{\{\tau \}\times \Proj^1}$ 
and the natural map 
$\F_\tau \cap \kappa_{\cH}(\F_\tau) \to \F_\tau/z\F_\tau$ 
corresponds to the restriction to $z=0$ 
(note that $\F_\tau/z\F_\tau \cong \hK_{(\tau,0)}$). 
Therefore (\ref{eq:FcapkappaF}) is an isomorphism 
if and only if $K|_{\{\tau \}\times \Proj^1}$ is trivial. 
$\F_\tau\cap \kappa_\cH(\F_\tau)$ 
is invariant under $\kappa_{\cH}$ 
and its real form is given by $\F_\tau\cap \cH_\R$. 
Therefore, we have 
$\F_\tau\cap \kappa_{\cH}(\F_\tau) 
\cong (\F_\tau\cap \cH_\R)\otimes \C$. 
Thus (\ref{eq:FcapkappaF}) is an isomorphism if and only if 
so is (\ref{eq:FcapHreal}). 
Similarly, we can see that 
(\ref{eq:FpluskappaF}) is an isomorphism 
if $\hK|_{\{\tau \}\times \Proj^1}$ is trivial. 
Conversely, we show that (\ref{eq:FcapkappaF}) is an isomorphism 
if so is (\ref{eq:FpluskappaF}).   
The injectivity of the map 
$\F_\tau \cap \kappa_{\cH}(\F_\tau) \to \F_\tau/z\F_\tau$ 
is easy to check. 
Take $v\in \F_\tau$. By assumption, $z^{-1}v= v_1 + v_2$ for some 
$v_1 \in \F_\tau$ and $v_2 \in z^{-1}\kappa_{\cH}(\F_\tau)$. 
Thus $v-zv_1 = z v_2 \in \F_\tau\cap \kappa_{\cH}(\F_\tau)$ and 
the image of this element in $\F_\tau/z\F_\tau$ is $[v]$.  
\del{The discussion on the spaces 
$\tbF_\tau$, $\tcH$ and $\tcH_\R$ are similar.}



The last statement: Since $\F_\tau\cap \kappa_{\cH}(\F_\tau) \cong 
(\F_\tau \cap \cH_\R)\otimes \C$, we can take a global basis 
of the trivial bundle $\hK|_{\{\tau \}\times \Proj^1}$ 
from $\F_\tau\cap \cH_\R$. The module $\cH_\R$ is freely 
generated by such a basis over $C^h(S^1,\R)$. 
\end{proof}

\begin{definition} 
\label{def:polarized} 
A graded \seminf VHS $\cF$ 
with a real structure is 
called \emph{polarized} at $\tau\in \cM$ 
if the Hermitian pairing $h$ on 
$\F_\tau \cap \kappa_{\cH}(\F_\tau) 
\cong \Gamma(\Proj^1,\hK|_{\{\tau\}\times \Proj^1})$ 
defined by 
\[
h\colon s_1 \times s_2 \longmapsto  (\kappa_{\cH}(s_1), s_2)_{\cH} 
\]
is positive definite. Note that this pairing 
takes values in $\C$ since 
$(\F_\tau,\F_\tau)_{\cH}\subset \C\{z\}$ and  
$(\kappa_{\cH}(\F_\tau),\kappa_\cH(\F_\tau))_{\cH} 
\subset \C\{z^{-1}\}$ by (\ref{eq:property_kappaH}). 
It is easy to show that a polarized \seminf VHS is 
necessarily pure at the same point. 
\end{definition}

\del{
\begin{remark} 
\label{rem:Birkhoff_Iwasawa} 
In order to obtain a basis of 
$\tbF_\tau\cap \kappa_{\cH}(\tbF_\tau)$ 
or $\tbF_\tau \cap \tcH_\R$,  
we can make use of Birkhoff or Iwasawa factorization.  
Take an $\cO(\D_0)$-basis $s_1,\dots,s_N$ of $\tbF_\tau$. 
Define an element $A(z)=(A_{ij}(z))$ 
of the loop group $LGL_N(\C)$ by 
\[
[\kappa_\cH(s_1),\dots,\kappa_\cH(s_N)] = [s_1,\dots,s_N] A(z), 
\quad \text{\emph{i.e.} } \kappa_\cH(s_i) = \sum_{j} s_j A_{ji}(z). 
\]
If $A(z)$ admits the Birkhoff factorization 
$A(z) = B(z) C(z)$, 
where $B(z)$ and $C(z)$ are holomorphic maps 
$B(z)\colon \D_0 \to GL_N(\C)$,  
$C(z)\colon \D_{\infty} \to GL_N(\C)$ such that $B(0)=\unit$, 
then we obtain a $\C$-basis of 
$\tbF_\tau \cap \kappa_{\cH}(\tbF_\tau)$ as 
\begin{equation}
\label{eq:A=BC}
[\kappa_\cH(s_1),\dots,\kappa_\cH(s_N)] C(z)^{-1} 
= [s_1,\dots,s_N] B(z).   
\end{equation} 
Here, $\cF$ is pure at $\tau\in \cM$ if and only if 
$A(z)$ admits the Birkhoff factorization, 
\emph{i.e.} $A(z)$ is in the ``big cell" 
of the loop group. 
In particular, the purity is 
an open condition for $\tau \in \cM$. 
On the other hand, the Iwasawa-type 
factorization appears as follows. 
Assume that we have a basis $e_1,\dots,e_N$ 
of $\tcH_\R$ over $C^\infty(S^1,\R)$ such that 
$(e_i,e_j)_{\tcH} =\delta_{ij}$ 
and a basis $s_1,\dots,s_N$ of $\tbF_\tau$ over $\cO(\D_0)$ 
such that $(s_i,s_j)_{\tcH}= \delta_{ij}$. 
Define a matrix $J(z)$ by 
\[
[s_1,\dots, s_N] = [e_1,\dots,e_N] J(z).  
\]
This $J(z)$ lies in the {\it twisted loop group} 
$LGL_N(\C)_{\rm tw}$: 
\[
LGL_N(\C)_{\rm tw} := 
\{J \colon S^1 \to GL_N(\C)\;;\;
J(-z)^{\rm T} J(z) = \unit \}.   
\]
If $J(z)$ admits an 
Iwasawa-type factorization $J(z)=U(z)B(z)$, where 
$U\colon S^1 \to GL_N(\R)$ with 
$U(-z)^{\rm T} U(z) = \unit$ and  
$B\colon \D_0 \to GL_N(\C)$ with $B(-z)^{\rm T} B(z) =\unit$, 
then we obtain an $\R$-basis of $\tbF_\tau \cap \tcH_\R$ as 
\[
[s_1,\dots,s_N] B(z)^{-1} = 
[e_1,\dots,e_N] U(z)    
\]
which is orthonormal with respect to $(\cdot,\cdot)_{\tcH}$. 
In this case, the pairing $(\cdot,\cdot)_{\tcH}$ 
restricted to $\tbF_\tau \cap \tcH_\R$ is an 
$\R$-valued \emph{positive definite}  
symmetric form. 
The map $\tau \mapsto J(z)$ gives rise to 
the semi-infinite period map in 
Section \ref{subsec:semi-inf_period}: 
\[
\cM \ni \tau \longmapsto [J(z)] \in LGL_N(\C)_{\rm tw}/ 
LGL^+_N(\C)_{\rm tw}.  
\]
Here, $\cF$ is pure at $\tau$ 
and $(\tbF_\tau \cap \tcH_\R, (\cdot,\cdot)_{\tcH})$ is 
positive definite if and only if the image of this map 
lies in the $LGL_N(\R)_{\rm tw}$-orbit of $[\unit]$. 
This orbit is open, but not dense.   
We owe the Lie group theoretic viewpoint here 
to Guest \cite{guest_durham, guest-qc_int}. 
\end{remark} 
}

\subsection{Cecotti-Vafa structure} 
We describe the Cecotti-Vafa structure ($tt^*$-geometry) 
associated to a pure graded \seminf VHS with a real structure. 

Define a complex vector bundle $K\to\cM$ by 
$K:= \hK|_{\cM\times \{0\}}$. 
This is the real analytic vector bundle 
underlying $\cF/z\cF \cong \cRz|_{\cM\times \{0\}}$. 
Let $\cA_{\cM}^p$ be the sheaf of complex-valued 
$C^\infty$ $p$-forms on $\cM$ and 
$\cA_{\cM}^1 = \cA_{\cM}^{1,0} \oplus \cA_{\cM}^{0,1}$ 
be the type decomposition. 
 

\begin{proposition}[{\cite[Theorem 2.19]{hertling-tt*}}]
\label{prop:CV-str} 
Assume that a graded \seminf VHS $\cF$ with a real structure 
is pure over $\cM$. 
Then the vector bundle $K$ is equipped 
with a Cecotti-Vafa structure 
$(\kappa, g, C,\tC, D, \cQ, \cU,\ov\cU)$. 
This is given by the data (see (\ref{eq:defofkappa}), 
(\ref{eq:defofg}), (\ref{eq:defofCCD}), (\ref{eq:defofUUQ})):  
\begin{itemize} 
\item A complex-antilinear involution $\kappa \colon K_\tau \to K_\tau$; 
\item A non-degenerate, symmetric, $\C$-bilinear metric 
$g\colon K_\tau \times K_\tau \to \C$ 
which is real with respect to $\kappa$, i.e. 
$g(\kappa u_1, \kappa u_2) = \ov{g(u_1,u_2)}$;   
\item Endomorphisms $C \in \End(K)\otimes \cA^{1,0}_\cM$, 
$\tC\in \End(K)\otimes \cA^{0,1}_\cM$ such that 
$\tC_{\oi} = \kappa C_i \kappa$;  
\item A connection $D\colon K\to K\otimes \cA^1_\cM$ 
real with respect to $\kappa$, i.e.
$D_{\oi} = \kappa D_i \kappa$; 
\item Endomorphisms $\cQ,\cU,\ov\cU\in \End(K)$ 
such that $\cU = C_E$, $\ov\cU=\kappa \cU \kappa = \tC_{\ov{E}}$ 
and $\cQ \kappa = -\kappa \cQ$   
\end{itemize}  
satisfying the integrability conditions 
\begin{align*}
&[D_i,D_j]=0, \quad D_iC_j - D_j C_i=0, \quad [C_i,C_j]=0, \\ 
&[D_\oi,D_\oj]=0, \quad  D_\oi \tC_\oj - D_\oj \tC_\oi=0, \quad 
[\tC_\oi, \tC_\oj]=0, \\
& D_i \tC_\oj =0, \quad D_\oi C_j =0,\quad  
[D_i,D_\oj] + [C_i,\tC_\oj]=0, \\ 
&D_i\ov\cU=0, \quad D_i \cQ - [\ov\cU, C_i]=0, 
\quad D_i\cU - C_i + [\cQ,C_i]=0, \quad [\cU_,C_i]=0, \\  
&D_\oi\cU=0, \quad D_\oi \cQ + [\cU, \tC_\oi]=0, \quad  
D_\oi\ov\cU - \tC_\oi - [\cQ,\tC_\oi]=0, \quad [\ov\cU,\tC_\oi]=0, 
\end{align*} 
and the compatibility with the metric 
\begin{align*}
&\partial_i g(u_1,u_2) = g(D_iu_1,u_2)+g(u_1,D_iu_2), \\
&\partial_\oi g(u_1,u_2) = g(D_\oi u_1, u_2) + g(u_1, D_\oi u_2), \\
& g(C_i u_1, u_2) = g(u_1, C_iu_2),\quad 
g(\tC_\oi u_1, u_2) = g(u_1, \tC_\oi u_2), \\  
& g(\cU u_1, u_2) = g(u_1, \cU u_2), \quad 
g(\ov\cU u_1, u_2) = g(u_1, \ov\cU u_2), \\ 
& g(\cQ u_1, u_2)+g(u_1,\cQ u_2)=0.   
\end{align*} 
Here we chose a local complex co-ordinate system $\{t^i\}$
on $\cM$ and used the notation 
$D_i = D_{\partial/\partial t^i}$, 
$D_\oi = D_{\partial/\partial \ov{t^i}}$, etc. 
\end{proposition} 

A concrete example of the Cecotti-Vafa structure 
will be given in Section \ref{sec:exampleP1tt*}. 
We explain the construction of the above data 
from the \seminf VHS $\cF$. 
Because $\cF$ is pure, we have a canonical identification 
\[
K_\tau \cong \Gamma(\Proj^1,\hK|_{\{\tau\}\times \Proj^1}) 
\cong \F_\tau \cap \kappa_{\cH}(\F_\tau).  
\]
The involution $\kappa_\cH$ and the pairing $(\cdot,\cdot)_\cH$ 
restricted to $\F_\tau \cap \kappa_{\cH}(\F_\tau)$ induce 
an involution $\kappa$ and a $\C$-bilinear pairing 
$g$ on $K_\tau$: 
\begin{align}
\label{eq:defofkappa}
\Phi_\tau(\kappa(u)) & := \kappa_\cH(\Phi_\tau(u)), \\
\label{eq:defofg} 
g(u_1,u_2) & := (\Phi_\tau(u_1),\Phi_\tau(u_2))_{\cH} 
\in \C 
\end{align} 
satisfying 
\[
g(\kappa u_1, \kappa u_2 ) = \ov{g(u_1,u_2)}, 
\quad g(u_1,u_2) = g(u_2,u_1).  
\]
Note that the subspace $\F_\tau\cap \kappa_{\cH}(\F_\tau)$ 
depends on the parameter $\tau$ real analytically.  
A $C^\infty$-version of the 
Griffiths transversality gives 
\begin{align*} 
&X^{(1,0)} \F_\tau \subset z^{-1} \F_\tau, 
& & X^{(0,1)} \F_\tau \subset \F_\tau, \\
&X^{(1,0)} \kappa_\cH(\F_\tau) \subset \kappa_\cH(\F_\tau), & 
&X^{(0,1)} \kappa_\cH(\F_\tau) \subset z \kappa_\cH(\F_\tau),   
\end{align*} 
where $X^{(1,0)}\in T^{1,0}\cM$ and $X^{(0,1)}\in T^{0,1}\cM$.  
For $X^{(1,0)}\in T_\tau^{1,0}\cM$, we have 
\[
X^{(1,0)}(\F_\tau \cap \kappa_\cH(\F_\tau) )
\subset z^{-1} \F_\tau \cap \kappa_{\cH}(\F_\tau) =  
z^{-1}(\F_\tau \cap \kappa_\cH(\F_\tau)) \oplus 
(\F_\tau \cap \kappa_\cH(\F_\tau)). 
\] 
Similarly for $X^{(0,1)}\in T_\tau^{(0,1)}\cM$, we have  
\begin{align*}
X^{(0,1)} (\F_\tau \cap \kappa_\cH(\F_\tau)) 
\subset (\F_\tau \cap \kappa_\cH(\F_\tau)) \oplus 
z (\F_\tau \cap \kappa_\cH(\F_\tau)). 
\end{align*} 
Hence we can define endomorphisms 
$C\colon K \to K\otimes \cA^{1,0}$, 
$\tC\colon K \to K\otimes \cA^{0,1}$,  
and a connection $D \colon K\to K\otimes \cA^1$ by 
\begin{equation}
\label{eq:defofCCD}
X\Phi_\tau(u_\tau) 
= z^{-1} \Phi_\tau(C_X(u_\tau)) + \Phi_\tau(D_X(u_\tau)) 
+ z \Phi_\tau(\tC_X(u_\tau))
\end{equation} 
for a section $u_\tau$ of $K$. 
By applying $\kappa_\cH$ on the both hand sides, 
\[
\ov{X} \Phi_\tau(\kappa u_\tau) 
= z^{-1} \Phi_\tau ( \kappa \tC_X(u_\tau))
+ \Phi_\tau(\kappa D_X(u_\tau)) + 
z \Phi_\tau (\kappa C_X(u_\tau)). 
\]
Therefore, we must have 
\[
C_{\ov{X}} \kappa = \kappa  \tC_X, \quad  
\kappa   D_X = D_{\ov{X}}  \kappa, \quad 
X\in T\cM \otimes_\R \C.   
\] 
Similarly, we can define endomorphisms 
$\cU,\ov{\cU},\cQ \colon K\to K$ by 
\begin{align}
\label{eq:defofUUQ} 
\hatnabla_{z\partial_z} \Phi_\tau(u_\tau) 
= -z^{-1} \Phi_\tau(\cU(u_\tau)) +  
\Phi_\tau(\cQ(u_\tau)) + z \Phi_\tau(\ov{\cU}(u_\tau)), 
\end{align} 
Because $\hatnabla_{z\partial_z}$ is purely imaginary 
(\ref{eq:property_kappaH}), we have 
\[ 
\kappa \cQ = -\cQ  \kappa, \quad 
\ov{\cU}=\kappa \cU \kappa.
\] 
By $(\hatnabla_{z\partial_z} + E)\F_\tau \subset \F_\tau$
in Proposition \ref{prop:property_VHS}, we find 
\[
\cU = C_E, \quad \ov\cU = \tC_{\ov{E}}. 
\]
We have a canonical isomorphism 
\[
\pi^* K \cong \hK, \quad \text{where } \ 
\pi\colon \cM\times \Proj^1 \to \cM.  
\]
Let $C^{\infty h}(\pi^*K)$ be the sheaf of $C^\infty$ 
sections of $\pi^*K\cong \hK$ which are holomorphic 
on each fiber $\{\tau\}\times \Proj^1$. 
Under the isomorphism above, the flat connection 
$\hatnabla$ on $\cRz = \hK|_{\cM\times \C}$ 
can be written in the form:   
\begin{align}
\nonumber 
\hatnabla  \colon C^{\infty h}(\pi^*K) \longrightarrow 
C^{\infty h}(\pi^*K)\otimes
\Bigl( & z^{-1}\cA^{1,0}_{\cM} \oplus \cA^{1}_{\cM}
\oplus z \cA^{0,1}_{\cM}   \\ 
\nonumber 
 & \oplus (z^{-1}\cA^0_{\cM} \oplus \cA^0_{\cM} \oplus z\cA^0_{\cM})
\frac{dz}{z} \Bigr) \\
\label{eq:hatnabla_K} 
\hatnabla = z^{-1}C + D + z \tC  
 \,+\, &(z\partial_z -z^{-1}\cU + \cQ + z\ov{\cU})\otimes \frac{dz}{z}.   
\end{align} 
Under the same isomorphism, 
the pairing $(\cdot,\cdot)_{\cRz}$ on $\cRz= \hK|_{\cM\times \C}$ 
can be written as  
\begin{align*}
C^{\infty h}((-)^*(\pi^*K)) 
\otimes 
C^{\infty h}(\pi^*K) & \to C^{\infty h}(\cM\times \Proj^1) \\
s_1(\tau,-z) \otimes   s_2(\tau,z) & \longmapsto 
g(s_1(\tau,-z),s_2(\tau,z)).  
\end{align*} 
Unpacking the flatness of $\hatnabla$ and 
$\hatnabla$-flatness of the pairing in terms of 
$C,\tC,D,\cU,\cQ$ and $g$,  
we arrive at the equations in Proposition \ref{prop:CV-str}. 
Note that the pairing $h$  
in Definition \ref{def:polarized} gives a 
Hermitian metric on $K$ and is related to $g$ by  
\[
h(u_1,u_2) = g(\kappa(u_1),u_2).  
\]

\del{
\begin{remark}
(i) The $(0,1)$-part $\hatnabla_{\oi} = D_\oi + z \tC_\oi$ 
of the flat connection (\ref{eq:hatnabla_K}) 
gives the holomorphic structure on 
$\hK|_{\cM\times \{z\}}$ 
which corresponds to the holomorphic structure on $\cRz$. 
In particular, $D$ is identified with  
the canonical connection associated to the Hermitian metric $h$ 
on the holomorphic vector bundle $\cF/z\cF$. 
Similarly, the $(1,0)$-part $D_i + z^{-1} C_i$ 
gives an anti-holomorphic structure 
on $\hK|_{\cM\times \{z\}}$  
which corresponds to the anti-holomorphic structure on 
$\ov{\gamma^* \cRz}$. 
} 

\del{
(ii) Among the data of the Cecotti-Vafa structure, 
one can define the data ($C$, $D_E +\cQ$, $\cU$, $g$) 
without choosing a real structure. 
In fact, $C_X$ is given by the map 
$\cF/z\cF \ni [s] \mapsto  [z\nabla_X s] \in \cF/z\cF$, 
$D_E + \cQ$ is given by the map 
$\cF/z\cF \ni [s] \mapsto [\frac{1}{2}(\Grading - n) s] 
\in \cF/z\cF$, 
$\cU = C_E$, and $g$ is given by 
$g([s_1],[s_2]) = (s_1,s_2)_{\cF}|_{z=0}$ for $s_i\in \cF$.  
In the case of quantum cohomology, 
$C_i$ is the quantum multiplication 
$\phi_i\circ$ by some $\phi_i\in H^*_{\rm orb}(\cX)$
(see (\ref{eq:quantumproduct_divisor}), 
(\ref{eq:Dubrovinconn})) and 
$g$ is the Poincar\'{e} pairing. 
\end{remark} 
}

\del{
\begin{remark} 
A Frobenius manifold structure on $\cM$ 
arises from a miniversal \seminf VHS
(in the sense of \cite[Definition 2.8]{CIT:I}) 
without a real structure. 
To obtain a Frobenius manifold structure, 
we need a choice of an opposite subspace 
$\cH_-\subset \cH$: 
a sub free $\C\{z^{-1}\}$-module $\cH_-$ of $\cH$ 
satisfying  
\[
\cH = \F_\tau \oplus \cH_-, \quad 
\hatnabla_{z\partial_z}\cH_- \subset \cH_-.  
\] 
The choice of $\cH_-$ corresponds to giving 
a logarithmic extension of the flat bundle 
$(\cH,\hatnabla_{z\partial_z})$ at $z=\infty$. 
A graded \seminf VHS with the choice of an opposite subspace 
corresponds to the (trTLEP)-structure in 
Hertling \cite{hertling-tt*}. 
See \cite{barannikov-qpI,hertling-tt*, CIT:I} 
for the construction of Frobenius manifolds 
from this viewpoint.   
In the $tt^*$-geometry, 
the complex conjugate $\kappa_{\cH}(\F_\tau)$ 
of the Hodge structure $\F_\tau$ 
plays the role of the opposite subspace 
(see (\ref{eq:FpluskappaF})). 
When a miniversal \seminf VHS is equipped 
with both a real structure and an opposite subspace, 
under certain conditions, 
$\cM$ has a CDV (Cecotti-Dubrovin-Vafa) structure,  
which dominates both 
Frobenius manifold structure and 
Cecotti-Vafa structure on $T\cM$. 
See \cite[Theorem 5.15]{hertling-tt*} 
for more details. 
\end{remark} 
}


\section{Real and integral structures on the A-model}
\label{sec:A-model} 

In this section, we study real and integral structures 
in orbifold quantum cohomology. 
The quantum cohomology and Gromov-Witten theory for orbifolds 
have been developed by Chen-Ruan \cite{chen-ruan:GW} 
in the symplectic category 
and by Abramovich-Graber-Vistoli \cite{AGV} in the algebraic category. 
The definition of real and integral structures 
makes sense for both categories, but 
we will work in the algebraic category. 
In the proof of Theorem \ref{thm:pure_polarized}, 
we will need Lefschetz decomposition which may not hold in 
the symplectic category. 
Also, we only consider the even parity part of 
the cohomology group. 

\subsection{Orbifold quantum cohomology} 
Let $\cX$ be a proper smooth Deligne-Mumford stack over $\C$.  
Let $I\cX$ denote the \emph{inertia stack} of $\cX$, 
defined by the fiber product $\cX\times_{\cX\times \cX} \cX$ 
of the two diagonal morphisms $\Delta\colon \cX\to \cX\times \cX$. 
An object of $I\cX$ is given by a pair $(x,g)$ of a point 
$x\in \cX$ and $g\in \Aut(x)$. We call $g$ a \emph{stabilizer}  
of $(x,g)\in I\cX$. 
Let $\sfT$ be the index set of components of the $I\cX$. 
Let $0\in \sfT$ be the distinguished element 
corresponding to the trivial stabilizer. 
Set $\sfT'=\sfT\setminus\{0\}$. 
We have  
\[
I\cX = \bigsqcup_{v\in \sfT} \cX_v = 
\cX_0 \cup \bigsqcup_{v\in \sfT'} \cX_v, \quad 
\cX_0 = \cX. 
\] 
For each connected component $\cX_v$ of $I\cX$, 
we can associate a rational number 
$\iota_v$ called \emph{age}. 
The (even parity) orbifold cohomology group $H_{\rm orb}^*(\cX)$ 
is defined to be 
\[
H_{\rm orb}^k(\cX) = 
\bigoplus_{v\in \sfT,\ k-2\iota_v \equiv 0 (2)} 
H^{k-2\iota_v}(\cX_v,\C).  
\]
The degree $k$ of the orbifold cohomology 
can be a fractional number in general. 
Each factor $H^*(\cX_v,\C)$ in the right-hand side 
is same as the cohomology group of $\cX_v$ as a topological space.  
If not otherwise stated, we will use $\C$ 
as the coefficient of cohomology groups. 
We have an involution $\inv \colon I\cX \to I\cX$ defined by 
$\inv(x,g) = (x,g^{-1})$. 
This defines the \emph{orbifold Poincar\'{e} pairing}: 
\[
(\alpha, \beta)_{\rm orb} : = \int_{I\cX} \alpha \cup \inv^*(\beta).   
\]
This pairing is symmetric, non-degenerate over $\C$ 
and of degree $-2n$, where $n=\dim_\C\cX$. 
This involution also induces a map $\inv\colon \sfT\to \sfT$. 
Take a homogeneous $\C$-basis 
$\{\phi_i\}_{i=1}^N$ of $H_{\rm orb}^*(\cX)$.  
Let $\{\phi^i\}_{i=1}^N$ be the basis dual to $\{\phi_i\}$ with 
respect to the orbifold Poincar\'{e} pairing, \emph{i.e.} 
$(\phi_i,\phi^j)_{\rm orb}=\delta_i^j$. 

Now assume that the coarse moduli space of $\cX$ is projective.  
The genus zero Gromov-Witten invariants are integrals of the form: 
\begin{equation}
\label{eq:GWcorrelator}
\corr{\alpha_1\psi^{k_1},\dots,\alpha_l \psi^{k_l}}_{0,l,d}^\cX 
= \int_{[\cX_{0,l,d}]^{\rm vir}} \prod_{i=1}^l \ev_i^*(\alpha_i) \psi_i^{k_i}
\end{equation} 
where $\alpha_i \in H_{\rm orb}^*(\cX)$, $d\in H_2(\cX,\Q)$ 
and $k_i$ is a non-negative integer. 
The homology class $[\cX_{0,l,d}]^{\rm vir}$ is 
the virtual fundamental class 
of the moduli stack $\cX_{0,l,d}$ of 
genus zero, $l$-pointed stable maps to $\cX$ of degree $d$;  
$\ev_i\colon \cX_{0,l,d} \to I\cX$ is the evaluation map\footnote
{The map $\ev_i$ here is defined only 
as a map of topological spaces (not as a map of stacks). 
The evaluation map defined in \cite{AGV} 
is a map of stacks but takes values in the 
\emph{rigidified inertia stack}, 
which is the same as $I\cX$ as a topological space but 
is different as a stack. } 
at the $i$-th marked point; 
$\psi_i$ is the first Chern class 
of the line bundle over $\cX_{0,l,d}$ 
whose fiber at a stable map is the cotangent space 
of the coarse curve at the $i$-th marked point. 
We refer the readers to \cite{AGV} for a more precise definition. 
(Our notation is taken from \cite{CCLT:wp}; 
$\cX_{0,l,d}$ is denoted by $\cK_{0,l}(\cX,d)$ in \cite{AGV}.)  
The correlator (\ref{eq:GWcorrelator}) is 
non-zero only when $d$ belongs to 
$\Eff_\cX\subset H_2(\cX,\Q)$, the semigroup generated 
by effective stable maps, and $\sum_{i=1}^l (\deg \alpha_i + 2k_i) = 
2n + 2\pair{c_1(T\cX)}{d} + 2l -6$. 
The quantum product $\bullet_\tau$ with $\tau\in H_{\rm orb}^*(\cX)$ 
is defined by the formula 
\begin{equation*} 
\phi_i \bullet_\tau \phi_j = \sum_{d\in \Eff_\cX} \sum_{l\ge 0} 
\sum_{k=1}^N \frac{1}{l!} 
\corr{\phi_i,\phi_j,\tau,\dots,\tau, 
\phi_k}_{0,l+3,d}^\cX Q^d \phi^k,   
\end{equation*} 
where $Q^d$ is an element of the group ring $\C[\Eff_\cX]$ 
corresponding to $d\in \Eff_\cX$. 
We decompose $\tau = \tau_{0,2} + \tau'$ with 
$\tau_{0,2}\in H^2(\cX)$,  
$\tau' \in \bigoplus_{k\neq 1} H^{2k}(\cX) \oplus 
\bigoplus_{v\in \sfT'} H^*(\cX_v)$. 
Using the divisor equation \cite{tseng:QRR, AGV}, 
we have 
\begin{align} 
\label{eq:quantumproduct_divisor}
\phi_i \bullet_\tau \phi_j = \sum_{d\in \Eff_\cX} \sum_{l\ge 0} 
\sum_{k=1}^N 
\frac{1}{l!} 
\corr{\phi_i,\phi_j,\tau',\dots,\tau',
\phi_k}_{0,l+3,d}^{\cX} 
e^{\pair{\tau_{0,2}}{d}} Q^d \phi^k. 
\end{align} 
The quantum product is a priori a formal power series in 
$e^{\tau_{0,2}}Q$ and $\tau'$. 
When this is a convergent power series, we can put $Q=1$ 
and define
\[
\circ_\tau := \bullet_\tau|_{Q=1}.  
\]
Throughout the paper, \emph{we assume that $\circ_\tau$ is convergent 
on some domain $U\subset H_{\rm orb}^*(\cX)$}.  
The domain $U$ here contains the following limit direction:  
\begin{equation}
\label{eq:largeradiuslimit}
\Re\pair{\tau_{0,2}}{d} \to -\infty, \quad 
\forall d\in \Eff_\cX\setminus\{0\}, 
\quad \tau' \to 0.  
\end{equation} 
This is called the \emph{large radius limit}. 
The (big) orbifold quantum cohomology is a family of 
associative algebras $(H_{\rm orb}^*(\cX),\circ_\tau)$ parametrized 
by $\tau\in U$. In the large radius limit, $\circ_\tau$ goes to the 
orbifold cup product 
in the sense of Chen-Ruan \cite{chen-ruan:new_coh_orb}. 

\subsection{A-model \seminf VHS}
\label{subsec:AmodelVHS} 
Let $\{t^i\}_{i=1}^N$ be a co-ordinate system on 
$H_{\rm orb}^*(\cX)$ dual to $\{\phi_i\}$.  
The orbifold quantum product defines a graded \seminf VHS $\tcF$ on $U$ 
\begin{align*} 
\tcF:= H^*_{\rm orb}(\cX) \otimes \cO_U\{z\} 
\end{align*}
endowed with a flat \emph{Dubrovin connection} $\nabla$ and 
a pairing $(\cdot,\cdot)_{\tcF}$  
\begin{equation}
\label{eq:Dubrovinconn}
\nabla := d + \frac{1}{z}\sum_{i=1}^N 
(\phi_i\circ_\tau )dt^i, 
\quad (f,g)_{\tcF}:= (f(-z), g(z))_{\rm orb},    
\end{equation} 
and a grading operator $\Grading$ and Euler vector field $E$ 
\[
\Grading := 2 z\partial_z + 2 E + 2 (\mu + \frac{n}{2}), \quad 
E := \sum_{i=1}^N (1- \frac{1}{2}\deg \phi_i) t^i \parfrac{}{t^i}  
+ \sum_{i=1}^N r^i \parfrac{}{t^i}, 
\]
where $n=\dim_\C\cX$, $c_1(T\cX)= \sum_{i} r^i \phi_i\in H^2(\cX)$ 
and $\mu\in \End(H_{\rm orb}^*(\cX))$ is defined by 
\begin{equation}
\label{eq:def_mu} 
\mu(\phi_i): = \frac{1}{2} (\deg \phi_i - n) \phi_i. 
\end{equation} 
The \seminf VHS $\tcF$ is also called the \emph{quantum $D$-module}. 
The standard argument (as in \cite{cox-katz,manin}) 
and the WDVV equation in orbifold Gromov-Witten theory 
\cite{AGV} show that the Dubrovin connection is flat and that 
the above data satisfy the axioms of a graded \seminf VHS. 

Let $H^2(\cX,\Z)$ denote the cohomology of the constant sheaf $\Z$ 
on the topological \emph{stack} $\cX$ 
(not on the topological \emph{space}). 
This group is the set of isomorphism classes of 
topological orbifold line bundles on $\cX$. 
Let $L_\xi \to \cX$ be the orbifold line bundle corresponding to 
$\xi\in H^2(\cX,\Z)$. 
Let $0\le f_v(\xi)<1$ be the rational number such that 
the stabilizer of $\cX_v$ ($v\in \sfT$) acts on $L_\xi|_{\cX_v}$ by 
a complex number $\exp(2\pi \iu f_v(\xi))$. 
This number $f_v(\xi)$ is called the \emph{age}  
of $L_\xi$ along $\cX_v$. 
We define a map $G(\xi)\colon H^*_{\rm orb}(\cX)\to H^*_{\rm orb}(\cX)$ 
by 
\begin{align}
\label{eq:Galois} 
G(\xi)(\tau_0 \oplus \bigoplus_{v \in \sfT'} \tau_v) 
=(\tau_0- 2\pi \iu \xi_0) \oplus \bigoplus_{v\in \sfT'} 
e^{2\pi \iu f_v(\xi)} \tau_v,   
\end{align} 
where $\tau_v \in H^*( \cX_v)$ and $\xi_0$ is the image of 
$\xi$ in $H^2( \cX,\Q)$. 
Let $dG(\xi)\colon H^*_{\rm orb}(\cX) \to H^*_{\rm orb}(\cX)$ be 
the linear isomorphism given by the differential of $G(\xi)$. 
\begin{equation*}
\label{eq:dGalois}
dG(\xi)(\tau_0 \oplus \bigoplus_{v \in \sfT'} \tau_v) 
=\tau_0 \oplus \bigoplus_{v\in \sfT'} e^{2\pi \iu f_v(\xi)} \tau_v.  
\end{equation*}
\begin{proposition}
The map $\tcF \to G(\xi)^* \tcF$ given by 
\[ 
\tcF_\tau \ni s \longmapsto dG(\xi)s \in \tcF_{G(\xi)\tau} 
\]
is a homomorphism of graded \seminf VHS's. 
We call this \emph{Galois action} of $H^2(\cX,\Z)$ on $\tcF$.  
\end{proposition} 
\begin{proof} 
For $\alpha_1,\dots,\alpha_l\in H_{\rm orb}^*(\cX)$, 
we claim that 
\[
\langle \alpha_1,\dots,\alpha_l \rangle_{0,l,d} 
= e^{- 2 \pi \iu \pair{\xi_0}{d}}  
\langle dG(\xi)(\alpha_1),\dots, dG(\xi)(\alpha_l)\rangle_{0,l,d}.  
\]
If there exists an orbifold stable map $f\colon (C,x_1,\dots,x_l) \to \cX$
of degree $d$, we have an orbifold line bundle $f^*L_\xi$ on $C$ 
such that the monodromy at $x_k$ equals $\exp(2\pi \iu f_{v_k}(\xi))$ 
where $\ev_k(f)\in \cX_{v_k}$. Then we must have 
\[
\deg f^*L_\xi - \sum_{k=1}^l f_{v_k} \in \Z, \quad \emph{i.e. } 
e^{-2\pi \iu \pair{\xi_0}{d}} 
\prod_{i=1}^l e^{ 2 \pi \iu f_{v_i}(\xi)} =1.  
\]
The claim follows from this. The lemma follows from this claim 
and (\ref{eq:quantumproduct_divisor}). 
\end{proof} 

We can assume that $U$ is invariant under the Galois action. 
\begin{definition} 
The \emph{A-model \seminf VHS} of $\cX$ is defined to be 
the quotient $\cF$ of $\tcF \to U$ 
by the Galois action of $H^2(\cX,\Z)$ given above. 
\[
\cF := (\tcF \to U)/H^2(\cX,\Z) 
\]
The flat connection, the pairing and 
the grading operator on $\tcF$ 
induce those on $\cF$. 
\end{definition}

\subsection{The space of solutions to 
quantum differential equations}
As in Section \ref{subsec:realintstr_VHS_def}, 
the graded \seminf VHS $\tcF$ yields 
a flat connection $\hatnabla$ on the locally free sheaf 
$\cRz = H^*_{\rm orb}(\cX) \otimes \cO_{U\times \C^*}$. 
This is referred to as the \emph{first structure connection}. 
A $\hatnabla$-flat section $s$ of this trivial bundle 
satisfies the differential equations: 
\begin{align}
\label{eq:qde_1}
&\nabla_i s = \hatnabla_i s=  
\parfrac{s}{t^i}  + \frac{1}{z} \phi_i \circ_\tau s =0, 
\quad i=1,\dots, N, \\ 
\label{eq:qde_2} 
&\hatnabla_{z\partial_z} s= 
z\parfrac{s}{z} - \frac{1}{z} E\circ_\tau s + \mu s = 0. 
\end{align} 
We call these equations \emph{quantum differential equations}. 
We give a fundamental solution $L(\tau,z)$ to the differential 
equations (\ref{eq:qde_1})  
using gravitational descendants. 
Let $\pr\colon I\cX \to \cX$ be the natural projection.  
We define the action of a class $\tau_0 \in H^*(\cX)$ on 
$H_{\rm orb}^*(\cX)$ by 
\[
\tau_0 \cdot \alpha = \pr^*(\tau_0) \cup \alpha, \quad 
\alpha \in H_{\rm orb}^*(\cX),  
\]
where the right-hand side is the cup product on $I\cX$.
We define 
\begin{equation}
\label{eq:fundamentalsol_L}
L(\tau,z) \phi_i := e^{-\tau_{0,2}/z} \phi_i + 
\sum_{\substack{(d,l)\neq (0,0) \\ d\in \Eff_\cX}}
\sum_{k=1}^N  
\frac{\phi^k}{l!} \corr{\phi_k, \tau',\dots,\tau', 
\frac{e^{-\tau_{0,2}/z} \phi_i}{-z-\psi}}_{0,l+2,d}^\cX 
e^{\pair{\tau_{0,2}}{d}},  
\end{equation} 
where $\tau=\tau_{0,2} + \tau'$ with $\tau_{0,2}\in H^2(\cX)$ 
and $\tau' \in \bigoplus_{k\neq 1} H^{2k}(\cX) \oplus 
\bigoplus_{v\in \sfT'} H^*(\cX_v)$ and 
$1/(-z-\psi)$ in the correlator should be expanded in the series 
$\sum_{k=0}^\infty (-z)^{-k-1}\psi^k$.  
It is well-known to specialists that $L(\tau,z)\phi_i$ 
is a solution to (\ref{eq:qde_1}). 
\begin{proposition}
\label{prop:fundamentalsol_A} 
$L(\tau,z)$ satisfies the following differential equations: 
\begin{align*}
\nabla_k L(\tau,z) \phi_i &=0, \quad k=1,\dots,N, \\ 
\hatnabla_{z\partial_z} L(\tau,z) \phi_i &= 
L(\tau,z) (\mu \phi_i -\frac{\rho}{z}\phi_i),  
\end{align*}
where $\rho := c_1(T\cX)\in H^2(\cX)$. 
Moreover, $L(\tau,z)$ satisfies 
\begin{align*} 
(L(\tau,-z) \phi_i, L(\tau,z)\phi_j)_{\rm orb} & 
= (\phi_i,\phi_j)_{\rm orb}, \\ 
dG(\xi) L(G(\xi)^{-1} \tau, z) \alpha & 
= L(\tau,z) e^{-2\pi\iu \xi_0/z} e^{2\pi\iu f_v(\xi)} 
\alpha, \quad \alpha \in H^*(\cX_v), 
\end{align*}
where $(dG(\xi), G(\xi))$ is the Galois action
associated to $\xi\in H^2(\cX,\Z)$.    
(See Section \ref{subsec:AmodelVHS}.) 
\end{proposition} 
\begin{proof} 
The first equation follows from the topological recursion relation 
\cite[2.5.5]{tseng:QRR} in orbifold Gromov-Witten theory. 
The proof for the case of manifolds can be found in 
\cite[Proposition 2]{pandharipande}, \cite[Chapter 10]{cox-katz} 
and the proof for orbifolds is completely parallel. 
Note that we can decompose $L$ as 
$L(\tau,z) = S(\tau,z) \circ e^{-\tau_{0,2}/z}$
for some $\End(H_{\rm orb}^*(\cX))$-valued function $S(\tau,z)$. 
By the homogeneity of Gromov-Witten invariants, it is 
easy to check that $S$ preserves the degree, \emph{i.e.} 
$\Grading S(\tau,z) = S(\tau,z) \Grading$. 
Therefore, $\Grading L(\tau,z) = L(\tau,z) (\Grading - 2\rho/z)$. 
The second equation follows from this and the first equation. 
Put $\bs'_i =L(\tau,-z)\phi_i$ and $\bs_j = L(\tau,z)\phi_j$. 
By using the first equation and the Frobenius 
property $(\alpha\circ_\tau \beta, \gamma)_{\rm orb}=
(\alpha, \beta\circ_\tau \gamma)_{\rm orb}$, we have 
\[
\parfrac{}{t^k}(\bs'_i,\bs_j)_{\rm orb} = 
\frac{1}{z}(\phi_k\circ_\tau \bs'_i, \bs_j)_{\rm orb}
- \frac{1}{z}(\bs'_i, \phi_k\circ_\tau \bs_j)_{\rm orb} =0.  
\] 
Hence $(\bs'_i,\bs_j)_{\rm orb}$ is constant in $\tau$. 
When $\tau'=0$ and $\Re\pair{\tau_{0,2}}{d} \to -\infty$ for 
$d\in \Eff_{\cX}$, we have asymptotics 
$\bs'_i \sim e^{\tau_{0,2}/z} \phi_i$ and 
$\bs_j \sim e^{-\tau_{0,2}/z} \phi_j$. 
Therefore we have 
\[(\bs'_i,\bs_j)_{\rm orb} \sim 
(e^{-\tau_{0,2}/z} \phi_i, e^{\tau_{0,2}/z}\phi_j)_{\rm orb} = 
(\phi_i,\phi_j)_{\rm orb}
\]
and the third equation follows. 
The Galois action sends a $\nabla$-flat section 
$L(\tau,z)\alpha$ to another $\nabla$-flat section 
$dG(\xi) L(G(\xi)^{-1} \tau,z)\alpha$. 
Note that a $\nabla$-flat 
section $\bs=L(\tau,z)\phi$ is characterized by the 
asymptotic initial condition $\bs \sim e^{-\tau_{0,2}/z}\phi$ 
in the large radius limit. 
The fourth equation follows from this 
and the asymptotics 
$dG(\xi) L( G(\xi)^{-1} \tau, z) \alpha 
\sim e^{-\tau_{0,2}/z} e^{-2\pi\iu \xi_0/z} e^{2\pi \iu f_v(\xi)} \alpha$.  
\end{proof} 
Although the convergence of $L(\tau,z)$ is not a priori clear, 
we know from the differential equations above 
and the convergence assumption of $\circ_\tau$ 
that $L(\tau,z)$ is convergent on $(\tau,z)\in U\times \C^*$. 
Consider the vector space $\cH^\cX$ introduced by 
Coates-Givental \cite{coates-givental}: 
\[
\cH^\cX:=H_{\rm orb}^*(\cX)\otimes \C\{z,z^{-1}\}.  
\] 
Recall the space $\cH$ of $\nabla$-flat sections 
introduced in Section \ref{subsec:semi-inf_period}. 
For the A-model \seminf VHS, $\cH$ consists of 
cohomology-valued functions $s(\tau,z)$ satisfying (\ref{eq:qde_1}). 
We identify $\cH^\cX$ with $\cH$ 
by using the fundamental solution $L(\tau,z)$:   
\[
\cH^\cX\ni \alpha \longmapsto L(\tau,z)\alpha 
\in \Gamma(U\times \C^*,H^*_{\rm orb}(\cX)). 
\] 
Under this identification $\cH^\cX \cong \cH$, 
the embedding $\cJ_\tau\colon \tcF_\tau \to \cH^\cX$ 
(see (\ref{eq:embedding_J})) for the A-model \seminf VHS is given by 
\begin{align}
\label{eq:fundamentalsol_A}
\begin{split}
\cJ_\tau(\alpha) &= L(\tau,z)^{-1}\alpha =L(\tau,-z)^\dagger \alpha \\ 
&= e^{\tau_{0,2}/z}\left(\alpha + 
\sum_{\substack{(d,l)\neq (0,0) \\ d\in \Eff_\cX}} \sum_{i=1}^N 
\frac{1}{l!} 
\corr{\alpha,\tau',\dots,\tau', 
\frac{\phi_i }{z-\psi}}_{0,l+2,d}^\cX 
e^{\pair{\tau_{0,2}}{d}} \phi^i\right),   
\end{split}
\end{align} 
where $L(\tau,-z)^\dagger$ is the adjoint of 
$L(\tau,-z)$ with respect to $(\cdot,\cdot)_{\rm orb}$. 
The second line follows from (\ref{eq:fundamentalsol_L}) 
and an easy computation of the adjoint $L(\tau,-z)^\dagger$.  
The function $\cJ_\tau(1)$ is called the \emph{$J$-function}. 
The image $\F_\tau = \cJ_\tau(H^*_{\rm orb}(\cX)\otimes \C\{z\})$ 
of the embedding gives the moving subspace realization of 
the A-model \seminf VHS.   
By Proposition \ref{prop:fundamentalsol_A}, 
the action of $\hatnabla_{z\partial_z}$ on $\cH^\cX$ 
is given by 
\begin{equation}
\label{eq:conn_z-direction}
\hatnabla_{z\partial_z} = z\partial_z + \mu -\frac{\rho}{z},  
\end{equation} 
and the pairing $(\cdot,\cdot)_{\cH^\cX}$ and the symplectic 
form $\Omega$ on $\cH^\cX$ are given by 
\begin{equation}
\label{eq:pairing_H}
(\alpha,\beta)_{\cH^\cX} = (\alpha(-z),\beta(z))_{\rm orb}, \quad 
\Omega(\alpha,\beta) = \Res_{z=0} dz (\alpha,\beta)_{\cH^\cX}.  
\end{equation} 
The Galois action on $\tcF$ acts on 
$\nabla$-flat sections as $s(\tau,z) 
\mapsto dG(\xi) s(G(\xi)^{-1}\tau,z)$. 
By Proposition \ref{prop:fundamentalsol_A}, 
this induces a map $G^\cH(\xi)\colon \cH^\cX\to \cH^\cX$ 
given by 
\begin{equation}
\label{eq:GaloisH}
G^{\cH}(\xi)(\tau_0 \oplus \bigoplus_{v\in \sfT'} \tau_v) 
= e^{-2\pi \iu \xi_0/z}\tau_0 \oplus \bigoplus_{v\in \sfT'} 
e^{-2\pi \iu \xi_0/z} e^{2\pi \iu f_v(\xi)} \tau_v,  
\end{equation}
where we used the decomposition 
$\cH^{\cX} = \bigoplus_{v\in \sfT} 
H^*( {\cX_v})\otimes \C\{z,z^{-1}\}$. 

Next we consider the space $\cV$ of $\hatnabla$-flat sections 
introduced in Section \ref{subsec:semi-inf_period}. 
For the A-model \seminf VHS,  
this consists of sections $s$ satisfying both (\ref{eq:qde_1}) 
and (\ref{eq:qde_2}). 
Then $\cV$ is identified with 
the space of flat sections of the following 
flat vector bundle $(\sfH^\cX, \hatnabla)$: 
\[
\sfH^\cX := H^*_{\rm orb}(\cX) \times \C^* \to \C^*, 
\quad 
\hatnabla_{z\partial_z} = z \partial_z + \mu - \frac{\rho}{z}. 
\]
Furthermore, we identify the space of flat sections 
of $(\sfH^\cX,\hatnabla_{z\partial_z})$ with the orbifold cohomology 
$\cV^\cX := H_{\rm orb}^*(\cX)$ via the (well-known) 
fundamental solution of $\hatnabla_{z\partial_z} s=0$: 
\begin{equation} 
\label{eq:solutionmap_z} 
z^{-\mu}z^\rho \colon \cV^\cX \to \Gamma(\widetilde{\C^*}, \sfH^\cX), 
\quad \phi 
\mapsto s(z) = e^{-\mu\log z} e^{\rho\log z} \phi. 
\end{equation} 
Then we can identify $\cV$ 
with the orbifold cohomology group $\cV^\cX = H^*_{\rm orb}(\cX)$.  
The pairing $(\cdot,\cdot)_{\cV^\cX}$ on 
$\cV^\cX=H_{\rm orb}^*(\cX)$ induced from 
that on $\cV$ (see Equation (\ref{eq:pairing_def_V})) 
is given by 
\begin{align}
\label{eq:pairing_Amodel_V}
(\alpha,\beta)_{\cV^\cX}  
= (e^{\pi \iu \rho}\alpha, e^{\pi \iu \mu} \beta)_{\rm orb}.  
\end{align}
The induced Galois action on $\cV^\cX$ is given by 
\begin{equation}
\label{eq:Galois_V}
G^{\cV}(\xi)(\tau_0 \oplus \bigoplus_{v\in \sfT'} \tau_v)  
= e^{-2\pi \iu \xi_0}\tau_0 \oplus \bigoplus_{v\in \sfT'} 
e^{-2\pi \iu \xi_0} e^{2\pi \iu f_v(\xi)} \tau_v. 
\end{equation}

\begin{remark}
The Galois actions on $\cH^\cX$, $\cV^\cX$ 
are the monodromy transformations of $\nabla$ on $U/H^2(\cX,\Z)$. 
The monodromy transformation of 
$\hatnabla_{z\partial_z}$ on $\C^*$ 
is given by 
\begin{equation}
\label{eq:z-monodromy_V}
e^{-2\pi \iu \mu} e^{2\pi \iu \rho} 
\colon \cV^\cX \longrightarrow \cV^\cX. 
\end{equation} 
This coincides with the Galois action $(-1)^n G^\cV([K_\cX])$. 
Here, $[K_\cX]$ is the class of 
the canonical line bundle.
When $\cX$ is Calabi-Yau, 
\emph{i.e.} $K_\cX$ is trivial,  
the pairing $(\cdot,\cdot)_{\cV^\cX}$ 
is either symmetric or anti-symmetric depending on
whether $n$ is even or odd.   
In general, this pairing is neither symmetric nor anti-symmetric.  
\end{remark} 

\begin{proposition}
\label{prop:char_A_real_int_str}
A real (integral) structure on the A-model \seminf VHS $\cF$ 
is given by a real subspace $\cV_\R^{\cX}$ (resp. integral lattice 
$\cV_\Z^\cX$) 
of $\cV^\cX = H^*_{\rm orb}(\cX)$ satisfying 
\begin{itemize}
\item[(i)] $\cV^{\cX} = \cV_\R^{\cX} \otimes_\R \C$ 
(resp. $\cV^{\cX} = \cV^\cX_\Z \otimes_\Z \C$);  
\item[(ii)] $\cV_\R^\cX$ (resp. $\cV_\Z^\cX$) is invariant 
under the Galois action (\ref{eq:Galois_V}); 
\item[(iii)] The pairing (\ref{eq:pairing_Amodel_V}) 
restricted on $\cV^\cX_\R$ (resp. $\cV_\Z^\cX$) takes values in $\R$
(resp. takes values in $\Z$ and is unimodular).    
\end{itemize} 
\end{proposition} 

Let $\kappa_\cH$ and $\kappa_\cV$ be the involution 
of $\cH^\cX$ and $\cV^\cX$ respectively. 
We discuss basic properties of $\kappa_\cH$ and $\kappa_\cV$ below. 
We decompose the Galois action on $\cH^\cX$ as 
\[
G^{\cH}(\xi) = e^{-2\pi \iu \xi_0/z} G_0^{\cH}(\xi), \quad 
G_0^{\cH}(\xi) := \bigoplus_{v\in \sfT} e^{2\pi \iu f_v(\xi)}. 
\]
\begin{proposition}
\label{prop:basic_realstr} 
For any real structure on the A-model \seminf VHS, the following holds. 
\begin{gather}
\label{eq:H2_imaginary}
\kappa_\cH (\tau_{0,2}/z) + (\tau_{0,2}/z) \kappa_{\cH} =0, \quad 
\kappa_\cV \tau_{0,2} + \tau_{0,2} \kappa_\cV =0, \\  
\label{eq:G_0_real} 
G_0^{\cH}(\xi) \kappa_\cH = \kappa_\cH G_0^\cH(\xi), \\ 
(z\partial_z + \mu) \kappa_\cH + 
\label{eq:grading_imaginary_H} 
\kappa_{\cH} (z\partial_z +\mu) =0, \\
\label{eq:relation_kappa_H_V} 
\kappa_{\cH} = z^{-\mu} \kappa_{\cV} z^\mu, \quad \text{when } z\in S^1,
\end{gather} 
where $\tau_{0,2}\in H^2(\cX,\R)$. 
Moreover, if $\cX$ satisfies the following condition: 
\begin{equation}
\label{eq:separation}
f_v(\xi) = f_{v'}(\xi), \ \forall \xi\in H^2(\cX,\Z) 
\Longrightarrow v=v',  
\end{equation} 
then we have 
\begin{equation}
\label{eq:kappa_induces_Inv}
\begin{split}
\kappa_\cH(H^*(\cX_v)\otimes \C\{z,z^{-1}\}) 
&= H^*(\cX_{\inv(v)})\otimes \C\{z,z^{-1}\}, \\ 
\kappa_\cV(H^*(\cX_v)) & = H^*(\cX_{\inv(v)}).   
\end{split} 
\end{equation}
When (\ref{eq:kappa_induces_Inv}) holds, $\kappa_\cV$ satisfies
\begin{equation}
\label{eq:kappaV_weightfiltr}
\kappa_\cV(\alpha) \in  \cC(\alpha) + H^{>2k}(\cX_{\inv(v)}), \quad 
\alpha\in H^{2k}(\cX_v)
\end{equation} 
for some complex antilinear isomorphism 
$\cC:H^{2k}(\cX_v) \to H^{2k}(\cX_{\inv(v)})$. 
\end{proposition} 
\begin{proof} 
Because $G_0^{\cH}(\xi)$ is nilpotent and commutes with 
$e^{-2\pi\iu\xi_0/z}$, 
$(G^\cH(\xi))^m= e^{-2\pi\iu m\xi_0/z}$ 
for $m>0$ satisfying $f_v(\xi) m\in \Z$ for all $v\in \sfT$. 
Hence $\tau_{0,2}/z$ is purely imaginary on $\cH^\cX$ for 
any $\tau_{0,2}\in H^2(\cX,\R)$.  
From (\ref{eq:solutionmap_z}), we know that 
the multiplication by $\tau_{0,2}$ is purely imaginary on $\cV^\cX$. 
Thus we have (\ref{eq:H2_imaginary}). 
From $G_0^{\cH}(\xi) = e^{2\pi\iu \xi_0/z}G^{\cH}(\xi)$, 
we have (\ref{eq:G_0_real}).  
Because $\hatnabla_{z\partial_z}=z\partial_z + \mu-\rho/z$ 
is purely imaginary on $\cH^\cX$ and so is $\rho/z$, 
we have (\ref{eq:grading_imaginary_H}).  
By (\ref{eq:solutionmap_z}), 
$\kappa_\cH$ and $\kappa_\cV$ are related by 
\[
\kappa_\cH = z^{-\mu} z^\rho \kappa_{\cV} z^{-\rho} z^{\mu}, 
\quad \text{for }z \in S^1.   
\]
Since $z^\rho = \exp(\rho \log z)$ 
is real on $\cV^\cX$ when $z\in S^1$, 
we have (\ref{eq:relation_kappa_H_V}). 
Under the condition (\ref{eq:separation}), 
the decomposition $\cH^\cX = \bigoplus_{v\in \sfT} 
H^*(\cX_v)\otimes \C\{z,z^{-1}\}$ 
is the simultaneous eigenspace decomposition for $G_0^\cH(\xi)$, 
$\xi\in H_2(\cX,\Z)$.  
Therefore, (\ref{eq:kappa_induces_Inv}) follows 
from $\ov{e^{2\pi\iu f_v(\xi)}} = e^{2\pi\iu f_{\inv(v)}(\xi)}$ 
and the reality of $G^\cH_0(\xi)$. 
Let $\omega$ be a K\"{a}hler class on $\cX$. 
The action of $\omega$ on $H^*(\cX_v)$ is nilpotent. 
In general, a nilpotent operator $\omega$ on a vector space  
defines an increasing filtration $\{W_k\}_{k\in \Z}$ on it, 
called a {\it weight filtration}, 
which is uniquely determined by the conditions: 
\[
\omega W_k \subset W_{k-2}, \quad 
\omega^k \colon \Gr^W_{k} \cong \Gr^W_{-k} 
\]
where $\Gr^W_k = W_k/W_{k-1}$. 
By the Lefschetz decomposition, we know that 
$W_k = H^{\ge n_v-k}(\cX_v)$ in this case 
($n_v:=\dim_\C\cX_v$).  
Since $\kappa_\cV$ anti-commutes with $\omega$ 
by (\ref{eq:H2_imaginary}), $\kappa_\cV$ preserves this filtration. 
This shows (\ref{eq:kappaV_weightfiltr}). 
Here, $\cC$ is the isomorphism on the associated graded quotient 
induced from $\kappa_\cV$. 
\end{proof}



\subsection{Purity and polarization}
\label{subsec:pure_polarized} 
For an arbitrary real structure, 
we study a behavior of the A-model \seminf VHS $\cF$ near 
the large radius limit point (\ref{eq:largeradiuslimit}). 
We show that it is pure and polarized 
(in the sense of Definitions 
\ref{def:pure}, \ref{def:polarized}) 
under suitable conditions. 
Recall that when $\tcF\to U$ is pure,  
this defines a Cecotti-Vafa structure 
on the vector bundle $K\to U$ by Proposition \ref{prop:CV-str}. 

\begin{theorem}
\label{thm:pure_polarized} 
Let $\cX$ be a smooth Deligne-Mumford stack with 
a projective coarse moduli space. 
Let $\cF$ be the A-model \seminf VHS of $\cX$ and 
take a real structure on $\cF$. 
Let $\omega$ be a K\"{a}hler class on $\cX$. 

{\rm (i)} 
If the real structure satisfies (\ref{eq:kappa_induces_Inv}), 
$\cF$ is pure at $\tau=-x\omega$ for sufficiently big 
$\Re(x)>0$.

{\rm (ii)}  
If moreover the real structure satisfies 
(c.f. (\ref{eq:kappaV_weightfiltr}))
\begin{align}
\label{eq:leadingterm_kappaV}
\begin{split}
&\kappa_\cV(\alpha) \in (-1)^k \R_{>0} \inv^*(\ov{\alpha}) 
+ H^{>2k}(\cX_{\inv(v)}), \\ 
\text{or equivalently} \quad &\kappa_\cH(\alpha) 
= (-1)^k\R_{>0} \inv^*(\ov{\alpha}) z^{-2k+n_v} 
+ O(z^{-2k+n_v-1})    
\end{split} 
\end{align} 
for $\alpha\in H^{2k}(\cX_v) \subset H_{\rm orb}^*(\cX)$, 
$n_v=\dim_\C\cX_v$, 
then the Hermitian metric 
$h(\cdot,\cdot)=g(\kappa(\cdot),\cdot)$ 
on the vector bundle $K\to U$ satisfies 
\[
(-1)^{\frac{p-q}{2}} h(u,u) >0, 
\quad u\in H^{p,q}(\cX_v) \subset K_{-x\omega}, \quad u\neq 0 
\]
for sufficiently big $\Re(x)>0$, 
where we identify $K_\tau$ with 
$\tcF_\tau/z\tcF_\tau \cong H^*_{\rm orb}(\cX)$.  
In particular, if $H^*_{\rm orb}(\cX)$ 
consists only of the $(p,p)$ part, i.e. 
$H^{2p}(\cX_v) = H^{p,p}(\cX_v)$ for all $v\in \sfT$ and $p\ge 0$, 
then $\cF$ is polarized at $\tau=-x\omega$ 
for sufficiently big $\Re(x)>0$. 
\end{theorem} 

\begin{remark}
\label{rem:after_pure_polarized} 
(i) The condition (\ref{eq:kappa_induces_Inv}) is satisfied 
when $\cX$ has enough line bundles to separate the inertia components 
(see (\ref{eq:separation}) in Proposition \ref{prop:basic_realstr}). 
In particular, (\ref{eq:kappa_induces_Inv}) is always satisfied when 
$\cX$ is a manifold.  

(ii) 
We can consider the \emph{algebraic A-model \seminf VHS}.  
Let $A^*(\cX)_\C$ denote the Chow ring of $\cX$ over $\C$. 
We set $\HH^*(\cX_v) := \Image(A^*(\cX_v)_\C \to H^*(\cX_v))$ and 
define $\HH_{\rm orb}^*(\cX) := \bigoplus_{v\in \sfT} \HH^*(\cX_v)$. 
The algebraic A-model \seminf VHS is defined to be 
\[
\HH^*_{\rm orb}(\cX) \otimes 
\cO_{U\cap \HH_{\rm orb}(\cX)}\{z\} 
\]
with the restriction of Dubrovin connection, the grading operator 
and pairing,   
modulo the Galois action given by an element of $\Pic(\cX)$. 
Here we used that the quantum product 
among classes in $\HH_{\rm orb}^*(\cX)$ 
again belongs to $\HH_{\rm orb}^*(\cX)$; 
this follows from the algebraic construction 
of orbifold Gromov-Witten theory \cite{AGV}.  
When we assume Hodge conjecture for all $\cX_v$, 
each $\HH^*(\cX_v)$ has Poincar\'{e} duality and
the orbifold Poincar\'{e} pairing is non-degenerate 
on $\HH_{\rm orb}^*(\cX)$. 
In this case, the algebraic A-model \seminf VHS 
is pure and polarized at $\tau=-x\omega$ for  
a K\"{a}hler class $\omega\in \HH^2(\cX)$ and $\Re(x)\gg 0$ 
if the conditions corresponding to (\ref{eq:kappa_induces_Inv}) and 
(\ref{eq:leadingterm_kappaV}) are satisfied. 
The proof below applies to the algebraic A-model \seminf VHS 
without change. 
Note that the Poincar\'{e} duality of $\HH^*(\cX_v)$ also 
implies the Hard Lefschetz of it
used in the proof below.  
\end{remark} 

\begin{remark} 
Hertling \cite{hertling-tt*} 
and Hertling-Sevenheck \cite{hertling-sevenheck} 
studied similar problems for general TERP structures. 
They considered the change of TERP structures induced 
by the rescaling $z \mapsto r z$ of the parameter $z$. 
This rescaling with $r \to \infty$ 
is called Sabbah orbit in \cite{hertling-sevenheck} 
and is equivalent to the flow of minus the Euler vector field: 
$\tau \mapsto \tau - \rho \log r$ for $\tau \in H^2(\cX)$. 
When $\cX$ is Fano and $\omega=c_1(\cX)=\rho$, 
the large radius limit corresponds to 
the Sabbah orbit\footnote{
The author thanks 
Claus Hertling for this remark.}, 
and the conclusions 
in Theorem \ref{thm:pure_polarized} can 
be deduced from a general theorem 
\cite[Theorem 7.3]{hertling-sevenheck} in this case. 
\end{remark} 

\begin{remark}
\label{rem:sabbah} 
Singularity theory gives a \seminf VHS with a real structure. 
According to the recent work of Sabbah \cite[Section 4]{sabbah}, 
the \seminf VHS arising from a cohomologically tame function 
on an affine manifold is pure and polarized.  
This result covers 
the case of Landau-Ginzburg model mirror to toric orbifolds 
treated in Section \ref{sec:intstr_via_toricmirrors}. 
\end{remark} 

The rest of this section is 
devoted to the proof of Theorem \ref{thm:pure_polarized}. 

From Equation (\ref{eq:fundamentalsol_A}), 
we see that 
$e^{x\omega/z} \cJ_{-x\omega}(\varphi) \to  \varphi$ 
as $\Re(x)\to \infty$. 
Thus, in the moving subspace realization, 
the Hodge structure 
$\F_{-x\omega} = \cJ_{-x\omega}(\cF_{-x\omega})$ 
has the asymptotics: 
\[
\F_{-x \omega} \sim e^{-x\omega/z}\F_{\rm lim} \quad 
\text{as $\Re(x)\to \infty$}, 
\]
where $\Flim:= H^*_{\rm orb}(\cX)\otimes \C\{z\}$ 
is the limiting Hodge structure. 
This is an analogue of the {\it nilpotent orbit theorem} 
\cite{schmid}  
in quantum cohomology. 
First we study the behavior of the nilpotent orbit 
$x\mapsto e^{-x\omega/z} \Flim$ for $\Re(x)\gg 0$ 
(see Proposition \ref{prop:nilpotentorb_pure} below). 

We study the purity of the \seminf VHS 
$x\mapsto e^{-x\omega/z} \Flim$, 
\emph{i.e.} if the natural map 
\begin{equation}
\label{eq:nilpotentorb_pure}
e^{-x\omega/z} \Flim \cap \kappa_{\cH}(e^{-x\omega/z}\Flim) 
\longrightarrow e^{-x\omega/z} (\Flim/z\Flim) 
\cong e^{-x\omega/z}H^*_{\rm orb}(\cX)  
\end{equation} 
is an isomorphism (see (\ref{eq:FcapkappaF}) 
in Proposition \ref{prop:purity_condition}). 
Under the condition (\ref{eq:kappa_induces_Inv}), 
this is equivalent to that the map   
\[
e^{-x\omega/z} H^*( \cX_v)\{z\} \cap 
\kappa_\cH (e^{-x\omega/z} H^*( \cX_{\inv(v)} )\{z\}) 
\to   e^{-x\omega/z} H^*( \cX_v)   
\]
is an isomorphism for each $v\in \sfT$, where 
we put $H^*(\cX_v)\{z\} = H^*(\cX_v)\otimes \C\{z\}$.  
Since $\kappa_\cH e^{-x\omega/z} = e^{\ov{x} \omega/z} \kappa_\cH$ 
(see (\ref{eq:H2_imaginary})), this is equivalent to that 
\[
H^*( \cX_v)\{z\} \cap 
e^{2t\omega/z} 
\kappa_\cH(H^*( \cX_{\inv(v)})\{z\}) \to H^*( \cX_{v}), 
\quad t:=\Re(x)
\]
is an isomorphism. 
We further decompose this into $(z\partial_z + \mu)$-eigenspaces. 
Because $z\partial_z +\mu$ 
is purely imaginary (\ref{eq:grading_imaginary_H}), 
the above map between the $(z\partial_z+\mu)$-eigenspaces  
of the eigenvalue 
$\frac{1}{2}(-k+\age(v)-\age(\inv(v)) )$ is of the form: 
\[
\left
(\bigoplus_{l\ge 0} H^{n_v-k-2l}( \cX_v) z^l
\right) 
\cap 
e^{2t\omega/z}\kappa_{\cH}
\left(\bigoplus_{l\ge 0} H^{n_v+k-2l}( \cX_{\inv(v)}) z^l
\right) 
\to  H^{n_v-k}(\cX_v).  
\]
Here, $n_v=\dim_\C \cX_v$ and $k$ is an integer such that 
$n_v-k$ is even. 
By using (\ref{eq:relation_kappa_H_V}), 
we find that this map is conjugate 
(via $z^{\mu+(k-\iota_v+\iota_{\inv(v)})/2}$) 
to the following map: 
\begin{equation}
\label{eq:kappaV_transversality}
H^{\le n_v-k}(\cX_v) \cap e^{2 t\omega} 
\kappa_\cV(H^{\le n_v+k}(\cX_{\inv(v)})) \to H^{n_v-k}(\cX_v)   
\end{equation} 
which is induced by 
$H^{\le n_v-k}(\cX_v) \to 
H^{\le n_v-k}(\cX_v)/H^{\le n_v-k-2}(\cX_v) 
\cong H^{n_v-k}(\cX_v)$. 
We will show that this becomes an isomorphism for 
$t=\Re(x)\gg 0$ 
in Lemma \ref{lem:nilpotentorb_pure} below. 
 
Let $\fra \colon H^*(\cX_v) \to H^{*+2}(\cX_v)$ 
be the operator defined by $\fra(\phi):=\omega\cup \phi$. 
There exists an operator 
$\fra^\dagger \colon H^*(\cX_v) \to H^{*-2}(\cX_v)$ 
such that $\fra$ and $\fra^\dagger$ generate
the Lefschetz $\mathfrak{sl}_2$-action on $H^*(\cX_v)$:  
\[
[\fra, \fra^\dagger] = h, \quad 
[h,\fra] = 2\fra, \quad [h,\fra^\dagger]
=-2\fra^\dagger,    
\]
where $h: = \deg - n_v$ is the (shifted) grading operator. 
Note that $\fra^\dagger$ is uniquely determined by the 
above commutation relation and that $\fra^\dagger$ annihilates 
the primitive cohomology 
$PH^{n_v-k}(\cX_v):= \Ker(\fra^{k+1} \colon H^{n_v-k}(\cX_v) 
\to H^{n_v+k+2}(\cX_v))$. 

\begin{lemma}
\label{lem:expa_expadagger}
The map $e^{-\fra} e^{\fra^\dagger} \colon H^*(\cX_v) \to H^*(\cX_v)$ 
sends $H^{\ge n_v-k}(\cX_v)$ onto $H^{\le n_v+k}(\cX_v)$  
isomorphically. Moreover, for 
$u \in \fra^{j} PH^{n_v-k-2j}(\cX_v)\subset H^{n_v-k}(\cX_v)$, 
one has 
\[
e^{-\fra} e^{\fra^\dagger} u = (-1)^{k+j} \frac{j!}{(k+j)!}
\omega^k u + H^{< n_v+k}(\cX_v). 
\]
\end{lemma} 
\begin{proof}
An easy calculation shows that 
\[
e^{-\fra} e^{\fra^\dagger} \fra= 
- \fra^\dagger e^{-\fra} e^{\fra^\dagger}.  
\]
Therefore, $e^{-\fra}e^{\fra^\dagger}$ should send 
the weight filtration for the nilpotent operator 
$\fra$ to that for $\fra^\dagger$.  
But the weight filtration for $\fra$ is 
$\{H^{\ge n_v-k}\}_k$ 
and that for $\fra^\dagger$ is $\{H^{\le n_v+k}\}_k$ 
(see the proof of Proposition \ref{prop:basic_realstr} 
for weight filtration).  
Take $u \in \fra^j PH^{n_v-k-2j}(\cX_v)$. 
Put $u = \fra^j \phi$ for $\phi \in PH^{n_v-k-2j}(\cX_v)$. 
We calculate 
\begin{align*}
e^{-\fra}e^{\fra^\dagger} u  &= e^{-\fra}e^{\fra^\dagger} \fra^j \phi 
= (-\fra^\dagger)^j e^{-\fra}e^{\fra^\dagger} \phi 
=  (-\fra^\dagger)^j e^{-\fra} \phi \\
&= (-\fra^\dagger)^j \frac{(-1)^{k+2j}}{(k+2j)!} \fra^{k+2j} \phi  
+ \text{ lower degree term},  
\end{align*} 
where in the second line we used that 
$e^{-\fra}e^{\fra^\dagger}u \in H^{\le n_v+k}(\cX_v)$. 
Using $\fra^\dagger \fra^l u = l(k+2j+1-l) \fra^{l-1} u$, 
we arrive at the formula for $e^{-\fra}e^{\fra^\dagger}u$.  
\end{proof}

\begin{lemma}
\label{lem:nilpotentorb_pure} 
The map (\ref{eq:kappaV_transversality}) is an isomorphism 
for sufficiently big $t>0$. 
Moreover, $u\in H^{n_v-k}(\cX_v)$ corresponds to 
an element of the form 
\[
(2t)^{(\deg +k-n_v)/2}(e^{\fra^\dagger} u + O(t^{-1}))  
\in H^{\le n_v-k}(\cX_v)  \cap e^{2\omega t}
\kappa_\cV(H^{\le n_v+k}(\cX_{\inv(v)})) 
\] 
under (\ref{eq:kappaV_transversality}), where $(2t)^{\deg/2}$ 
is defined by $(2t)^{\deg/2} = (2t)^k$ on $H^{2k}(\cX_v)$. 
\end{lemma} 
\begin{proof}
First we rescale (\ref{eq:kappaV_transversality}) by 
$(2t)^{-\deg/2}$:  
\[
\begin{CD}
H^{\le n_v-k}(\cX_v) \cap e^{2\omega t}
\kappa_\cV(H^{\le n_v+k}(\cX_{\inv(v)}))  @>>> H^{n_v-k}(\cX_v)  \\ 
@VV{(2t)^{-\deg/2}}V @VV{(2t)^{-\deg/2}}V \\ 
H^{\le n_v-k}(\cX_v) \cap e^{\omega} 
\kappa_t (H^{\le n_v+k}(\cX_{\inv(v)})) @>>> H^{n_v-k}(\cX_v), 
\end{CD} 
\]
where $\kappa_t := (2t)^{-\deg/2} \kappa_\cV (2t)^{\deg/2}$. 
Since the column arrows are isomorphisms for all $t\in \R$, 
it suffices to show that the bottom arrow is an isomorphism 
for $t\gg 0$.  
Observe that the expected dimension of 
$H^{\le n_v-k}(\cX_v) \cap e^{2\omega} 
\kappa_t (H^{\le n_v+k}(\cX_{\inv(v)}))$ equals 
$\dim H^{n_v-k}(\cX_v)$ by Poincar\'{e} duality. 
Thus that the bottom arrow becomes an isomorphism 
is an open condition for $\kappa_t$. 
By (\ref{eq:kappaV_weightfiltr}) 
in Proposition \ref{prop:basic_realstr}, 
we have 
\begin{equation}
\label{eq:kappat_goestoC}
\kappa_t = \cC + O(t^{-1}), 
\end{equation} 
for a degree preserving isomorphism 
$\cC\colon H^*(\cX_{\inv(v)})\cong H^*(\cX_v)$. 
Therefore, we only need to check that the map at $t=\infty$ 
\begin{equation}
\label{eq:transversality_tinfinity}
H^{\le n_v-k}(\cX_v) \cap e^{\fra} H^{\le n_v+k}(\cX_v) 
\to H^{n_v-k}(\cX_v) 
\end{equation}
is an isomorphism (recall that $\fra=\omega\cup$). 
Note that this factors through $\exp(-\fra^\dagger)$ as  
\[
\begin{CD}
H^{\le n_v-k} \cap e^{\fra} H^{\le n_v+k}
@>{\exp(-\fra^\dagger)}>> 
H^{\le n_v-k}\cap e^{-\fra^\dagger}e^{\fra}  
H^{\le n_v+k} 
@>>> H^{n_v-k}, 
\end{CD} 
\]
where we omitted the space $\cX_v$ from the notation.  
The second map is induced from the 
projection $H^{\le n_v-k} \to H^{n_v-k}$ 
again. 
Because $e^{-\fra^\dagger}e^{\fra}(H^{\le n_v+k})=H^{\ge n_v-k}$ 
by Lemma \ref{lem:expa_expadagger}, 
we know that the map 
(\ref{eq:transversality_tinfinity}) is an isomorphism 
and that the inverse map is given by 
$u\mapsto \exp(\fra^\dagger)u$. 
Now the conclusion follows. 
\end{proof} 


\begin{proposition}
\label{prop:nilpotentorb_pure} 
Assume that (\ref{eq:kappa_induces_Inv}) holds.  
Then the nilpotent orbit 
$x\mapsto e^{-x\omega/z} \Flim$ is pure 
for sufficiently big $t=\Re(x)>0$
i.e. the map (\ref{eq:nilpotentorb_pure}) is an isomorphism 
for $t\gg 0$. 
The inverse image of 
$e^{-x\omega/z} u$, $u\in H^{n_v-k}(\cX_v)$
under (\ref{eq:nilpotentorb_pure}) is of the form 
$e^{-x\omega/z} \varpi_t(u)$ with 
\begin{equation}
\label{eq:lift_to_FcapkappaF} 
\varpi_t(u) = z^{-\mu-(k-\iota_v+ \iota_{\inv(v)})/2} 
(2t)^{(\deg + k-n_v)/2} 
(e^{\fra^\dagger} u+ O (t^{-1})) 
\in \bigoplus_{l\ge 0} H^{n_v-k-2l}(\cX_v)z^l. 
\end{equation} 
When $u = \fra^j \phi$ and $\phi \in PH^{n_v-k-2j}(\cX_v)$, 
we have  
\begin{equation}
\label{eq:nilpotentorb_Hermitianmetric}
(\kappa_\cH(e^{-x\omega/z}\varpi_t(u)),
e^{-x\omega/z} \varpi_t(u))_{\cH^\cX} 
= \frac{(2t)^{k} j!}{(k+j)!}\int_{\cX_v} \omega^{k+2j}
\phi \cup \inv^*\cC(\phi)   + O(t^{k-1})   
\end{equation}
where $\cC\colon H^*(\cX_v) \to H^*(\cX_{\inv(v)})$ is 
the isomorphism appearing in (\ref{eq:kappaV_weightfiltr}) 
and $(\cdot,\cdot)_{\cH^\cX}$ is given in (\ref{eq:pairing_H}).  
If moreover $u\in H^{p,q}(\cX_v)\setminus \{0\}$ 
and the condition (\ref{eq:leadingterm_kappaV}) holds, 
\[
(-1)^{(p-q)/2}(\kappa_\cH(e^{-x\omega/z}\varpi_t(u)),
e^{-x\omega/z} \varpi_t(u))_{\cH^\cX} >0 
\]
for $t=\Re(x)\gg 0$.
(Here $p+q = n_v-k$.)  
\end{proposition} 
\begin{proof} 
The purity of $e^{-x\omega/z} \Flim$ 
and the formula for $\varpi_t(u)$ 
follow from Lemma \ref{lem:nilpotentorb_pure} and 
the discussion preceding (\ref{eq:kappaV_transversality}). 
Putting $c= (-k +\iota_v-\iota_{\inv(v)})/2$, we calculate 
\begin{align*}
&(\kappa_\cH(e^{-x\omega/z}\varpi_t(u)),
e^{-x\omega/z} \varpi_t(u))_{\cH^\cX} 
 = (\kappa_\cH(\varpi_t(u)), e^{-2t\omega/z} \varpi_t(u))_{\cH^\cX} \\
& = (2t)^{k-n_v} 
(z^{-\mu-c}\kappa_\cV (2t)^{\deg/2}
(e^{\fra^\dagger} u + O(t^{-1})), 
z^{-\mu+c} 
(2t)^{\deg/2}(e^{-\fra}e^{\fra^\dagger} u +O(t^{-1})))_{\cH^\cX} \\   
& = (2t)^{k} ( (-1)^{-\mu-c} \kappa_t 
( e^{\fra^\dagger} u + O(t^{-1})), 
e^{-\fra}e^{\fra^\dagger} u + O(t^{-1}))_{\rm orb}  
\end{align*} 
where we used $e^{-2t\omega/z} z^{-\mu} (2t)^{\deg/2} 
= z^{-\mu}(2t)^{\deg/2} e^{-\fra}$ and (\ref{eq:relation_kappa_H_V}) 
in the second line (we assume $|z|=1$) and set 
$\kappa_t := (2t)^{-\deg/2} \kappa_\cV (2t)^{\deg/2}$ again 
in the third line. 
From (\ref{eq:kappat_goestoC}),    
the highest order term in $t$ becomes 
\[
(2t)^{k} ( (-1)^{-\mu-c} e^{-\fra^\dagger} \cC(u), 
e^{-\fra}e^{\fra^\dagger} u)_{\rm orb}.  
\] 
Note that $\cC$ anticommutes with $\fra,\fra^\dagger$ 
by (\ref{eq:H2_imaginary}). 
By a calculation using Lemma \ref{lem:expa_expadagger}, 
we find that this equals the highest order term 
of the right-hand side of (\ref{eq:nilpotentorb_Hermitianmetric}). 
The last statement on positivity follows from the 
classical Hodge-Riemann bilinear inequality: 
\[
(-1)^{(p-q)/2} (-1)^{(n_v-k)/2 - j} 
\int_{\cX_v} \omega^{k+2j} \phi \cup \ov{\phi} >0  
\]
for $\phi \in PH^{n_v-k-2j}(\cX_v)\cap 
H^{p-j,q-j}(\cX_v)\setminus\{0\}$, $n_v-k$ even. 
\end{proof} 

Next we show that $x\mapsto \F_{-x\omega}$ is pure for 
$t=\Re(x)\gg 0$.   
We set $\F'_{-x\omega} = e^{x\omega/z}\F_{-x\omega}$. 
Again by (\ref{eq:FcapkappaF}) 
in Proposition \ref{prop:purity_condition} 
and $\kappa_\cH e^{-x\omega/z} = 
e^{\ov{x}\omega/z} \kappa_\cH$, 
it is sufficient to show that 
\[
\F'_{-x\omega} \cap e^{2 t\omega/z} \kappa_{\cH}(\F'_{-x\omega}) 
\longrightarrow \F'_{-x\omega}/ z \F'_{-x\omega} 
\]
is an isomorphism. 
Put $\kappa^t = e^{2t\omega/z}\kappa_{\cH}$ 
($\kappa^t$ is different from $\kappa_t$ 
appearing in (\ref{eq:kappat_goestoC})).  
Fix a basis $\{\phi_1,\dots,\phi_N\}$ of 
$H^*_{\rm orb}(\cX)$. 
Define an $N\times N$ matrix $A_t(z,z^{-1})$ by 
\begin{equation}
\label{eq:matrixA}
[\kappa^t(\phi_1),\dots,\kappa^t(\phi_N)] 
= [\phi_1,\dots,\phi_N] A_t(z,z^{-1}).  
\end{equation} 
This matrix $A_t$ is a Laurent polynomial in $z$ 
(by (\ref{eq:relation_kappa_H_V})) and 
a polynomial in $t$. 
We already showed that (\ref{eq:nilpotentorb_pure}) 
is an isomorphism for $t=\Re(x) \gg 0$. Therefore, 
$A_t(z)$ admits the Birkhoff factorization 
$A_t(z) = B_t(z) C_t(z)$ for $t\gg 0$, 
where $B_t\colon \D_0 \to GL_N(\C)$ with $B_t(0)=\unit$ and 
$C_t\colon \D_\infty \to GL_N(\C)$ 
(see Remark \ref{rem:Birkhoff_Iwasawa}).  
The matrix $B_t(z)$ here is given by 
\[
[\varpi_t(\phi_1),\dots, \varpi_t(\phi_N)] 
= [\phi_1,\dots,\phi_N] B_t(z)
\] 
for $\varpi_t(\phi_i)$ appearing in (\ref{eq:lift_to_FcapkappaF}).  
In particular, $B_t(z)$ and $C_t(z)$ are polynomials in $z$ 
and $z^{-1}$ respectively and have 
at most polynomial growth in $t$. 
We define 
$Q_x \colon \Proj^1\setminus \{0\} \to GL_N(\C)$ by 
\begin{equation}
\label{eq:matrixQ}
[j_1,\dots,j_N] = [\phi_1,\dots,\phi_N] Q_x(z), \quad 
j_i :=  e^{x\omega/z} \cJ_{-x\omega}(\phi_i)
\end{equation} 
where $\cJ_{\tau}$ is given in (\ref{eq:fundamentalsol_A}). 
The vectors $j_1,\dots, j_N$ form a basis of $\F'_{-t\omega}$
and $Q_x(\infty) = \unit$. 
Note that $Q_x = \unit + O(e^{-\epsilon_0 t})$ as 
$t=\Re(x) \to \infty$  
for $\epsilon_0 := \min(\pair{\omega}{d}\;;\; 
d\in \Eff_\cX\setminus\{0\})$. 
From (\ref{eq:matrixA}) and (\ref{eq:matrixQ}), we find 
\[
[\kappa^t(j_1),\dots,\kappa^t(j_N)] 
= [j_1,\dots,j_N] Q_x^{-1} A_t \ov{Q}_x,  
\]
where $\ov{Q}_x$ is the complex conjugate 
of $Q_x$ with $z$ restricted to $S^1 = \{|z|=1\}$. 
As we did in Remark \ref{rem:Birkhoff_Iwasawa}, 
it suffices to show that 
$Q^{-1}_x A_t \ov{Q}_x$ admits the Birkhoff factorization. 
We have 
\begin{align*}
Q_x^{-1} A_t \ov{Q}_x = B_t (B_t^{-1} Q_x^{-1} B_t) 
(C_t\ov{Q}_x C_t^{-1}) 
C_t   
\end{align*} 
and for $0<\epsilon<\epsilon_0$, 
\[
B_t^{-1} Q^{-1}_x B_t = \unit +  O(e^{-\epsilon t}), \quad 
C_t \ov{Q}_x C_t^{-1} = \unit+ O(e^{-\epsilon t}), \quad 
\text{ as } t=\Re(x) \to \infty.    
\]
Here we used that $B_t$ and $C_t$ 
have at most polynomial growth in $t$. 
By the continuity of Birkhoff factorization, 
$(B_t^{-1} Q^{-1}_x B_t) (C_t\ov{Q}_x C_t^{-1}) = 
\unit +O(e^{-\epsilon t})$ 
admits the Birkhoff factorization of the form: 
\begin{equation}
\label{eq:Birkhoff_withestimate}
(B_t^{-1} Q^{-1}_x B_t) (C_t\ov{Q}_x C_t^{-1}) = 
\tilde{B}_x(z) \tilde{C}_x(z), 
\quad \tilde{B}_x = \unit + O(e^{-\epsilon t}), \quad 
\tilde{C}_x = \unit+ O(e^{-\epsilon t}) 
\end{equation} 
for $t=\Re(x)\gg 0$, where 
$\tilde{B}_x\colon \D_0\to GL_N(\C)$, $\tilde{B}_x(0)=\unit$ 
and $\tilde{C}_x\colon \D_\infty \to GL_N(\C)$. 
The order estimate $O(e^{-\epsilon t})$ holds 
in the $C^0$-norm on the loop space $C^\infty(S^1,\End(\C^N))$. 
See Appendix \ref{subsec:orderestimate_Birkhoff} for 
the proof of the order estimate in (\ref{eq:Birkhoff_withestimate}).   
Therefore $Q_x^{-1} A_t \ov{Q}_x$ also 
has the Birkhoff factorization 
for $t=\Re(x)\gg 0$ and we know that 
\[
 [\Pi_x(\phi_1), \dots,  \Pi_x(\phi_N) ] 
:= [\phi_1,\dots,\phi_N] Q_x(z) B_t(z) \tilde{B}_x(z) 
\] 
form a basis of $\F'_{-x\omega} \cap \kappa^t(\F'_{-x\omega})$, 
\emph{i.e.} 
$e^{-x\omega/z} \Pi_x(\phi_1),\dots, e^{-x\omega/z}\Pi_x(\phi_N)$ 
form a basis of $\F_{-x\omega}\cap \kappa_\cH(\F_{-x\omega})$. 
Using that $\Pi_x(\phi_i) = \varpi_x(\phi_i) + 
O(e^{-\epsilon t})$ 
and Proposition \ref{prop:nilpotentorb_pure}, we have  
\[
(-1)^{(p-q)/2}
(\kappa_{\cH}(e^{-x\omega/z} \Pi_x(\phi)), 
e^{-x\omega/z} \Pi_x(\phi) )_{\cH^\cX} >0, \quad 
\phi\in H^{p,q}(\cX_v)\setminus \{0\} 
\]
for sufficiently big $\Re(x)>0$. 
This completes the proof of Theorem \ref{thm:pure_polarized}.

\subsection{An A-model integral structure} 
\label{subsec:Amodel_intstr}
Real or integral structures in the A-model \seminf VHS 
(in the sense of Definition \ref{def:realintstr}) are not unique. 
In this section, we construct an example of 
A-model integral structures 
which comes from $K$-theory and 
will make sense for general symplectic orbifolds. 
By Theorem \ref{thm:pure_polarized}, 
this yields a Cecotti-Vafa 
structure near the large radius limit point. 
For weak Fano toric orbifolds, 
assuming Conjecture \ref{conj:toricmirror}, 
we will show in Section \ref{sec:intstr_via_toricmirrors} 
that this A-model integral structure 
coincides with the singularity B-model integral structure 
under mirror symmetry. 

Let $K(\cX)$ denote the Grothendieck group of topological 
orbifold vector bundles on $\cX$. 
See \emph{e.g.} \cite{adem-ruan, moerdijk} 
for vector bundles on orbifolds. 
For an orbifold vector bundle $\widetilde{V}$ 
on the inertia stack $I\cX$, 
we have an eigenbundle decomposition of $\widetilde{V}|_{\cX_v}$
\[
\widetilde{V}|_{\cX_v} = \bigoplus_{0\le f<1} 
\widetilde{V}_{v,f} 
\] 
with respect to the action of the stabilizer of $\cX_v$.    
Here, the stabilizer acts on 
$\widetilde{V}_{v,f}$ by $\exp(2\pi\iu f) \in \C$. 
Let $\pr \colon I\cX \to \cX$ be the projection. 
The Chern character map $\tch \colon K(\cX) \to H^*(I\cX)$ 
is defined by 
\[
\tch(V) := \bigoplus_{v\in \sfT} \sum_{0\le f<1} e^{2\pi\iu f}
\ch((\pr^*V)_{v,f}) 
\]
where $\ch$ is the ordinary Chern character and 
$V$ is an orbifold vector bundle on $\cX$. 
For an orbifold vector bundle $V$ on $\cX$, 
let $\delta_{v,f,i}$, $i=1,\dots,l_{v,f}$ be the Chern roots of 
$(\pr^*V)_{v,f}$. 
The Todd class $\tTd\colon K(\cX) \to H^*(I\cX)$ is defined by  
\[
\tTd(V) = \bigoplus_{v\in \sfT} 
\prod_{0<f<1,1\le i\le l_{v,f}}\frac{1}{1-e^{-2\pi\iu f}e^{-\delta_{v,f,i}}}
\prod_{f=0,1\le i\le l_{v,0}} \frac{\delta_{v,0,i}}{1-e^{-\delta_{v,0,i}}}
\]
We put $\Td_\cX := \tTd(T\cX)$. 
These characteristic classes appear in the following theorem. 

\begin{theorem}[Orbifold Riemann-Roch \cite{kawasaki-rr,toen}] 
Let $V$ be a holomorphic orbifold vector bundle on $\cX$.  
The holomorphic Euler characteristic 
$\chi(V):=\sum_{i=0}^{\dim \cX} (-1)^i \dim H^i(\cX,V)$ is given by 
\begin{equation}
\label{eq:orbifoldRR}
\chi(V)= 
\int_{I\cX} \tch(V)\cup \Td_\cX.  
\end{equation} 
\end{theorem}

Introduce another multiplicative 
characteristic class $\hGamma \colon K(\cX) \to H^*(I\cX)$ as 
\[
\hGamma(V) := \bigoplus_{v\in \sfT} 
\prod_{0\le f<1} \prod_{i=1}^{l_{v,f}} 
\Gamma(1- f + \delta_{v,f,i}) 
\in H^*(I\cX),  
\]
where $\delta_{v,f,i}$ is the same as above.  
The Gamma function on the right-hand side 
should be expanded in series at $1-f>0$. 
We put $\hGamma_\cX := \hGamma(T\cX)$. 
When defining the A-model integral structure below, 
we need to assume the following: 
\begin{itemize}
\item[(A1)] The map $\tch\colon K(\cX) \to H^*(I\cX)$ 
becomes an isomorphism after tensored with $\C$. 

\item[(A2)] The right-hand side of the orbifold Riemann-Roch formula 
(\ref{eq:orbifoldRR}) takes values in $\Z$ 
for any complex orbifold vector bundle $V$. 
Define $\chi(V)$ to be the value of the right-hand side of (\ref{eq:orbifoldRR})
for any orbifold vector bundle $V$. 

\item[(A3)] The pairing $(V_1,V_2)\mapsto 
\chi(V_1\otimes V_2)$ on $K(\cX)$ 
induces a surjective map $K(\cX) \to \Hom(K(\cX),\Z)$. 

\end{itemize}

\begin{remark}  
(i) When $\cX$ is a quotient orbifold of the form $\cX = Y/G$, 
where $Y$ is a compact manifold and $G$ is a compact Lie group 
acting on $Y$ with at most finite stabilizers,  
(A1) follows from Adem-Ruan's decomposition 
theorem \cite[Theorem 5.1]{adem-ruan}. 
Note that an orbifold without generic stabilizers 
can be presented as a quotient orbifold $Y/G$ 
(see \emph{e.g.} \cite{adem-ruan}). 

(ii) 
When $\cX$ is again a quotient orbifold $Y/G$,  
(A2) follows from Kawasaki's index theorem \cite{kawasaki-Vind} 
for elliptic operators on orbifolds (whose proof uses 
the $G$-equivariant index).   
The right-hand side of (\ref{eq:orbifoldRR}) 
becomes the index of a certain elliptic operator 
$\ov{\partial}+\ov{\partial}^*\colon 
V\otimes \Omega^{0,{\rm even}}_{\cX} \to 
V\otimes \Omega^{0,{\rm odd}}_{\cX}$, 
where $\ov{\partial}$ is a 
not necessarily integrable $(0,1)$ connection 
and $\ov{\partial}^*$ is its adjoint. 
The author does not know a purely topological proof. 

(iii) 
(A3) would follow from a universal coefficient theorem and 
Poincar\'{e} duality for 
orbifold $K$-theory (which are true for manifolds), 
but the author does not know a proof nor a reference.  
\end{remark}

\begin{definition-proposition}
\label{def-prop:A-model_int} 
Define $\cV^\cX_\Z\subset \cV^\cX = H_{\rm orb}^*(\cX)$ 
as the image of the map  
\begin{equation}
\label{eq:stdintstr_A}
\Psi \colon K(\cX) \longrightarrow \cV^\cX, \quad 
[V] \longmapsto \frac{1}{(2\pi)^{n/2}} 
\hGamma_\cX \cup (2\pi\iu)^{\deg/2} \inv^* (\tch(V)),   
\end{equation} 
where $\deg\colon H^*(I\cX)\to H^*(I\cX)$ is a grading operator 
on $H^*(I\cX)$ defined by $\deg = 2k$ on $H^{2k}(I\cX)$ 
and $\cup$ is the cup product in $H^*(I\cX)$. 
Then 

\begin{itemize}
\item[(i)] 
Under the assumption (A1) above, $\cV^\cX_\Z$ is a lattice in $\cV^\cX$ 
such that $\cV^\cX \cong \cV^\cX_\Z \otimes_\Z \C$. 
\item[(ii)]
The Galois action $G^\cV(\xi)$ on $\cV^\cX$ in (\ref{eq:Galois_V}) 
corresponds to tensoring by the line bundle $\otimes L_\xi^\vee$ 
in $K(\cX)$, i.e. 
$\Psi([V\otimes L_\xi^\vee]) = G^\cV(\xi)(\Psi([V]))$.  
\item[(iii)]  
The pairing $(\cdot,\cdot)_{\cV^\cX}$ on $\cV^\cX$ 
in (\ref{eq:pairing_Amodel_V}) 
corresponds to the Mukai pairing on $K(\cX)$ defined by 
$([V_1],[V_2])_{K(\cX)} := \chi(V_2^\vee \otimes V_1)$, i.e. 
$(\Psi([V_1]),\Psi([V_2]))_{\cV^\cX} = ([V_1],[V_2])_{K(\cX)}$.  
In particular, 
the pairing $(\cdot,\cdot)_{\cV^\cX}$ restricted 
on $\cV_\Z^\cX$ takes values in $\Z$ under assumption (A2) and 
is unimodular under assumption (A3). 
\end{itemize} 
Therefore $\cV^\cX_\Z$ satisfies 
the conditions in Proposition \ref{prop:char_A_real_int_str} 
and defines an integral structure in the A-model.  
We call $\cV^\cX_\Z$ a \emph{$\hGamma$-integral structure}. 
The real involution $\kappa_\cV$ on $\cV^\cX$ 
for the $\hGamma$-integral structure is given by 
\[
\kappa_{\cV} (\alpha) = (-1)^k \prod_{0\le f <1} 
\prod_{i=1}^{l_{\inv(v),f}} \frac{\Gamma(1-f+\delta_{\inv(v),f,i})}
{\Gamma(1-\ov{f}-\delta_{\inv(v),f,i})} \inv^* \ov{\alpha}, \quad 
\alpha \in H^{2k}(\cX_v) \subset \cV^\cX, 
\]
where $\delta_{\inv(v),f,i}$, $i=1,\dots,l_{\inv(v),f}$ are 
the Chern roots of $(\pr^*T\cX)_{\inv(v),f}$ and 
\[
\ov{f} := 
\begin{cases} 
1-f  & \text{if  $0<f<1$} \\
0    & \text{if  $f=0$}. 
\end{cases}
\]
Therefore, this $\kappa_\cV$ satisfies 
(\ref{eq:kappa_induces_Inv}) and (\ref{eq:leadingterm_kappaV}).  
In particular, the conclusions of Theorem \ref{thm:pure_polarized}
hold for the $\hGamma$-integral structure 
on the A-model \seminf VHS. 
\end{definition-proposition} 
\begin{proof} 
Because $\hGamma_\cX\cup$ and $(2\pi\iu)^{\deg/2}$ are invertible
operators over $\C$, (A1) implies (i). 
It is easy to check the second statement (ii).  
For (iii), we calculate 
\begin{align*}
(\Psi(V_1),\Psi(V_2))_{\cV^\cX} 
&= (e^{\pi \iu\mu} \Psi(V_2), e^{\pi \iu \rho}\Psi(V_1))_{\rm orb}   \\ 
&= \frac{1}{(2\pi)^n} \sum_{v\in \sfT} 
(2\pi\iu)^{\dim \cX_v} \times \\ 
\int_{\cX_v}  
\prod_{f,i} \Gamma(1-f-&\tfrac{\delta_{v,f,i}}{2\pi\iu})
\Gamma(1-\ov{f}+\tfrac{\delta_{v,f,i}}{2\pi\iu}) 
\cdot e^{\frac{\rho}{2}} \tch(V_1)_v \cdot
e^{\pi\iu(\iota_v-\frac{n}{2}+\frac{\deg}{2})}\tch(V_2)_{\inv(v)},  
\end{align*} 
where 
$\tch(V_1)_v$ and $\tch(V_2)_{\inv(v)}$ are $H^*(\cX_v)$ 
and $H^*(\cX_{\inv(v)})$ components 
of $\tch(V_1)$ and $\tch(V_2)$ respectively.  
Using $\Gamma(1-z)\Gamma(z) = \pi/\sin(\pi z)$ 
and $\sum_{f,i}\delta_{v,f,i} =\pr^*\rho|_{\cX_v}$,   
we calculate    
\[
\prod_{f,i} \Gamma(1-f-\tfrac{\delta_{v,f,i}}{2\pi\iu})
\Gamma(1-\ov{f}+\tfrac{\delta_{v,f,i}}{2\pi\iu})  
= (2\pi\iu)^{n-\dim \cX_v} 
e^{-\frac{\rho}{2}} e^{-\pi\iu \iota_v} \Td_{\cX,v}.  
\]
Putting these together, we have 
$(\Psi(V_1),\Psi(V_2))_{\cV^\cX} = \chi(V_2^\vee\otimes V_1)$. 
A straightforward calculation shows the rest of the statements. 
\end{proof} 

\begin{remark}
Instead of working with topological $K$-groups, 
we can use the $K$-group of 
algebraic vector bundles (or coherent sheaves) 
on the smooth Deligne-Mumford stack $\cX$. 
In this case, the map $\Psi$ in 
Definition-Proposition \ref{def-prop:A-model_int} 
defines an integral structure on the algebraic A-model \seminf VHS 
introduced in Remark \ref{rem:after_pure_polarized}. 
\end{remark} 

\section{Integral structures via toric mirrors}
\label{sec:intstr_via_toricmirrors}
In previous sections, we studied properties of 
arbitrary real and integral structures. 
In this section, we use mirror symmetry to find 
the ``most natural" integral structure in the A-model.  
We calculate the integral structures in the singularity 
mirrors (Landau-Ginzburg model) of toric orbifolds and 
study the pulled back integral structures in orbifold quantum 
cohomology.

\subsection{Toric orbifolds} 
\label{subsec:toricorbifolds} 
To fix the notation, we give the definition 
of toric orbifolds and collect several facts.  
By a toric orbifold, we mean 
a toric Deligne-Mumford stack in the sense of 
Borisov-Chen-Smith \cite{borisov-chen-smith}. 
We only deal with a compact toric orbifold 
with a projective coarse moduli space 
and define a toric orbifold as a quotient of 
$\C^m$ by a \emph{connected} torus $\T \cong (\C^*)^r$. 
For toric orbifolds/varieties, 
we refer the reader to \cite{borisov-chen-smith, oda, fulton, audin}. 

\subsubsection{Definition}  
\label{subsubsec:def_toricorbifolds}
We begin with the following data: 
\begin{itemize}
\item an $r$-dimensional algebraic torus $\T\cong (\C^*)^r$; we set 
$\bL:=\Hom(\C^*,\T)$; 
\item $m$ elements $D_1,\dots,D_m \in \bL^\vee = \Hom(\T,\C^*)$ 
such that $\bL^\vee\otimes \R= \sum_{i=1}^m \R D_i$; 
\item a vector $\eta\in \bL^\vee \otimes \R$.  
\end{itemize} 
The elements $D_1,\dots, D_m$ define a 
homomorphism $\T\rightarrow (\C^*)^m$. 
Let $\T$ act on $\C^m$ via this homomorphism. 
The vector $\eta$ defines a stability condition 
of this torus action. 
Set 
\[
\cA := \{ I\subset \{1,\dots,m\} \;;\; 
\sum_{i\in I} \R_{>0} D_i \ni \eta\}.   
\]
We define a quotient stack $\cX$ to be  
\[
\cX = [\cU_\eta/\T], \quad 
\cU_\eta := \C^m \setminus \bigcup_{I\notin \cA} \C^I,  
\]
where
$\C^I := \{(z_1,\dots,z_m)\;;\; z_i= 0 \text{ for } i\notin I\}$.
Under the following conditions, $\cX$ is a smooth Deligne-Mumford 
stack with a projective coarse moduli space:   
\begin{itemize}
\item[(A)] $\{1,\dots,m\} \in \cA$. 
\item[(B)] $\sum_{i\in I} \R D_i = \bL^\vee \otimes \R$ for $I\in \cA$. 
\item[(C)] $\{(c_1,\dots,c_m)\in \R_{\ge 0}^m \;;\; 
\sum_{i=1}^m c_i D_i =0\} = \{0\}$. 
\end{itemize} 
The conditions (A), (B) and (C) ensure 
that $\cX$ is non-empty, that the stabilizer is finite 
and that $\cX$ is compact respectively. 
The generic stabilizer of $\cX$ is given by the kernel of 
$\T \to (\C^*)^m$ and $\dim_\C\cX=n:=m-r$.

We can also construct $\cX$ as a symplectic quotient as follows 
(see also \cite{audin}).  
Let $\T_\R$ denote the maximal compact subgroup of $\T$ 
isomorphic to $(S^1)^r$. 
Let $\mathfrak{h} \colon \C^m \to \bL^\vee\otimes\R$ be the moment 
map for the $\T_\R$-action on $\C^m$: 
\[
\mathfrak{h}(z_1,\dots,z_m) = \sum_{i=1}^m |z_i|^2 D_i.  
\]
The $\T_\R$-action on the level set $\mathfrak{h}^{-1}(\eta)$ 
has only finite stabilizers and we have an isomorphism of 
symplectic orbifolds: 
\begin{equation}
\label{eq:X_symplecticquot}
\cX \cong \mathfrak{h}^{-1}(\eta) / \T_\R.  
\end{equation} 

By renumbering the indices if necessary, we can assume that 
\[
\{1,\dots,m\} \setminus \{i\} \in \cA \quad \text{ if and only if } 
\quad 1\le i\le m' 
\]
where $m'$ is less than or equal to $m$. 
We can easily check that 
$I\supset \{m'+1,\dots,m\}$ for any $I\in \cA$ 
and $D_{m'+1},\dots,D_m$ are linearly independent 
over $\R$. 
The elements $D_1,\dots,D_m$ define the following exact sequence 
\begin{equation}
\label{eq:exactsequence_toric}
\begin{CD}
0 @>>> \bL @>{(D_1,\dots,D_m)}>> \Z^m @>{\beta}>> \bN @>>> 0,  
\end{CD}  
\end{equation} 
where $\bN$ is a finitely generated abelian group. 
By the long exact sequence associated with 
the functor $\Tor_\bullet(-,\C^*)$, 
we find that the torsion part 
$\bN_{\rm tor}=\Tor_1(\bN,\C^*)$ of $\bN$ is isomorphic to  
the generic stabilizer $\Ker(\T\to (\C^*)^m)$. 
The free part 
$\bN_{\rm free}=\bN/\bN_{\rm tor}$ is of rank $n=\dim_\C\cX$. 
Let $b_1,\dots,b_m$ be the images in $\bN$ 
of the standard basis of $\Z^m$ under $\beta$. 
The {\it stacky fan} of $\cX$, in the sense of 
Borisov-Chen-Smith \cite{borisov-chen-smith}, is given by 
the following data: 
\begin{itemize}
\item vectors $b_1,\dots, b_{m'}$ in $\bN$;   
\item a complete simplicial fan $\Sigma$ in $\bN\otimes \R$ such that 

(i) the set of one dimensional cones 
is $\{\R_{\ge 0} b_1, \dots, \R_{\ge 0} b_m\}$; 

(ii) 
$\sigma_I = \sum_{i\notin I} \R_{\ge 0} b_i$ defines a cone 
of $\Sigma$ if and only if $I \in \cA$. 
\end{itemize} 
The toric variety defined by the fan $\Sigma$ is 
the coarse moduli space of $\cX$. 
The conditions (B) and (C) correspond to that $\Sigma$ is simplicial 
and that $\Sigma$ is complete, 
\emph{i.e.} the union of all cones in $\Sigma$ 
is $\bN\otimes \R$. 
We may refer to $\cA$ as the set of ``anticones"\footnote
{This name is due to Tom Coates.}. 

\begin{remark}
Borisov-Chen-Smith \cite{borisov-chen-smith} 
defined a toric Deligne-Mumford stack 
starting from data of a stacky fan. 
Our construction can give every toric Deligne-Mumford stack 
in their sense which has a projective coarse moduli space. 
Note that the vectors $b_{m'+1},\dots, b_{m}$ do not appear 
as data of a stacky fan. 
These redundant information in our initial data 
makes $\beta$ surjective and 
allows us to define $\cX$ as a quotient 
by a \emph{connected} torus $\T$. 
\end{remark}

\subsubsection{K\"{a}hler cone and a choice of a nef basis}
Since every element of $\cA$ contains $\{m'+1,\dots,m\}$, 
it is convenient to put 
\[
\cA' = \{ I'\subset \{1,\dots,m'\}\;
;\; I'\cup\{m'+1,\dots,m\}\in \cA\}. 
\]
We can easily see that $\cU_\eta$ factors as 
\[
\cU_\eta = \cU'_\eta \times (\C^*)^{m-m'}, \quad 
\cU'_\eta = \C^{m'} \setminus \bigcup_{I'\notin \cA'} \C^{I'}. 
\]
Thus we can write 
\[
\cX = [\cU'_\eta/\G], \quad 
\G:=\Ker(\T\to (\C^*)^m \to (\C^*)^{\{m'+1,\dots,m\}}).
\] 
Note that $\G$ is isomorphic to $(\C^*)^{r'}$ 
times a finite abelian group 
for $r':=r-(m-m')$.  
Every character $\xi \colon \G \to \C^*$ of $\G$ 
defines an orbifold line bundle 
$L_\xi:=\C\times_{\xi,\G} \cU'_\eta \to \cX$. 
Under this correspondence between $\xi$ and $L_\xi$, 
the Picard group $\Pic(\cX)$ is identified 
with the character group $\Hom(\G,\C^*)$ 
and also with $H^2(\cX,\Z)$ (via $c_1$): 
\[
\Pic(\cX) \cong \Hom(\G,\C^*) 
\cong \bL^\vee/\textstyle\sum_{i=m'+1}^m \Z D_i 
\cong H^2(\cX,\Z).  
\]
The image $\ov{D}_i$ of $D_i$ in $H^2(\cX,\R)$ 
is the Poincar\'{e} dual of 
the toric divisor $\{z_i=0\}\subset \cX$ for 
$1\le i\le m'$. 
Over rational numbers, we have 
\begin{align*}
H^2(\cX,\Q) & \cong  
\bL^\vee\otimes \Q/\textstyle \sum_{i=m'+1}^m \Q D_i,  \\
H_2(\cX,\Q) & \cong
\Ker((D_{m'+1},\dots, D_m)\colon \bL\otimes \Q \to \Q^{m-m'}) 
\subset \bL\otimes \Q. 
\end{align*} 
Now we introduce a canonical splitting (over $\Q$) 
of the surjection $\bL^\vee \otimes \Q \to H^2(\cX,\Q)$.   
For $m'<j\le m$,  $b_j$ is contained in some cone in $\Sigma$ 
since $\Sigma$ is complete. Namely, 
\begin{equation}
\label{eq:bjcontainedinacone}
b_j = \sum_{i\notin I_j} c_{ji} b_i, \quad \text{ in } \bN\otimes \Q, 
\quad c_{ji}\ge 0, \quad \exists I_j \in \cA, 
\end{equation} 
where $I_j$ corresponds to the complement of 
the cone containing $b_j$ in its interior. 
By the exact sequence 
(\ref{eq:exactsequence_toric}) tensored with $\Q$, 
we can find $D^\vee_j \in \bL\otimes \Q$ such that 
\[
\pair{D_i}{D^\vee_j} = 
\begin{cases} 
1 \quad & i=j \\
-c_{ji} \quad & i\notin I_j \\
0 \quad & i\in I_j\setminus \{j\}.  
\end{cases}
\]
Note that $D^\vee_j$ is uniquely determined by these conditions. 
These vectors $D^\vee_j$ define a decomposition 
\begin{align*}
\bL^\vee \otimes \Q  &= 
\Ker((D_{m'+1}^\vee,\dots,D_{m}^\vee)\colon 
\bL^\vee\otimes \Q\to \Q^{m-m'}) 
\oplus \bigoplus_{j=m'+1}^m \Q D_j.  
\end{align*} 
The first factor $\Ker(D_{m'+1}^\vee,\dots,D_{m}^\vee)$ 
is identified with $H^2(\cX,\Q)$ under the surjection 
$\bL^\vee\otimes \Q \to H^2(\cX,\Q)$. 
Using this, we regard $H^2(\cX,\Q)$ 
as a subspace of $\bL^\vee\otimes \Q$. 
We define an {\it extended K\"{a}hler cone} $\tilde{C}_\cX$ as  
\[
\tilde{C}_\cX  = \bigcap_{I\in \cA} (\sum_{i\in I} \R_{>0} D_i) 
\subset \bL^\vee \otimes \R. 
\] 
Then $\eta\in \tilde{C}_\cX$ and the image of $\eta$ in $H^2(\cX,\R)$ 
is the class of the reduced symplectic form. 
The set $\tilde{C}_\cX$ is the connected component of the set of 
regular values of the moment map 
$\mathfrak{h}\colon \C^m \to \bL^\vee\otimes \R$, 
which contains $\eta$. 
The extended K\"{a}hler cone depends not only on $\cX$ 
but also on the choice of our initial data. 
The genuine {\it K\"{a}hler cone} $C_\cX$ of $\cX$ is 
the image of $\tilde{C}_\cX$ 
under $\bL^\vee \otimes \R \to H^2(\cX,\R)$: 
\[
C_\cX = \bigcap_{I'\in \cA'} (\sum_{i\in I'} \R_{>0} \ov{D}_i) \subset 
H^2(\cX,\R) = H^{1,1}(\cX,\R) 
\]
where $\ov{D}_i$ is the image of $D_i$ in $H^2(\cX,\R)$.  
The next lemma means that the extended K\"{a}hler cone also ``splits". 
\begin{lemma}
\label{lem:decomp_extendedKaehler} 
$\tilde{C}_\cX = C_\cX + \sum_{j=m'+1}^m \R_{>0} D_j$ 
in $\bL^\vee \otimes \R \cong 
H^2(\cX,\R)\oplus \bigoplus_{j=m'+1}^m \R D_j$. 
\end{lemma} 
\begin{proof}
First note that for $1\le i\le m'$, 
$\ov{D}_i = D_i + \sum_{j>m'} c_{ji} D_j$, 
where $c_{ji} = -\pair{D_i}{D_j^\vee}\ge 0$. 
Take $I'\in \cA'$ and put $I=I'\cup \{m'+1,\dots,m\}$. 
It is easy to check that 
\[
\sum_{i\in I'} \R_{>0} \ov{D}_i + \sum_{j=m'+1}^{m} \R_{>0}D_j 
 = \sum_{k\in I} \R_{>0} D_k \cap \bigcap_{j=m'+1}^m 
\{D_j^\vee>0\},  
\]
where we regard $D_j^\vee$ 
as a linear function on $\bL^\vee\otimes \R$. 
Thus $C_\cX+\sum_{j>m'} \R_{>0} D_j = \tilde{C}_\cX \cap 
\bigcap_{j=m'+1}^m \{D_j^\vee>0\}$. 
For $j>m'$, take $I_j\in \cA$ appearing in 
(\ref{eq:bjcontainedinacone}). 
Then $\tilde{C}_\cX \subset \sum_{k\in I_j} \R_{>0} D_k\subset 
\{D_j^\vee>0\}$. The conclusion follows. 
\end{proof} 

We choose an integral basis $\{p_1,\dots,p_r\}$ of $\bL^\vee$ such that 
$p_a$ is in the closure $\cl(\tilde{C}_\cX)$ of 
$\tilde{C}_\cX$ for all $a$ 
and $p_{r'+1},\dots,p_r$ are in $\sum_{i=m'+1}^m \R_{\ge 0} D_i$. 
Then the images $\ov{p}_1,\dots, \ov{p}_{r'}$ of $p_1,\dots,p_{r'}$ 
in $H^{2}(\cX,\R)$ are nef and those of 
$p_{r'+1},\dots, p_r$ are zero. 
Define a matrix $(m_{ia})$ by  
\begin{equation}
\label{eq:Dprel}
D_i = \sum_{a=1}^r m_{ia} p_a, \quad m_{ia} \in \Z.  
\end{equation} 
Then the class $\ov{D}_i$ of the toric divisor 
$\{z_i=0\}$ is given by 
\begin{equation}
\label{eq:ovDprel}
\ov{D}_i = \sum_{a=1}^{r'} m_{ia} \ov{p}_a,  
\end{equation} 
and $\ov{D}_j =0$ for $m'< j\le m$. 

\subsubsection{Inertia components}
We introduce subsets $\K$, $\K_{\rm eff}$ of $\bL\otimes \Q$ by 
\begin{align*}
\K &= \{d\in \bL\otimes \Q \;;\; \{i\in \{1,\dots,m\} 
\;;\; \pair{D_i}{d}\in \Z \}\in \cA\}, \\
\K_{\rm eff} & = \{ d\in \bL\otimes \Q \; ; \; 
\{i \in \{1,\dots,m\}  \; ; 
\; \pair{D_i}{d} \in \Z_{\ge 0} \} \in \cA\}.    
\end{align*} 
Note that $\K$ and $\K_{\rm eff}$ are not closed under addition but 
$\K$ is acted on by $\bL$. 
An element of $\K_{\rm eff}\cap H_2(\cX,\R)$ can be 
realized as a degree of a stable map of the form 
$\Proj(1,a)\to \cX$. 
Following \cite{borisov-chen-smith}, 
we introduce the set $\Boxop$ as 
\[
\Boxop = \Big \{ v\in \bN \;;\; v = \sum_{k\notin I} c_k b_k \text{ in } 
\bN \otimes \Q, \ c_k\in [0,1), \ I \in \cA  \Big \}. 
\]
For a real number $r$, let $\ceil{r}$, $\floor{r}$ 
and $\{r\}$ denote the ceiling, floor and fractional part 
of $r$ respectively. 
For $d\in \K$, define an element $v(d)\in \Boxop$ by 
\[
v(d) := \sum_{i=1}^m \ceil{\pair{D_i}{d}} b_i \in \bN. 
\]
By the exact sequence (\ref{eq:exactsequence_toric}), 
we have $v(d)= \sum_{i=1}^m \{-\pair{D_i}{d}\} b_i$ 
in $\bN\otimes \Q$ and so $v(d)$ is an element of $\Boxop$. 
This map $d\mapsto v(d)$ factors through $\K \to \K/\bL$ 
and identifies $\K/\bL$ with $\Boxop$. 
For $d\in \K$, we define a component $\cX_{v(d)}$ 
of the inertia stack $I\cX$ by 
\[
\cX_{v(d)} = \{[z_1,\dots,z_m] \in \cX \; ;\; z_i = 0 \text{ if } 
\pair{D_i}{d} \notin \Z \}.  
\]
The stabilizer along $\cX_{v(d)}$ is defined to be   
$\exp(-2\pi\sqrt{-1}d) \in 
\bL\otimes \C^*\cong \T$, which acts on $\C^m$ by 
\[
(e^{-2\pi\iu\pair{D_1}{d}}, \cdots, e^{-2\pi\iu\pair{D_m}{d}}).  
\] 
We can easily check that $\cX_{v(d)}$ depends only on 
the element $v(d) \in \Boxop$. 
When $d\in \K_{\rm eff}\cap H_2(\cX,\Q)$,  
the evaluation image of a stable map 
$\Proj(1,a)\to \cX$ of degree $d$ at $[\Proj(a)]\in \Proj(1,a)$ 
lies on $\cX_{\inv(v(d))}$.  
The age of $\cX_{v(d)}$ is calculated as 
\begin{equation} 
\label{eq:ageofX_v}
\iota_{v(d)} : = \age (\cX_{v(d)}) = \sum_{i=1}^m \{ - \pair{D_i}{d} \} 
= \sum_{i=1}^{m'} \{ -\pair{D_i}{d} \}.   
\end{equation} 
We have 
\[
I\cX = \bigsqcup_{v\in \Boxop} \cX_v, \quad 
H_{\rm orb}^i(\cX) = \bigoplus_{v\in \Boxop} H^{i - 2 \iota_v} (\cX_v).  
\]
Denote by $\unit_v$ the unit class of $H^*(\cX_v)$. 
The coarse moduli space of $\cX_v$ is a toric variety and 
its cohomology ring is generated by the 
degree two classes $\ov{p}_1,\dots,\ov{p}_{r'}$: 
\[
H^*(\cX_{v(d)}) = \left(\C[\ov{p}_1,\dots,\ov{p}_{r'}]/
\mathfrak{J}_{v(d)}\right) \unit_{v(d)} 
\]
where the ideal $\mathfrak{J}_{v(d)}$ is 
generated by $\prod_{i\in I} \ov{D}_i$ 
for $I\subset \{1,\dots,m\}$ such that 
$\{i\;;\; \pair{D_i}{d} \in \Z\} \setminus I\notin \cA$. 
(See (\ref{eq:ovDprel}) for $\ov{D}_i$ in terms of $\ov{p}_a$.)  
Take $\xi\in \bL^\vee$. Let $[\xi]$ be its image in 
$\bL^\vee/\sum_{j=m'+1}^m \Z D_j \cong H^2(\cX,\Z)$. 
The age $f_v([\xi])\in [0,1)$ of the line bundle 
$L_\xi$ (introduced in Section \ref{subsec:AmodelVHS}) is given by 
\begin{equation}
\label{eq:toric_fv}
f_{v(d)}([\xi]) = \{-\pair{\xi}{d}\}, \quad d\in \K. 
\end{equation} 

\subsubsection{Weak Fano condition} 
\label{subsubsec:weakFano} 
The first Chern class $\rho=c_1(T\cX)\in H^2(\cX,\Q)$ of 
$\cX$ is the image of the vector 
\[
\hrho := D_1 +\cdots + D_m = \sum_{a=1}^r \rho_a p_a  
\in \bL^\vee, \quad 
\rho_a := \sum_{i=1}^m m_{ia}.  
\] 
$\cX$ is {\it weak Fano} if $\rho$ is in the closure $\cl(C_\cX)$ 
of the K\"{a}hler cone $C_\cX$. 
Later, we will use the condition that 
$\hrho\in \cl(\tilde{C}_\cX)$ 
which is stronger than that $\cX$ is weak Fano. 
\begin{lemma}
We have $\hrho \in \cl(\tilde{C}_\cX)$ if and only if 
$\rho \in \cl(C_\cX)$ (i.e. $\cX$ is weak Fano) 
and $\age(b_j) \le 1$ for all $j>m'$.
Here we put $\age(b_j) := \sum_{i\notin I_j} c_{ji}$. 
(This coincides with $\iota_{b_j}$ in (\ref{eq:ageofX_v})   
when $b_j\in \Boxop$. 
See (\ref{eq:bjcontainedinacone}) for the definition 
of $I_j$ and $c_{ji}$.) 
\end{lemma} 
\begin{proof}
From $\ov{D}_i = D_i + \sum_{j>m'} c_{ji}D_j$, we have 
\[
\hrho = \rho + \sum_{j>m'} (1-\age (b_j)) D_j  
\]
The conclusion follows from Lemma \ref{lem:decomp_extendedKaehler}. 
\end{proof} 
When $\hrho\in \cl(\tilde{C}_\cX)$, we can choose a basis 
$p_1,\dots,p_r\in \cl(\tilde{C}_\cX)$ so that $\hrho$ is in the 
cone generated by $p_a$'s. 
Thus \emph{in this case, we will assume 
$\rho_a\ge 0$ without loss of generality}. 

\begin{remark}
When $\cX$ is weak Fano, we can choose our toric data 
so that $\hrho \in \cl(\tilde{C}_\cX)$ 
if the following condition holds for the stacky fan: 
\[
\{v\in \Boxop \;;\; \age(v)\le 1\} \cup \{b_1,\dots,b_{m'}\} 
\text{ generates $\bN$ over $\Z$.} 
\]
If this holds, we can choose $b_{m'+1},\dots,b_m \in \Boxop$ 
so that $\{b_1,\dots,b_m\}$ generates $\bN$ and 
$\age(b_j)\le 1$ for $m'<j\le m$.    
Then the exact sequence (\ref{eq:exactsequence_toric}) 
determines $D_1,\dots,D_m$ and 
$\hrho=D_1+\dots+D_m\in \cl(\tilde{C}_\cX)$ holds. 
That $\bN$ is generated only by $b_1,\dots,b_{m'}$ 
is equivalent to that $\cX$ is simply-connected 
in the sense of orbifold: $\pi_1^{\rm orb}(\cX)=\{1\}$. 
\end{remark} 

\begin{remark} 
The vectors $D_j$, $m' < j \le m$ in $\bL^\vee$ 
``correspond" to the following elements 
in the twisted sector:  
\begin{equation}
\label{eq:frD}
\frD_j = \prod_{i\notin I_j} 
\ov{D}_i^{\floor{c_{ji}}} \unit_{v(D^\vee_j)} 
\in H_{\rm orb}^{*}(\cX), \quad \text{where } 
v(D^\vee_j) = b_j + \sum_{i\notin I_j} \ceil{-c_{ji}} b_i. 
\end{equation} 
This correspondence can be seen from 
the expansion of the mirror map $\tau(q)$ 
in Remark \ref{rem:aftermirrorconj} below. 
We have $\frD_j=\unit_{b_j}$ when $b_j\in \Boxop$. 
Therefore, if $\hrho\in \cl(\tilde{C}_\cX)$ and 
$b_{m'+1},\dots,b_{m}$ are mutually different elements in $\Boxop$, 
we can identify $\bL^\vee\otimes \C$ 
with the subspace $H^2(\cX) \oplus \bigoplus_{j>m'} H^0(\cX_{b_j})$ 
of $H^{\le 2}_{\rm orb}(\cX)$. 
\end{remark} 

\subsection{Landau-Ginzburg model} 
\label{subsec:LGmodel} 
Following \cite{givental-ICM, givental-mirrorthm-toric, hori-vafa}, 
we use the Landau-Ginzburg model as a mirror of a toric variety. 
By applying the exact functor $\Hom(-,\C^*)$ to 
the short exact sequence (\ref{eq:exactsequence_toric}), 
we have 
\begin{equation}
\label{eq:fibration_LG}
\begin{CD}
\unit @>>> \Hom(\bN,\C^*) 
@>>> Y:= (\C^*)^m @>{\pr}>> \cM:= \Hom(\bL,\C^*) @>>> \unit.  
\end{CD} 
\end{equation} 
The {\it Landau-Ginzburg model} associated to a toric orbifold 
is the family $\pr \colon Y \to \cM$ of affine varieties 
given by the third arrow and a fiberwise Laurent polynomial 
$W\colon Y \to \C$, called potential, given by 
\[
W=w_1+\cdots+w_m 
\] 
where $w_1,\dots,w_m$ are the standard 
$\C^*$-valued co-ordinates on $Y=(\C^*)^m$. 
The basis of $\bL$ dual to $p_1,\dots,p_r$ in the previous section 
defines $\C^*$-valued co-ordinates 
$q_1,\dots,q_r$ on $\cM=\Hom(\bL,\C^*)$. 
Then the projection is given by (see (\ref{eq:Dprel})) 
\begin{equation}
\label{eq:fibration_LG_formula}
\pr(w_1,\dots,w_m) 
= (q_1,\dots,q_r), \quad q_a = \prod_{i=1}^m w_i^{m_{ia}}.  
\end{equation}
Let $Y_q := \pr^{-1}(q)$ be the fiber at $q\in \cM$ 
and set $W_q:=W|_{Y_q}$. 
Note that $Y_q$ has $|\bN_{\rm tor}|$ connected components and 
each connected component is isomorphic to 
$\Hom(\bN_{\rm free},\C^*) \cong (\C^*)^n$. 
Let $e_1,\dots,e_n$ be an arbitrary basis of $\bN_{\rm free}$ and 
$y_1,\dots,y_n$ be the corresponding 
$\C^*$-valued co-ordinate on $\Hom(\bN_{\rm free},\C^*)$. 
We choose a splitting of the exact sequence dual to  
(\ref{eq:exactsequence_toric}) over rational numbers. 
Namely, we take a matrix $(\ell_{ia})_{1\le i\le m, 1\le a\le r}$ with 
$\ell_{ia}\in \Q$ such that $p_a = \sum_{i=1}^m D_i \ell_{ia}$. 
This splitting defines a multi-valued section of $\pr:Y\to \cM$ 
and identifies $Y_q$ with $\Hom(\bN,\C^*)$. 
Under this identification, $y_1,\dots,y_n$ give 
co-ordinates on each connected component of $Y_q$ and we have 
\begin{equation}
\label{eq:W_q}
W|_{Y_q} = W_q = q^{\ell_1} y^{b_1} + \cdots + q^{\ell_m} y^{b_m}, \quad
q^{\ell_i} = \prod_{a=1}^r q_a^{\ell_{ia}}, \quad  
y^{b_i} = \prod_{j=1}^n y_j^{b_{ij}},  
\end{equation} 
where $b_i = \sum_{j=1}^n b_{ij} e_j$ in $\bN_{\rm free}$. 
Here, the choice of the branches of fractional powers of $q_a$ 
appearing in $q^{\ell_i}$ depends on a connected component of $Y_q$. 

To proceed further, we need to restrict the parameter 
$q\in \cM$ to some Zariski open subset $\cMo\subset \cM$ 
so that $W_q$ satisfies the  ``non-degeneracy condition at infinity" 
due to Kouchnirenko \cite[1.19]{kouchnirenko}. 
\begin{definition}
Let $\hat{S}$ denote the convex hull of 
$b_1,\dots,b_m\in \bN\otimes \R$. 
We call the Laurent polynomial $W_q(y)$ of the form (\ref{eq:W_q}) 
{\it non-degenerate at infinity} 
if for every face 
$\Delta$ of $\hat{S}$ (where $0\le \dim\Delta \le n-1$),   
$W_{q,\Delta}(y):=\sum_{b_i\in \Delta} q^{\ell_i} y^{b_i}$ 
does not have critical points on $y\in (\C^*)^n$. 
Let $\cMo$ be the subset of $\cM$ consisting of $q$ 
for which $W_q$ is non-degenerate at infinity.   
\end{definition} 

\begin{proposition}
\label{prop:kouchnirenko} 
(i) Under the condition (C) 
in Section \ref{subsubsec:def_toricorbifolds}, 
$0\in \bN\otimes \R$ is in the interior of $\hat{S}$. 
Therefore, the Laurent polynomial $W_q$ is 
\emph{convenient} in the sense of 
Kouchnirenko \cite[1.5]{kouchnirenko}. 

(ii) $\cMo$ is an open and dense subset of $\cM$ in Zariski topology.  

(iii) For $q\in \cMo$, $W_q(y)$ has   
$|\bN_{\rm tor}|\times n!\Vol(\hat{S})$ critical points 
on $Y_q$ (counted with multiplicities). 
\end{proposition} 
\begin{proof}

The condition (C) implies that 
there exists $d\in \bL$ such that $c_i:=\pair{D_i}{d}>0$. 
Then by the exact sequence (\ref{eq:exactsequence_toric}),  
we have $\sum_{i=1}^m c_i b_i=0$. This proves (i). 
The statements (ii) and (iii) are due to Kouchnirenko. 
(ii) follows from (i) 
and the same argument as in \cite[6.3]{kouchnirenko}. 
One of main theorems in \cite[1.16]{kouchnirenko} 
states that $W_q(y)$ has $n!\Vol(\hat{S})$ number of critical points 
on each connected component of $Y_q$. 
(iii) follows from this and $|\pi_0(Y_q)| = |\bN_{\rm tor}|$. 
\end{proof} 

Let $f_{q,z}\colon Y_q \to \R$ be the real part of the function 
$y\mapsto W_q(y)/z$. 
The following lemma allows us to use Morse theory 
for the improper function $f_{q,z}(y)$. 
\begin{lemma}
\label{lem:PScond}
For each $\epsilon>0$, the family of topological spaces 
\[ 
\bigcup_{(q,z)\in \cMo\times \C^*}
\{y\in Y_q \;;\; \|df_{q,z}(y)\|\le \epsilon \} \to \cMo\times \C^*
\]
is proper, i.e. pull-back of a compact set is compact.  
Here the norm $\|df_{q,z}(y)\|$ is taken with respect to 
the complete K\"{a}hler metric 
$\frac{1}{\iu} \sum_{i=1}^n d\log y_i \wedge d\ov{\log y_i}$ 
on $Y_q$. 
\end{lemma} 

A similar statement for polynomial functions can be found 
in \cite[Proposition 2.2 and Remarque]{pham_GM} 
and this lemma may also have been well-known. 
We will include a proof in Appendix \ref{subsec:proof_PS}
since we do not know a good reference. 
Lemma \ref{lem:PScond} implies that $f_{q,z}$ 
satisfies the Palais-Smale condition, 
so that usual Morse theory applies to $f_{q,z}$ 
(see \emph{e.g.} \cite{palais}). 
Take $(q,z)\in \cMo\times \C^*$. 
Since  the set $\{y\in Y_q\; ;\; \|df_{q,z}(y)\|<\epsilon\}$ 
is compact, we can choose $M\ll 0$ so that 
this set is contained in $\{y\in Y_q\;;\;f_{q,z}(y)>M\}$. 
Then the relative homology group 
$H_n(Y_q, \{y\in Y_q\;;\; f_{q,z}(y)\le M\};\Z)$ 
is independent of the choice of such $M$ and 
we denote this by 
\begin{equation}
\label{eq:relative_hom}
R_{\Z, (q,z)}^\vee = 
H_n(Y_q, \{y\in Y_q\;;\; f_{q,z}(y) \ll 0\};\Z), \quad 
(q,z) \in \cMo \times \C^*.    
\end{equation} 
The number of critical points of $f_{q,z}(y)$ is 
$N:=|\bN_{\rm tor}|\times n!\Vol(\hat{S})$ 
by Proposition \ref{prop:kouchnirenko}. 
If all the critical points of $W_q(y)$ 
are non-degenerate, by the standard argument in Morse theory,  
we know that $Y_q$ is obtained from $\{f_{q,z}(y)\le M\}$ by attaching 
$N$ $n$-handles and so 
$R_{\Z,(q,z)}^\vee$ is a free abelian group of rank $N$. 
If $W_{q}(y)$ has a critical point $y_0$ 
of multiplicity $\mu_0>1$, one can find\footnote{
We can find $\tilde{f}_{q,z}$ in the following way: 
Let $\rho\colon \R_{\ge 0} \to [0,1]$ be a $C^\infty$-function 
such that $\rho(r)=1$ for $0\le r\le 1/2$ and $\rho(r)=0$ for $r\ge 1$. 
Let $U_0$ be an $\epsilon$-neighborhood of $y_0$ 
(in the above K\"{a}hler metric) 
which does not contain other critical points. 
Let $t=(t_1,\dots,t_n)$ be co-ordinates given by $y_i=y_{0,i} e^{t_i}$. 
For $a=(a_1,\dots,a_n)\in \C^n$, put 
$f^a_{q,z}(y) = f_{q,z}(y) +  \rho(|t|/\epsilon)\Re(at)$. 
Then for a generic, sufficiently small $a$, $\tilde{f}_{q,z}=f^a_{q,z}$ 
satisfies the conditions above (here,  
new critical points are all in $|t|<\epsilon/2$).} 
a small $C^\infty$-perturbation $\tilde{f}_{q,z}$ of 
$f_{q,z}$ on a small neighborhood $U_0$ of $y_0$ such that 
$\tilde{f}_{q,z}$ has just $\mu_0$ 
non-degenerate critical points in $U_0$ 
with Morse index $n$. 
By considering such a perturbation and 
Morse theory for $f_{q,z}$ in families (parametrized by $q$ and $z$), 
we obtain the following. 

\begin{proposition}
The relative homology groups 
$R^\vee_{\Z,(q,z)}$ in (\ref{eq:relative_hom}) 
form a local system of rank $|\bN_{\rm tor}|\times 
n! \Vol(\hat{S})$ over $\cMo\times \C^*$. 
\end{proposition} 

When all the critical points $\crit_1,\dots,\crit_N$ 
of $W_q\colon Y_q \to \C$ are non-degenerate, 
a basis of the local system $R^\vee_\Z$ is given by a set of 
{\it Lefschetz thimbles} $\Gamma_1,\dots, \Gamma_N$:  
the image of $\Gamma_i$ under $W_q/z$ is given by 
a curve $\gamma_i:[0,\infty) \to \C$ such that 
$\gamma(0)=W_q(\crit_i)/z$, that 
$\Re\gamma_i(t)$ decreases monotonically 
to $-\infty$ as $t\to \infty$ 
and that $\gamma_i$ does not pass through critical values 
other than $W_q(\crit_i)/z$;   
$\Gamma_i$ is the union of cycles in $W_q^{-1}(z\gamma_i(t))$ 
collapsing to $\crit_i$ along the path $\gamma_i(t)$ as $t\to 0$.  
When the imaginary parts 
$\Im(W_q(\crit_1)/z),\dots,\Im(W_q(\crit_N)/z)$ 
are mutually different, $\Gamma_i$ can be taken to be 
the union of downward gradient flowlines 
of $f_{q,z}(y)$ emanating from $\crit_i$. 
(Note that the gradient 
flow of $f_{q,z}=\Re(W_q/z)$ with respect to a K\"{a}hler metric 
coincides with the Hamiltonian flow  
generated by $\Im(W_q/z)$.) 
Then $\gamma_i$ becomes a half-line parallel to the real axis. 
The intersection pairing defines a unimodular pairing:
\begin{equation}
\label{eq:pairing_Rvee}
R_{\Z,(q,-z)}^\vee \times R_{\Z,(q,z)}^\vee \to\Z.  
\end{equation} 
Let $R_\Z \to \cMo \times \C^*$ be the local system dual to $R_\Z^\vee$ 
and $\cR = R_\Z \otimes \cO_{\cMo\times \C^*}$ be the associated 
locally free sheaf on $\cMo\times \C^*$. 
This inherits from the local system  
a Gau\ss-Manin connection $\hatnabla$ and a pairing 
$((-)^*\cR) \otimes \cR \to \cO_{\cMo\times \C^*}$. 

Let $\omega_1$ be the following holomorphic volume form 
on $Y_1 = \Hom(\bN,\C^*)$: 
\[
\omega_1= 
\frac{1}{|\bN_{\rm tor}|}
\frac{dy_1\cdots dy_n}{y_1\cdots y_n} \quad 
\text{on each connected component.}
\]
This is characterized as a unique translation-invariant 
holomorphic $n$-form $\omega_1$ satisfying 
$\int_{\Hom(\bN,S^1)} \omega_1 = (2\pi \iu)^n$. 
By translation, $\omega_1$ defines a holomorphic volume form 
$\omega_q$ on each fiber $Y_q$. 
Let $\pr\colon \Yo\to \cMo$ be the restriction of the family 
$\pr\colon Y\to \cM$ to $\cMo$. 
Consider a relative holomorphic $n$-form of 
$\Yo\times \C^* \to \cMo \times \C^*$ 
of the form 
\begin{equation}
\label{eq:relative_nform}
\varphi = f(q,z,y) e^{W_q(y)/z} \omega_q, \quad 
f(q,z,y) \in \cO_{\cMo \times \C^*}[y_1^\pm,\dots,y_n^\pm]   
\end{equation} 
where $\cO_{\cMo\times \C^*}$ is the analytic structure sheaf. 
This relative $n$-form 
gives a holomorphic section $[\varphi]$ of $\cR$ via the integration 
over Lefschetz thimbles: 
\begin{equation}
\label{eq:pairing_R_Rvee}
\pair{[\varphi]}{\Gamma} = \frac{1}{(-2\pi z)^{n/2}} \int_{\Gamma} 
f(q,z,y) e^{W_q(y)/z} \omega_q
\in \cO_{\cMo \times \C^*}  
\end{equation} 
The convergence of this integral is ensured by the fact that 
$f(q,z,y)$ has at most polynomial growth in $y$ 
and that $\Re(W_q(y)/z)$ goes to $-\infty$ at the end of $\Gamma$. 
More technically, as done in \cite{pham_GM}, 
one may prove the convergence of the integral 
by replacing the end of $\Gamma$ with a semi-algebraic chain.  

Let $\cR'$ be the $\cO_{\cMo \times \C^*}$-submodule of $\cR$ 
consisting of the sections which locally 
arise from relative $n$-forms $\varphi$ 
of the form (\ref{eq:relative_nform}). 
The Gau\ss-Manin connection on $\cR$ preserves the subsheaf $\cR'$. 
In fact, we have  
\begin{align*}
\hatnabla_{a}[\varphi] &= 
[(\partial_a f+ \frac{1}{z} (\partial_aW_q) f) 
e^{W_q/z} \omega_q],  \\
\hatnabla_{z\partial_z}[\varphi] &= 
[(z\partial_z f - \frac{1}{z} W_q f -\frac{n}{2}f ) 
e^{W_q/z} \omega_q], 
\end{align*} 
where $\varphi$ is given in (\ref{eq:relative_nform}) 
and $\partial_a = q_a (\partial/\partial q_a)$. 
Take a generic $q\in \cMo$ such that all the critical points 
of $W_q(y)$ are non-degenerate. 
Let $\Gamma_1,\dots,\Gamma_N$ be Lefschetz thimbles of $W_q(y)/z$ 
corresponding to critical points $\crit_1,\dots,\crit_N$.  
Then we have the following asymptotic expansion 
as $z \to 0$ with $\arg(z)$ fixed: 
\begin{equation}
\label{eq:asymptotic_exp_in_z} 
\frac{1}{(-2\pi z)^{n/2}} 
\int_{\Gamma_i} f(q,z,y) e^{W_q(y)/z} \omega_q 
\sim \frac{1}{|\bN_{\rm tor}|} 
\frac{e^{W_q(\crit_i)/z} }{\sqrt{\Hess(W_q)(\crit_i)}} 
(f(q,0,\crit_i) +O(z)) 
\end{equation} 
where $f(q,z,y) \in \cO_{\cMo\times \C}[y_1^\pm,\dots,y_n^\pm]$ 
is regular at $z=0$ and 
$\Hess(W_q)$ is the Hessian of $W_q$ calculated in 
co-ordinates $\log y_1,\dots,\log y_n$. 
Let $J(W_q)$ be the Jacobi ring of $W_q$: 
\[
J(W_q):= \C[Y_q]\Big/ \langle 
\parfrac{W_q}{y_1},\dots, \parfrac{W_q}{y_n} \rangle,   
\]
where $\C[Y_q]$ is the co-ordinate ring of $Y_q$. 
Let $\phi_i(y)\in \C[Y_q]$ be a function 
which represents a delta-function supported on 
$\crit_i$ in the Jacobi ring $J(W_q)$. 
Put $\varphi_i = \phi_i(y) e^{W_q/z} \omega_q$.  
By the asymptotics of $\pair{[\varphi_i]}{\Gamma_j}$,  
we know that $[\varphi_1],\dots,[\varphi_N]$ form a basis of $\cR$ for 
sufficiently small $|z|>0$. 
Since $\cR'$ is preserved by the Gau\ss-Manin connection, 
we have $\cR=\cR'$ on the whole $\cMo\times \C^*$. 
In other words, $\cR$ is generated by 
relative $n$-forms of the form (\ref{eq:relative_nform}). 
Let $\Gamma_1^\vee,\dots,\Gamma_N^\vee$ be 
Lefschetz thimbles of $W_q/(-z)$    
dual to $\Gamma_1,\dots,\Gamma_N$ with respect to 
the intersection pairing (\ref{eq:pairing_Rvee}).  
Then the pairing on $\cR$ can be written as   
\begin{equation}
\label{eq:B-model_pairing}
([\varphi(-z)],[\varphi'(z)])_{\cR} 
= \frac{1}{(2\pi \iu z )^n}\sum_{i=1}^N  
\int_{\Gamma_i^\vee} \varphi(-z)  \cdot \int_{\Gamma_i} 
\varphi'(z).  
\end{equation} 
We define an extension $\cRz$ of $\cR$ to $\cMo \times \C$ 
as follows: a section of $\cR$ on an open set 
$U\times \{0<|z|<\epsilon\}$ 
is defined to be extendible to $z=0$ if it is the image of 
a relative $n$-form $s$ of the form (\ref{eq:relative_nform}) such that 
$f(q,z,y)$ in (\ref{eq:relative_nform}) is regular at $z=0$. 
When $[\varphi]$ and $[\varphi']$ are extendible to $z=0$, 
we have from (\ref{eq:B-model_pairing}) 
and (\ref{eq:asymptotic_exp_in_z}) 
\[
([\varphi],[\varphi'])_{\cR} 
\sim \frac{1}{|\bN_{\rm tor}|^2}
\sum_{i=1}^N \frac{f(q,0,\crit_i) f'(q,0,\crit_i)}{\Hess W_q(\crit_i)} 
+ O(z) 
\]
where we put $\varphi = f(q,z,y) e^{W_q(y)/z} \omega_q$ and 
$\varphi'=f'(q,z,y) e^{W_q(y)/z} \omega_q$. 
This shows that $([\varphi],[\varphi'])_{\cR}$ is 
regular at $z=0$ and the value at $z=0$ equals the residue 
pairing on $J(W_q)$.  
By continuity, we have at all $q\in \cMo$: 
\[
([\varphi],[\varphi'])_{\cR}|_{z=0} = \frac{1}{|\bN_{\rm tor}|^2}
\Res_{\Yo/\cMo} \left [ 
\frac{f(q,0,y) f'(q,0,y) \frac{dy_1\cdots dy_n}{y_1\cdots y_n}}
{y_1\parfrac{W_q}{y_1},\dots, y_n\parfrac{W_q}{y_n}} 
\right].  
\]
Let $\phi'_1,\dots,\phi'_N$ be an arbitrary basis 
of the Jacobi ring and 
put $s_i:=[\phi'_i(y) e^{W_q(y)/z}\omega_q]$. 
Then the Gram matrix $(s_i,s_j)_{\cR}$ is non-degenerate 
in a neighborhood of $z=0$ 
since the residue pairing is non-degenerate. 
This implies that $s_1,\dots,s_N$ 
form a local basis of $\cRz$ around $z=0$. 
Summarizing, 
\begin{proposition}[{\cite[Lemma 2.21]{CIT:I}}]
The $\cO_{\cMo\times \C^*}$-module $\cR$ 
is generated by relative $n$-forms 
of the form (\ref{eq:relative_nform}). The extension $\cRz$ of $\cR$ 
to $\cMo\times \C$ is locally free and the pairing on 
$\cR$ extends to a non-degenerate pairing 
$((-)^*\cRz) \otimes \cRz \to \cO_{\cMo\times \C}$.   
\end{proposition} 
Sabbah \cite{sabbah-hypergeometric} gave 
a different construction of $\cRz$ 
based on the algebraic Gau\ss-Manin system 
and the Fourier-Laplace transformation. 
The corresponding results were shown  
in \cite[Corollary 10.2]{sabbah-hypergeometric}.

We introduce the \emph{Euler vector field} $E$  
on $\cMo$ by 
\[
E := \pr_*\left(\sum_{i=1}^m w_i\parfrac{}{w_i}\right)  
= \sum_{a=1}^r \rho_a q_a \parfrac{}{q_a}, \quad 
\rho_a = \sum_{i=1}^m m_{ia}. 
\]
The grading operator $\Grading$ 
acting on sections of $\cRz$ is defined by 
\begin{align}
\label{eq:B-model_grading}
\Grading[\varphi] 
&= 2\left[\left(z\parfrac{f}{z} + \sum_{i=1}^m w_i \parfrac{f}{w_i}\right) 
e^{W/z} \omega \right]  
\end{align} 
for a section $[\varphi]$ of the form (\ref{eq:relative_nform}). 
This grading operator can be written in terms of 
the Gau\ss-Manin connection and the Euler vector field 
(\emph{c.f.} (\ref{eq:hatnabla})):  
\begin{lemma} 
$\Grading = 2(\hatnabla_E + \hatnabla_{z\partial_z} + \frac{n}{2} )$.
\end{lemma} 
\begin{proof}
By the multi-valued splitting of the fibration 
(\ref{eq:fibration_LG}) appearing in the beginning 
of this section, we can regard $E$ as a vector field on $Y$. 
Using the co-ordinate system $(q_a, y_i)$ associated to this splitting, 
we write $\sum_{i=1}^m w_i \parfrac{}{w_i} 
= E + \sum_{i=1}^n c_i y_i\parfrac{}{y_i}$ for some $c_i\in \Q$. 
Because $(\sum_{i=1}^m w_i \parfrac{}{w_i})W = W$, 
we have 
\begin{align*} 
\frac{1}{2} \Grading[\varphi] &= 
\left[\left(\left(z\partial_z + 
\textstyle\sum_{i=1}^m w_i \partial_{w_i}\right) 
(f e^{W/z})\right) \omega\right] \\ 
& = \left(
\hatnabla_{z\partial_z} + \frac{n}{2} + \hatnabla_E \right)[\varphi] 
+ \left[\left(\left(\textstyle\sum_{i=1}^n 
c_i y_i \partial_{y_i}\right) 
(f e^{W/z})\right) \omega \right]. 
\end{align*} 
The second term is zero in cohomology since it is exact. 
\end{proof} 

\begin{definition}
Let $\pi\colon \cMo\times \C \to \cMo$ be the projection. 
The \emph{B-model \seminf VHS of the Landau-Ginzburg model} 
is a locally free $\pi_*\cO_{\cMo\times \C}$-module $\cF:=\pi_*\cRz$ 
endowed with a flat connection   
$\nabla\colon \cF \to z^{-1} \cF\otimes \Omega^1_{\cM}$ 
induced from the Gau\ss-Manin connection $\hatnabla$,  
a pairing $(\cdot,\cdot)_{\cF}\colon 
\cF\times \cF \to \pi_*\cO_{\cM\times \C}$ 
in (\ref{eq:B-model_pairing}) 
induced from the intersection pairing (\ref{eq:pairing_Rvee}) 
and a grading operator $\Grading\colon \cF\to \cF$ 
in (\ref{eq:B-model_grading}).  
The B-model \seminf VHS has a natural integral structure given by 
the local system $R_\Z$ of relative cohomology groups 
on $\cMo\times \C^*$.   
\end{definition} 

\subsection{Mirror symmetry for toric orbifolds} 
\label{subsec:mirrorsym_toric} 
We explain a version of mirror symmetry conjecture 
for weak Fano toric orbifolds, 
which we assume in the rest of the paper. 
This has been proved for weak Fano toric manifolds 
\cite{givental-mirrorthm-toric} and 
weighted projective spaces \cite{CCLT:wp}. 
A general case for toric orbifolds will be proved in 
\cite{CCIT:toric}. 

Mirror symmetry roughly states that 
the B-model \seminf VHS in the previous section is isomorphic to 
the A-model \seminf VHS restricted to a certain subspace (basically 
$H^{\le 2}_{\rm orb}(\cX)$) of $H^*_{\rm orb}(\cX)$ 
under a suitable identification of the base space. 
This isomorphism is given by a so called {\it $I$-function}  
which is a cohomology-valued function on an open domain of 
the base space $\cMo$ of the B-model \seminf VHS. 
\begin{definition}[\cite{CCIT:toric}] 
\label{def:I-funct} 
The $I$-function 
for a projective toric orbifold $\cX$ is 
a cohomology-valued power series on $\cM$ defined by 
\[
I(q,z) = e^{\sum_{a=1}^{r} \ov{p}_a \log q_a/z} 
\sum_{d \in \K_{\rm eff}} q^d 
\frac{\prod_{i; \pair{D_i}{d}<0} \prod_{\pair{D_i}{d}\le \nu <0} 
(\ov{D}_i + (\pair{D_i}{d}-\nu) z) }
{\prod_{i; \pair{D_i}{d}>0} \prod_{0\le \nu< \pair{D_i}{d}}
(\ov{D}_i + (\pair{D_i}{d}-\nu) z) }
\unit_{v(d)} 
\]
where $q^d = q_1^{\pair{p_1}{d}}\dots q_r^{\pair{p_r}{d}}$ 
for $d\in \bL\otimes \Q$, the index $\nu$ moves in $\Z$. 
Recall that $\ov{p}_a$ and $\ov{D}_j$ are images of 
$p_a$ and $D_j$ under the projection 
$\bL^\vee\otimes \Q \to H^2(\cX,\Q)$.   
Note that $\ov{p}_a = 0$ for $a>r'$, 
$\ov{D}_j =0$ for $j>m'$ and 
$\pair{p_a}{d}\ge 0$ for $d\in \K_{\rm eff}$. 
\end{definition} 
The following lemma is easy to check.  
\begin{lemma}
\label{lem:convergence} 
The $I$-function is a convergent power series in $q_1,\dots,q_r$ 
if and only if 
$\hrho$ is in the closure $\cl(\tilde{C}_\cX)$ 
of the extended K\"{a}hler cone.  
In this case, the $I$-function has the asymptotics
\[
I(q,z) = 1 + \frac{\tau(q)}{z} + o(z^{-1})
\]
where $\tau(q)$ is the multi-valued function 
taking values in $H^{\le 2}_{\rm orb}(\cX)$.  
\end{lemma}

To state mirror symmetry conjecture,  
we need to assume that 
$\hrho\in \cl(\tilde{C}_\cX)$. 
See Section \ref{subsubsec:weakFano} for this condition. 
The $I$-function is multi-valued  
and the fundamental group $\pi_1(\cM) \cong\bL^\vee$ 
acts on it by monodromy transformations. 
Take a loop $t\mapsto e^{-2\pi\iu \xi t} q 
= (e^{-2\pi\iu\xi_1 t} q_1, \dots, e^{-2\pi\iu\xi_r t} q_r)$ 
for $\xi = \sum_{a=1}^r \xi_a p_a\in \bL^\vee$.  
The monodromy of $I(q,z)$ along this loop is given by  
\[
I(e^{-2\pi\iu \xi}q,z) = G^\cH([\xi]) I(q,z)
\] 
where $G^\cH([\xi])$ is the Galois action (\ref{eq:GaloisH}) 
corresponding to $[\xi] \in \bL^\vee/\sum_{j>m'} \Z D_j 
\cong H^2(\cX,\Z)$.   
Therefore, we have 
\[
\tau(e^{-2\pi\iu\xi}q) = G([\xi]) \tau(q)  
\]
where $\tau(q)$ is a function in Lemma \ref{lem:convergence} and 
$G([\xi])$ is given in (\ref{eq:Galois}).  
This shows that $\tau(q)$ induces a single-valued map 
\[
\tau \colon 
\text{(neighborhood of $q=0$ in $\cM$)} 
\longrightarrow H^{\le 2}_{\rm orb}(\cX)/H^2(\cX,\Z).   
\]
Let $\cF_{\rm B}\to \cMo$ be the B-model \seminf VHS associated with 
Landau-Ginzburg model mirror to $\cX$ and 
$\cF_{\rm A} = (\tcF_{\rm A} \to U)/H^2(\cX,\Z)$ 
be the A-model \seminf VHS of $\cX$, where $U$ is a suitable open domain 
in $H^*_{\rm orb}(\cX)$ (see Section \ref{sec:A-model}). 
\begin{conjecture} 
\label{conj:toricmirror} 
Assume that our toric data satisfy $\hrho\in \cl(\tilde{C}_\cX)$. 
There exists an isomorphism of graded \seminf VHS 
$\Mir \colon \cF_{\rm B} \cong \tau^* \cF_{\rm A}$. 
This isomorphism sends the section 
$[e^{W_q/z}\omega_q]$ of $\cF_{\rm B}$ 
to the $I$-function $I(q,z)\in \cH^\cX$, i.e. 
\[
\cJ_{\tau(q)} (\Mir[e^{W_q/z} \omega_q])  = I(q,z)  
\]
where $\cJ_{\tau(q)} \colon \tcF_{{\rm A},\tau(q)} \to \cH^\cX$ 
is the embedding (\ref{eq:fundamentalsol_A}) given by the 
fundamental solution.    
\end{conjecture}

\begin{remark} 
\label{rem:aftermirrorconj} 
(i) We have $\rank \cF_{\rm A}=\dim H^*_{\rm orb}(\cX)$ 
and $\rank \cF_{\rm B}=|\bN_{\rm tor}| \times n! \Vol(\hat{S})$.  
These two numbers match if and only if $\hrho\in \cl(\tilde{C}_\cX)$. 
 
(ii) The map $\tau$ takes the form 
\[
\tau(q) = \sum_{a=1}^{r'} (\log q_a) p_a +
 \sum_{j=m'+1}^m q^{D^\vee_j} \frD_j 
+ \text{ higher terms}.    
\]
Thus $\tau$ is a local embedding (isomorphism) 
near $q=0$ 
if $p_1,\dots,p_{r'},\frD_{m'+1},\dots,\frD_m$ 
are linearly independent (resp. basis of $H^{\le 2}_{\rm orb}(\cX)$).   
See (\ref{eq:frD}) for $\frD_j$.  

(iii) Because of the asymptotic of the $I$-function 
in Lemma \ref{lem:convergence}, the conjecture 
further implies that $\Mir[e^{W_q/z} \omega_q]$ 
is the unit section $\unit$ of the A-model \seminf VHS 
$\cF_{\rm A}$. 
This, however, fails to hold for non-weak-Fano case 
\cite{iritani-genmir}.  
\end{remark} 
As in the previous section, we denote by $\cR$ 
the locally free $\cO_{\cMo\times \C^*}$-module 
associated to the local system $R$ of the Landau-Ginzburg model.   
We regard $[e^{W_q/z}\omega_q]$ as a section of $\cR$. 
Via the identification (\ref{eq:solutionmap_z}) 
of $\cV^\cX = H^*_{\rm orb}(\cX)$ 
with the quantum cohomology local system, 
the mirror map $\Mir$ induces an isomorphism
\begin{equation}
\label{eq:mirror_isom}
\Mir_{(q,z)} \colon \cR_{(q,z)} \rightarrow \cV^\cX \quad \text{s.t. }
\Mir_{(q,z)}([e^{W_q/z} \omega_q]) = z^{-\rho} z^\mu I(q,z),  
\end{equation} 
at $(q,z)\in (\cMo\times \C^*)\sptilde$. 
We will compute integral linear co-ordinates on $\cV^\cX$ 
corresponding to Lefschetz thimbles $\Gamma_k \in R^\vee_{\Z}$
through this map $\Mir_{(q,z)}$.

Now we can state the main theorem. 

\begin{theorem}
\label{thm:pulledbackintstr} 
Let $\cX$ be a weak Fano toric orbifold defined by 
initial data satisfying $\hrho\in \cl(\tilde{C}_\cX)$. 
Assume that Conjecture \ref{conj:toricmirror} and (A3) in Section 
\ref{subsec:Amodel_intstr} hold for $\cX$. 
Then via the mirror isomorphism, the integral structure in 
Landau-Ginzburg B-model induces 
the $\hGamma$-integral structure of $\cX$ 
in Definition-Proposition \ref{def-prop:A-model_int}.  
\end{theorem} 
The next section is devoted to the proof of this theorem. 
\begin{remark}
Since the $\hGamma$-integral structure 
is defined to be the image of the $K$-group, 
we can identify 
the integral lattice $R_{\Z,(q,z)}^\vee$ 
generated by Lefschetz thimbles 
with (the dual of) the $K$-group $K(\cX)$ by 
Theorem \ref{thm:pulledbackintstr}. 
This correspondence also identifies 
the intersection numbers 
of vanishing cycles  
with the Mukai pairing on the $K$-group. 
On the other hand, by mirror symmetry again, 
the Stokes matrices of the quantum differential equations 
at the irregular singular point $z=0$ are known to 
arise as the intersection numbers of vanishing cycles. 
In particular, there exist 
$V_1,\dots,V_N\in K(\cX)$ 
(which correspond to a set of Lefschetz thimbles) 
such that the Stokes matrix is 
given by a matrix of the Mukai pairing 
$(V_i,V_j)_{K(\cX)}$. 
Dubrovin's conjecture \cite{dubrovin-analyticth} 
says that $V_1,\dots,V_N$ here 
should come from an exceptional collection in the derived category. 
This should follow from homological mirror symmetry 
for toric orbifolds. 
Several versions of homological mirror symmetry for toric 
manifolds were proved by Abouzaid \cite{abouzaid}, 
Fang-Liu-Treumann-Zaslow \cite{FLTZ}. 
\end{remark} 

\subsection{Oscillatory integrals}
The proof of Theorem \ref{thm:pulledbackintstr} 
is based on a calculation of oscillatory integrals  
$\int_{\Gamma_0} e^{W_q/z} \omega_q$ for some special 
Lefschetz thimble $\Gamma_0$. 
We will make use of Givental's \emph{equivariant mirror}  
which gives a perturbation of oscillatory integrals. 
This is considered to be the mirror of 
the equivariant quantum cohomology of toric orbifolds,  
although we do not give a precise formulation of 
equivariant mirror symmetry. 

\subsubsection{Equivariant mirror} 
Let $T:=(\C^*)^m$ act on our toric orbifold $\cX=\C^m/\!/\T$ 
via the diagonal action of $(\C^*)^m$ on $\C^m$. 
Let $-\lambda_1,\dots,-\lambda_m$ be the equivariant variables 
corresponding to generators of $H^*_{T}(\pt)$. 
We will regard $\lambda_i$ either as a cohomology class or 
as a complex number depending on the context. 
Givental's equivariant mirror \cite{givental-mirrorthm-toric} 
is given by the following perturbed potential $W^{\lambda}$: 
\[
W^{\lambda} := \sum_{i=1}^m (w_i + \lambda_i \log w_i) 
= W+ \sum_{i=1}^m \lambda_i \log w_i  
\] 
Here $\lambda_i$ denotes a complex number. 
This is a multi-valued function on each fiber $Y_q$. 
Morse theory for $\Re(W^\lambda(y)/z)$ will compute 
relative homology with coefficients in some local system. 
For a cycle $\Gamma\subset Y_q$ in such a relative homology, 
we can define the \emph{equivariant oscillatory integral}: 
\[
\int_{\Gamma} e^{W^{\lambda}/z} \omega_q 
= \int_{\Gamma} e^{W/z} \prod_{i=1}^m w_i^{\lambda_i/z} \omega_q.  
\]
For our purpose, it is more convenient to 
use the exponent $\lambda_i/(2\pi\iu)$ instead of  $\lambda_i/z$.  
Define 
\begin{equation}
\label{eq:equiv_osc_int}
\cI^\lambda_\Gamma(q,z) := 
\frac{1}{(-2\pi z)^{n/2}} 
\int_{\Gamma} e^{\frac{w_1+\dots+w_m}{z}} 
\prod_{i=1}^m w_i^{\frac{\lambda_i}{2\pi\iu}} \omega_q. 
\end{equation} 
Consider the fibration formed by real points on 
(\ref{eq:fibration_LG}):  
\[
\begin{CD}
\unit @>>> \Hom(\bN,\R_{>0}) @>>> Y_\R := \R_{>0}^m 
@>{\pr|_{Y_\R}}>> \cM_\R :=\Hom(\bL,\R_{>0}) @>>> \unit. 
\end{CD}
\]
Here we regard $\R_{>0}$ as an abelian group 
with respect to the multiplication.  
Note that this exact sequence splits and that 
the section given by the matrix $(\ell_{ia})$ 
in Section \ref{subsec:LGmodel} is single-valued when restricted  
to this real locus. 
Take a point $q\in \cMo_\R=\cM_\R\cap \cMo$. 
Let $\Gamma_0:=Y_q\cap Y_\R \cong \Hom(\bN,\R_{>0})$ 
be a cycle formed by real points in the fiber $Y_q$. 
In co-ordinates $y_1,\dots,y_n$ in Section \ref{subsec:LGmodel}, 
we have $\Gamma_0 = \{(y_1,\dots,y_n)\in Y_q\;;\; y_i\in \R_{>0}\}$. 
Then the integral $\cI^\lambda_{\Gamma_0}(q,z)$ is well-defined  
when $q\in \cMo_\R$ and $\Re(z)<0$.

\subsubsection{$H$-function}
It is convenient to introduce another cohomology-valued 
hypergeometric function $H(q,z)$, which is related to 
$I(q,z)$ by a $q$-independent linear transformation. 
This has been used by Horja \cite{horja1},  
Hosono \cite{hosono} and Borisov-Horja \cite{borisov-horja-FM}
in the context of homological mirror symmetry. 
Using Gamma functions, we can write 
\[
I(q,z) = e^{\sum_{a=1}^{r} \ov{p}_a \log q_a/z} 
\sum_{d\in \K_{\rm eff}} 
\frac{q^d}{z^{\pair{\hrho}{d}}} 
\prod_{i=1}^m 
\frac{\Gamma( 1- \{-\pair{D_i}{d}\} + \ov{D}_i/z)}
     {\Gamma( 1+ \pair{D_i}{d}+ \ov{D}_i/z) }
     \frac{\unit_{v(d)}}{z^{\iota_{v(d)}}}. 
\]
Via the identification 
$z^{-\rho}z^\mu \colon \sfH^\cX_z \cong \cV^\cX$ 
in (\ref{eq:solutionmap_z}),  
the $I$-function gives a $\cV^\cX$-valued function 
\[
z^{-\rho} z^\mu I(q,z) = z^{-n/2} \sum_{d\in \K_{\rm eff}}
x^{\ov{p}+d} 
\prod_{i=1}^m 
\frac{\Gamma( 1- \{-\pair{D_i}{d}\} + \ov{D}_i)}
     {\Gamma( 1+ \pair{D_i}{d}+ \ov{D}_i) } 
\unit_{v(d)} 
\]
where we used the notation: 
\[
x^{\ov{p}+d} := e^{\sum_{a=1}^{r} \ov{p}_a \log x_a} 
\prod_{a=1}^r x_a^{\pair{p_a}{d}}, \quad 
x_a := \frac{q_a}{z^{\rho_a}}.   
\]  
We can decompose the map $\Psi\colon K(\cX) \to \cV^\cX$ in 
(\ref{eq:stdintstr_A}) as 
\[
\begin{CD}
\Psi \colon K(\cX) @>{\tch}>> H^*(I\cX) 
@>{(2\pi)^{-n/2} \hGamma_\cX (2\pi\iu)^{\deg/2}\inv^*}>> \cV^\cX. 
\end{CD} 
\]
The $H$-function is defined to be a function 
which takes values in the middle vector space $H^*(I\cX)$ 
and corresponds to $z^{-\rho}z^{\mu} I(q,z)$ via the second map above: 
\begin{align*}
H(q,z) &:= (2\pi)^{\frac{n}{2}} \inv^* (2\pi\iu)^{-\frac{\deg}{2}} 
\hGamma_\cX^{-1} ( z^{-\rho}z^\mu I(q,z)) \\ 
&= (2\pi)^{\frac{n}{2}} z^{-\frac{n}{2}} \sum_{d\in \K_{\rm eff}}
x^{\frac{\ov{p}}{2\pi\iu}+d} 
\frac{\unit_{\inv(v(d))}} 
{\prod_{i=1}^m\Gamma( 1+ \pair{D_i}{d}+ \frac{\ov{D}_i}{2\pi \iu}) }.  
\end{align*} 
Here we used that $\hGamma_{\cX}^{-1}$ cancels exactly with the 
numerator in the formula of $z^{-\rho}z^\mu I(q,z)$.  

The $I$-function and the $H$-function admit 
equivariant generalizations. 
An element $\xi\in \bL^\vee$ defines a 
$T$-equivariant orbifold line bundle $L_\xi$ on $\cX$:  
\[
L_\xi = \cU_\eta \times \C\Big/
(z_1,\dots,z_m,c) 
\sim (t^{D_1}z_1,\dots,t^{D_m}z_m,t^\xi c), \ t\in \T  
\]
where $T=(\C^*)^m$ acts on $L_\xi$ by the diagonal action 
on the first factor and the trivial action on the second factor. 
By taking the equivariant first Chern class, 
we can identify $\xi\in \bL^\vee$ with $c_1^{T}(L_\xi)$ of 
$H^2_T(\cX)$. 
By abuse of notation, 
we denote by $\ov{p}_1,\dots,\ov{p}_r\in H_T^2(\cX)$ 
the \emph{$T$-equivariant} cohomology classes 
corresponding to $p_1,\dots,p_r\in \bL^\vee$. 
Note that $\ov{p}_{r'+1},\dots,\ov{p}_r$ are 
non-zero only in equivariant cohomology. 
Similarly, we denote by $\ov{D}_i\in H_T^2(\cX)$ 
the $T$-equivariant Poincar\'{e} dual of the toric divisor $\{z_i=0\}$. 
Note that $\ov{D}_j=0$ for $j>m'$ even in equivariant cohomology 
(since $\{z_j=0\}$ is empty). 
Then we have (\emph{c.f.} Equation (\ref{eq:Dprel})) 
\begin{equation}
\label{eq:Dprel_equiv}
\ov{D}_i = \sum_{a=1}^r m_{ia} \ov{p}_a - \lambda_i  \quad 
\text{in $H_T^2(\cX)$.}  
\end{equation} 
The equivariant $I$-function is defined by the same formula 
in Definition \ref{def:I-funct} with all the appearance of 
$\ov{p}_a$ and $\ov{D}_j$ replaced by their equivariant counterparts. 
The equivariant $H$-function 
$H^\lambda(q,z)$ is a slight modification 
of $H(q,z)$: 
\begin{equation}
\label{eq:equiv_H}
H^{\lambda}(q,z) := (2\pi)^{\frac{n}{2}}  
z^{-\frac{n}{2}+\frac{\lambda_1+\dots+\lambda_m}{2\pi\iu}} 
\sum_{d\in \K_{\rm eff}} x^{\frac{\ov{p}}{2\pi\iu}+d} 
\frac{\unit_{\inv(v(d))}}
{\prod_{i=1}^m \Gamma(1+ \pair{D_i}{d}+ \frac{\ov{D}_i}{2\pi\iu})}, 
\end{equation} 
where $\ov{p}_a, \ov{D}_i \in H_T^2(\cX)$. 
The equivariant $I$- and $H$-functions take values in 
$H_{{\rm orb},T}^*(\cX)$ and $H^*_{T}(I\cX)$ respectively 
(here $H_{{\rm orb},T}^*(\cX)$ is $H_{T}^*(I\cX)$ 
with a different grading). 

\subsubsection{Oscillatory integral and $H$-function}
We will show the following: 
\begin{theorem}
\label{thm:connection_cI_H} 
Assume that $\hrho\in \cl(\tilde{C}_\cX)$.  
The equivariant oscillatory integral (\ref{eq:equiv_osc_int}) 
and the equivariant $H$-function (\ref{eq:equiv_H}) are related by
\begin{equation}
\label{eq:connection_cI_H}
\cI_{\Gamma_0}^\lambda (q,z) = 
\int_{I\cX} 
H^\lambda (q,z) \Td^\lambda_\cX, \quad 
q\in \cMo_\R, \ z<0. 
\end{equation} 
where $\Td^\lambda_\cX$ is the $T$-equivariant Todd class. 
The branch of the $H$-function is chosen so that 
$\Im \log z=\pi, \Im \log q_a =0$.  
In the non-equivariant limit, we have 
\begin{equation}
\label{eq:Gamma_0_strsheaf}
\pair{[e^{W_q/z}\omega_q]}{\Gamma_0} 
= (\tch^{-1}H(q,z),\cO_\cX)_{K(\cX)_\C}, 
\end{equation} 
where the left-hand side is the pairing in (\ref{eq:pairing_R_Rvee}) 
and the right-hand side is the complexified Mukai pairing. 
Therefore, under the mirror isomorphism (\ref{eq:mirror_isom}) 
and $\Psi^{-1}\colon \cV^\cX \cong K(\cX)\otimes \C$, 
the real thimble $\Gamma_0\in \cR^\vee_\Z$ corresponds to 
the linear form $(\cdot, \cO_\cX)_{K(\cX)_\C}$ on $K(\cX)$. 
\end{theorem} 

\begin{remark}
(i) Even if $\hrho\notin \cl(\tilde{C}_\cX)$, the left-hand side of 
(\ref{eq:connection_cI_H}) 
makes sense as an analytic function in $q$ and $z$. 
In this case, the right-hand side could be understood as 
the asymptotic expansion in $q_1,\dots,q_r$ 
of the left-hand side in the limit $q_a \searrow +0$.  


(ii) The relation (\ref{eq:connection_cI_H}) gives a connection 
between solutions to ordinary differential equations in $z$. 
Both hand sides satisfy the same differential equations in $z$ 
(see below). 
The oscillatory integral admits an asymptotic expansion 
in $z$ and the $H$-function is by definition a power series in $z^{-1}$ 
(when $\hrho\in \cl(\tilde{C}_\cX)$). 

(iii) This theorem suggests that, under homological mirror symmetry, 
the thimble $\Gamma_0$, 
an object of Fukaya-Seidel category of the Landau-Ginzburg model, 
should correspond to the structure sheaf $\cO_\cX$, an object of 
the derived category of coherent sheaves on $\cX$.  
This correspondence is consistent with the SYZ picture. 
The Lefschetz thimble $\Gamma_0$ 
gives a Lagrangian section of the SYZ fibration, 
so should correspond to the structure sheaf. 
\end{remark}

By the localization theorem \cite{atiyah-bott} 
in equivariant cohomology, 
the inclusion $i\colon I\cX^T \rightarrow I\cX$ induces an isomorphism 
$i^* \colon H^*_T(I\cX)\otimes_{H^*_T(\pt)} \C(\lambda)
\to H^*(I\cX^T)\otimes \C(\lambda)$, 
where $I\cX^T$ is the set of $T$-fixed points in $I\cX$ 
and $\C(\lambda)$ is the fraction ring of 
$H^*_T(\pt)=\C[\lambda_1,\dots,\lambda_m]$. 
For the case of toric orbifolds, the number of fixed points in $I\cX$ 
is equal to $N=\dim H^*_{\rm orb}(\cX)$. 
A $T$-fixed point in $I\cX$ is labeled by a pair $(\sigma, v)$ 
of a fixed point $\sigma \in \cX^T$ and $v\in \Boxop$ 
such that $\sigma \in \cX_v$. 
Note that a fixed point $\sigma\in \cX$ is in one-to-one 
correspondence with a top dimensional cone of the fan $\Sigma$ 
spanned by $\{b_i\;;\; \sigma\in \{z_i=0\}\}$. 
By restricting $H^\lambda(q,z)$ to a fixed point $(\sigma,v)$, 
we get a function $H^\lambda_{\sigma,v}(q,z)$ 
in $q$, $z$ and $\lambda$. 
We call it a \emph{component} of the $H$-function.

\begin{lemma}
The equivariant $H$-function $H^\lambda(q,z)$ and 
the oscillatory integral $\cI^\lambda_{\Gamma_0}(q,z)$ are solutions to 
the following GKZ-type differential equations: 
\begin{align*}
\left[ 
\prod_{\pair{D_i}{d}>0} \prod_{\nu =0}^{\pair{D_i}{d}-1} 
(\hD^\lambda_i - \nu z ) 
- q^d \prod_{\pair{D_i}{d}<0} \prod_{\nu=0}^{-\pair{D_i}{d}-1}
(\hD^\lambda_i - \nu z) 
\right] f &=0, \quad d\in \bL, \\
\left (z\partial_z + \sum_{a=1}^r \rho_a \partial_a -
\frac{\lambda_1+\cdots+\lambda_m}{2\pi\iu} + \frac{n}{2}\right) 
f &=0,  
\end{align*}
where $\hD^\lambda_i= 
z \sum_{a=1}^r m_{ia}\partial_a - z \lambda_i/(2\pi\iu)$ and 
$\partial_a = q_a (\partial/\partial q_a)$. 
Note that $\pair{D_i}{d}\in\Z$ for $d\in \bL$. 
The $N$ components $H^\lambda_{\sigma,v}(q,z)$ of the $H$-function 
form a basis of solutions to these differential equations 
for a generic $\lambda$. 
\end{lemma} 

From this lemma\footnote{For the proof, 
see the revised version \cite{iritani-Int} of this paper.}, 
we know that there exist coefficient functions  
$c_{\sigma,v}(\lambda)$ such that 
\begin{equation}
\label{eq:oscint_H_coeff}
\cI^\lambda_{\Gamma_0} (q,z) = 
\sum_{(\sigma,v)\in I\cX^T} c_{\sigma,v}(\lambda) 
H_{\sigma,v}^\lambda(q,z). 
\end{equation} 
We will determine a holomorphic function 
$c_{\sigma,v}(\lambda)$ in $\lambda$ 
by putting $z=-1$ and studying the asymptotic behavior of 
the both hand sides in the limit $q_a\searrow +0$. 
Take a fixed point $\sigma\in\cX^T$. 
Define $I^\sigma\in \cA$ by 
$I^\sigma=\{i\;;\; \sigma \notin \{z_i=0\}\}$. 
We can take $\{w_j \;;\; j\notin I^\sigma \}$ 
as a co-ordinate system on 
$Y_q\cap Y_\R=\Gamma_0$.  
We can express $w_i$ for $i\in I^\sigma$ in terms of 
$\{w_j\;;\; j\notin I^\sigma\}$ and $q_a$, $a=1,\dots,r$  
by solving (\ref{eq:fibration_LG_formula}). 
Put 
\[
w_i = \prod_{a=1}^r q_a^{\ell_{ia}^\sigma} 
\prod_{j\notin I^\sigma} w_j^{b^{\sigma}_{ij}}, \quad i\in I^\sigma. 
\]
Note that $(\ell_{ia}^\sigma)_{i\in I^\sigma,1\le a\le r}$ 
is the matrix inverse to $(m_{ia})_{i\in I^\sigma, 1\le a\le r}$. 
Because $p_a\in \cl(\tilde{C}_\cX) \subset 
\sum_{i\in I^\sigma} \R_{\ge 0} D_i$, 
it follows that $\ell_{ia}^\sigma \ge 0$. 
It is also easy to see that $b^\sigma_{ij}$ is determined by 
$b_i = \sum_{j\notin I^\sigma} b^\sigma_{ij} b_j$ in $\bN\otimes \R$. 
Let $V(\sigma)$ be $n!|\bN_{\rm tor}|$ times 
the volume of the convex hull of  
$\{b_j\;;\; j\notin I^\sigma\}\cup \{0\}$ in $\bN\otimes \R$.  
Then the holomorphic volume form $\omega_q$ is given by 
\[
\omega_q = \frac{1}{V(\sigma)} 
\prod_{j\notin I^\sigma} \frac{dw_j}{w_j}.  
\]
We set 
\[
\K_{\rm eff,\sigma} := 
\{d\in \bL\otimes \Q \;;\; \pair{D_i}{d}\in \Z_{\ge 0},  
\forall i\in I^\sigma\} 
= \bigoplus_{i\in I^\sigma} \Z_{\ge 0} \ell_i^\sigma.    
\]
Here, $\ell_i^\sigma\in \bL\otimes \Q$ is defined by 
$\pair{p_a}{\ell_i^\sigma} = \ell_{ia}^\sigma$. 
Then we have 
$\K_{\rm eff} = \bigcup_{\sigma\in \cX^T} \K_{\rm eff,\sigma}$. 
We denote by $\ov{p}_a(\sigma)$ and $\ov{D}_j(\sigma)$  
the restrictions of $\ov{p}_a, \ov{D}_j \in H^*_T(\cX)$ 
to the fixed point $\sigma$. 
By using $\ov{D}_i(\sigma) =0$ for $i\in I^\sigma$ and 
(\ref{eq:Dprel_equiv}), we calculate 
\begin{equation}
\label{eq:ovDjsigma}
\ov{p}_a(\sigma) = 
\sum_{i\in I^\sigma} \lambda_i \ell_{ia}^\sigma, \quad 
\ov{D}_j(\sigma) = 
-\lambda_j - \sum_{i\in I^\sigma} \lambda_i b_{ij}^\sigma, \quad 
j\notin I^\sigma.  
\end{equation} 
For a function $f(q_1,\dots,q_r)$ in $(q_1,\dots,q_r)\in (\R_{>0})^r$, 
we write $f(q_1,\dots,q_r) =O(M)$ for $M\in \R$ 
when $f(tq_1,\dots,tq_r) =O(t^M)$ as $t\searrow +0$. 
\begin{lemma}
\label{lem:expansion_oscint} 
Let $\sigma$ be a fixed point in $\cX$. 
For any $M>0$, there exists $M'>0$ such that the following holds. 
For $\lambda_1,\dots,\lambda_m$ such that 
$\Re(-\frac{\ov{D}_j(\sigma)}{2\pi\iu}) >M'$ 
for all $j\notin I^\sigma$, $\cI^\lambda_{\Gamma_0}(q,-1)$ with 
$(q_1,\dots,q_r)\in (\R_{>0})^r$ 
has the expansion  
\begin{align*} 
\cI^\lambda_{\Gamma_0}&(q,-1) =
\frac{(2\pi)^{n/2}e^{(\lambda_1+\dots+\lambda_m)/2}}
{\iu^n V(\sigma)} 
(e^{-\pi\iu \hrho}q)^{\frac{\ov{p}(\sigma)}{2\pi\iu}}\times \\ 
 & \left(
\sum_{\substack{d\in \K_{\rm eff,\sigma}, \\ |d|<M}} 
\frac{(e^{-\pi\iu \hrho} q)^d} 
{\prod_{j\notin I^\sigma}
(1-e^{-2\pi\iu\pair{D_j}{d}-\ov{D}_j(\sigma)}) 
\prod_{i=1}^m 
\Gamma(1+\pair{D_i}{d} + \frac{\ov{D}_i(\sigma)}{2\pi\iu}) }  
+ O(M) \right). 
\end{align*}
where $|d| = \sum_{a=1}^r \pair{p_a}{d}$ and we set 
\begin{align*}
(e^{-\pi\iu\hrho}q)^{\frac{\ov{p}(\sigma)}{2\pi\iu}} &:= 
\prod_{a=1}^r 
(e^{-\pi\iu\rho_a}q_a)^{\frac{\ov{p}_a(\sigma)}{2\pi\iu}}, 
\quad 
(e^{-\pi\iu\hrho}q)^d :
= \prod_{a=1}^r (e^{-\pi\iu\rho_a}q_a)^{\pair{p_a}{d}}. 
\end{align*}
\end{lemma} 
\begin{proof}
Using the notation above, we can write 
\[
\cI_{\Gamma_0}^\lambda(q,-1) = 
\frac{q^{\frac{\ov{p}(\sigma)}{2\pi\iu}}}
{(2\pi)^{n/2}V(\sigma)}
\int_{(0,\infty)^n} 
\exp( -\sum_{i\in I^\sigma} q^{\ell_i^\sigma} w_\sigma^{b_i})  
e^{-\sum_{j\notin I^\sigma} w_j}
w_\sigma^{-\frac{\ov{D}(\sigma)}{2\pi\iu}}  
\frac{dw_\sigma}{w_\sigma}. 
\]
where we put 
$w_\sigma^{b_i} := \prod_{j\notin I^\sigma} w_j^{b^\sigma_{ij}}$, 
$w_\sigma^{-\frac{\ov{D}(\sigma)}{2\pi\iu}} := \prod_{j\notin I^\sigma} 
w_j^{-\frac{\ov{D}_j(\sigma)}{2\pi\iu}}$ 
and $dw_\sigma/w_\sigma := \prod_{j\notin I^\sigma} (dw_j/w_j)$. 
Consider the Taylor expansion: 
\[
\exp(-\sum_{i\in I^\sigma} q^{\ell_i^\sigma} w_\sigma^{b_i}) 
= \sum_{\substack{n_i\ge 0\;;\; i\in I^\sigma, \\ 
|\sum_{i\in I^\sigma} n_i \ell_i^\sigma |<M} } 
\frac{\prod_{i\in I^\sigma} (-1)^{n_i}q^{ n_i\ell_i^\sigma}
w_\sigma^{n_ib_i}}{\prod_{i\in I^\sigma} n_i!}  
+O(M). 
\] 
When $\Re(-\frac{\ov{D}_j(\sigma)}{2\pi\iu})$ is sufficiently big 
for all $j\notin I^\sigma$, each term in the right-hand side 
is integrable for the measure $e^{-\sum_{j\notin I^\sigma} w_j}
w_\sigma^{-\frac{\ov{D}(\sigma)}{2\pi\iu}}  
(dw_\sigma/w_\sigma)$ on $(0,\infty)^n$. 
Therefore, we calculate 
\[
\cI_{\Gamma_0}^\lambda(q,-1) = \frac{q^{\frac{\ov{p}(\sigma)}{2\pi\iu}}}
{(2\pi)^{n/2}V(\sigma)} 
\left( 
\sum_{\substack{d\in \K_{\rm eff,\sigma},\\ |d|<M} }
\frac{(-1)^{\sum_{i\in I^\sigma} n_i} q^d}
{\prod_{i\in I^\sigma}n_i!} 
\prod_{j\notin I^\sigma} 
\Gamma\left
(\sum_{i\in I^\sigma} n_i b_{ij}^\sigma -\tfrac{\ov{D}_j(\sigma)}{2\pi\iu}
\right) +O(M)\right), 
\]
where $d=\sum_{i\in I^\sigma} n_i \ell_i^\sigma$. 
Using $n_i = \pair{D_i}{d}$, $\sum_{i\in I^\sigma} n_i b_{ij}^\sigma 
= -\pair{D_j}{d}$ and $\Gamma(z)\Gamma(1-z)= \pi /\sin (\pi z)$, 
we arrive at the formula in the lemma. 
\end{proof} 

Next we study the asymptotic behavior of 
$H^\lambda_{\sigma,v}(q,-1)$ in the limit $q\searrow +0$. 
We can take the vector 
$\eta\in \bL\otimes \R$ in our initial data of $\cX$ 
to be $p_1+\cdots + p_r \in \tilde{C}_\cX$. 
Recall that $\cX$ can be written as a symplectic quotient 
(\ref{eq:X_symplecticquot}), so is endowed with 
a reduced symplectic form which depends on $\eta$. 
For simplicity, we assume that $\lambda_j$ is purely imaginary. 
Define a Hamiltonian function 
$\mathfrak{h}_{\eta,\lambda}\colon \cX\to \R$ by 
\[
\mathfrak{h}_{\eta,\lambda} (z_1,\dots,z_m) 
= -\sum_{i=1}^m \frac{\lambda_j}{2\pi\iu} |z_j|^2, \quad 
(z_1,\dots,z_m)\in \mathfrak{h}^{-1}(\eta).  
\]
This generates a Hamiltonian $\R$-action $(z_1,\dots,z_m) \mapsto 
(e^{-\lambda_1 s}z_1,\dots,e^{-\lambda_m s} z_m), s\in \R$ on $\cX$. 
In general, the moment map for an $\R$-action 
preserving the complex structure on $\cX$ 
attains its global maximum value at  
every critical point of index $2n = \dim_\R \cX$.  
(This follows from the so-called Mountain-Path Lemma and 
the fact that there are no critical points of odd index. 
See \emph{e.g.} \cite{audin}). 
Because the weights of $T_\sigma \cX$ for this $\R$-action are  
$\{\frac{\ov{D}_j(\sigma)}{2\pi\iu}\;;\; j\notin I^\sigma\}$, it follows 
that 
\begin{equation}
\label{eq:maximum}
\text{$\mathfrak{h}_{\eta,\lambda}$ 
attains its unique maximum value at $\sigma$} 
\Longleftrightarrow 
-\frac{\ov{D}_j(\sigma)}{2\pi\iu} >0, \ \forall j \notin I^\sigma. 
\end{equation} 
For a given $M>0$ and a fixed point $\sigma\in \cX$, 
we can choose $\lambda_1,\dots,\lambda_m \in \iu \R$ 
such that the expansion in Lemma \ref{lem:expansion_oscint} holds. 
(This is possible since one can make 
$\Re(-\frac{\ov{D}_j(\sigma)}{2\pi\iu})$ arbitrarily large. 
See \eqref{eq:ovDjsigma}.) 
Then by (\ref{eq:maximum}), we know that 
$\mathfrak{h}_{\eta,\lambda}(\sigma)
>\mathfrak{h}_{\eta,\lambda}(\sigma')$ for any other 
fixed point $\sigma' \neq \sigma$. 
On the other hand, we can easily see that 
$\mathfrak{h}_{\eta,\lambda}(\sigma) =
 -\sum_{a=1}^r \frac{\ov{p}_a(\sigma)}{2\pi\iu}$.  
Therefore, by rescaling $\lambda_i$ if necessary, 
we can assume that 
\[
\sum_{a=1}^r \frac{\ov{p}_a(\sigma)}{2\pi\iu} +M 
< \sum_{a=1}^r \frac{\ov{p}_a(\sigma')}{2\pi\iu}, 
\quad \forall \sigma'\neq \sigma. 
\]
Then we have the following expansions:  
\begin{align*} 
H_{\tau,v}^\lambda(q,-1) &= 
\frac{(2\pi)^{n/2}e^{(\lambda_1+\cdots+\lambda_m)/2} }{\iu^n} 
(e^{-\pi\iu\hrho} q)^{\frac{\ov{p}(\sigma)}{2\pi\iu}} \\
\times &
\begin{cases}
\displaystyle
\sum_{\substack{d\in \K_{\rm eff,\sigma}; \\ \inv(v(d)) = v, |d|<M}}
\frac{(e^{-\pi\iu\hrho}q)^d}{\prod_{i=1}^m 
\Gamma(1+\pair{D_i}{d} + \frac{\ov{D}_i(\sigma)}{2\pi\iu})}+O(M) \quad 
& \text{if }\tau = \sigma; \\
O(M) & \text{if }\tau\neq \sigma.
\end{cases} 
\end{align*} 
Comparing this with the expansion in Lemma \ref{lem:expansion_oscint}, 
we conclude 
\[
c_{\sigma,v}(\lambda) = 
\frac{1}{V(\sigma) 
\prod_{i\notin I^\sigma} (1-e^{-2\pi\iu f_{v}([D_i])-\ov{D}_i(\sigma)})}, 
\]
where $c_{\sigma,v}$ is the coefficient 
appearing in (\ref{eq:oscint_H_coeff}) 
and $f_v([D_i])\in [0,1)$ is the rational number associated 
to $[D_i]\in H^2(\cX,\Z)$ (see Section \ref{sec:A-model} and 
(\ref{eq:toric_fv})). 
Hence, we find 
\[
c_{\sigma,v}(\lambda) = \frac{1}{V(\sigma)} 
\frac{\Td_{\cX}^\lambda|_{(\sigma,v)} }{e_T(T_\sigma\cX_v)}, 
\]
where $\Td_\cX^\lambda|_{(\sigma,v)}$ 
is the restriction of the equivariant Todd class $\Td_\cX^\lambda$ 
to the fixed point $(\sigma,v)$ in $I\cX$ and 
$e_T(T_\sigma\cX_v)$ is the $T$-equivariant Euler class, \emph{i.e.} 
products of the $T$-weights of $T_\sigma\cX_v$. 
Since $V(\sigma)$ is the order of the stabilizer at $\sigma\in \cX$, 
the Equation (\ref{eq:connection_cI_H}) follows from the localization 
theorem in $T$-equivariant cohomology \cite{atiyah-bott}. 
In the non-equivariant limit, we have (\ref{eq:Gamma_0_strsheaf}). 
Because $[e^{W_q/z}\omega_q]$ and 
its derivatives $z\nabla_{a_1}\cdots z\nabla_{a_k} [e^{W_q/z} \omega_q]$  
generate the $B$-model \seminf VHS, it follows that 
$\Gamma_0$ corresponds to 
the linear form $(\cdot, \cO_\cX)_{K(\cX)_\C}$. 

\subsubsection{Proof of Theorem \ref{thm:pulledbackintstr}}
Let $\cV_\Z^\cX\subset \cV^\cX$ 
be the $\hGamma$-integral structure 
in Definition-Proposition 
\ref{def-prop:A-model_int} and 
$\tilde{\cV}_\Z^\cX\subset \cV^\cX$ 
be the integral structure pulled back 
from the B-model via mirror isomorphism (\ref{eq:mirror_isom}). 
We know by Theorem \ref{thm:connection_cI_H} that 
the integral vector $\Gamma_0\in \cR_\Z^\vee$ 
corresponds to the linear form 
$\alpha\mapsto (\alpha,\Psi(\cO_\cX))_{\cV^\cX}$. 
Since the B-model integral structure is unimodular 
(and the B-model pairing corresponds to $(\cdot,\cdot)_{\cV^\cX}$ 
by mirror conjecture assumption),  
it follows that $\Psi(\cO_\cX)\in \tilde{\cV}^\cX_\Z$.  
By Proposition \ref{prop:char_A_real_int_str}, 
$\tilde{\cV}^\cX_\Z$ must 
be invariant under Galois action. Hence by Definition-Proposition  
\ref{def-prop:A-model_int}, we know that 
$\Psi(\Z[\Pic(\cX)] \cO_\cX ) \subset \tilde{\cV}_\Z^\cX$. 
Because $K(\cX)$ is generated by 
line bundles \cite{borisov-horja-K}, 
we have $\cV_\Z^\cX \subset \tilde{\cV}^{\cX}_\Z$. 
Since the pairing $(\cdot,\cdot)_{\cV^\cX}$ 
is unimodular on $\cV_\Z^\cX$ by assumption (A3), 
we must have $\cV_\Z^\cX = \tilde{\cV}^\cX_\Z$.

\section{Example: $tt^*$-geometry of $\Proj^1$} 
\label{sec:exampleP1tt*} 
We calculate the Cecotti-Vafa structure on quantum 
cohomology of $\Proj^1$ with respect to 
the real structure induced from the 
$\hGamma$-integral structure 
in Definition-Proposition \ref{def-prop:A-model_int}. 
By Theorem \ref{thm:pulledbackintstr},  
this is the same as 
the Cecotti-Vafa structure associated to the 
Landau-Ginzburg model (mirror of $\Proj^1$): 
\[
\pi\colon (\C^*)^2 \to (\C^*), \quad (x,y) \mapsto q=xy, 
\quad W = x + y. 
\]

Let $\omega\in H^2(\Proj^1)$ be the unique 
integral K\"{a}hler class.  
Let $\{t^0,t^1\}$ be the linear co-ordinate system 
on $H^*(\Proj^1)$ 
dual to the basis $\{\unit, \omega\}$. 
Put $\tau = t^0 \unit + t^1 \omega$. 
The quantum product $\circ_\tau$ is given by 
\[
(\unit \circ_\tau) = 
\begin{bmatrix}
1 & 0 \\
0 & 1 
\end{bmatrix} 
, \quad 
(\omega \circ_\tau) =  
\begin{bmatrix}
0 & e^{t^1} \\ 
1 & 0 
\end{bmatrix}, 
\]
where we identify $\unit,\omega$ with  
column vectors $[1,0]^{\rm T}$, $[0,1]^{\rm T}$ and
the matrices act on vectors by the 
left multiplication. 
The exponential $e^{t^1}$ corresponds to 
$q$ in the Landau-Ginzburg model via the mirror map, 
so we set $q = e^{t^1}$.  
Hereafter, we restrict $\tau$ to lie on $H^2(\Proj^1)$ 
but we will not lose any information by this 
(see Remark \ref{rem:P1tt*} below). 
Recall that the Hodge structure $\F_\tau$ 
associated with the quantum cohomology 
of $\Proj^1$ is given by the image of 
$\cJ_\tau \colon H^*(\Proj^1) \otimes \C\{z\} 
\to \cH^{\Proj^1} = H^*(\Proj^1) 
\otimes \C\{z,z^{-1}\}$ in (\ref{eq:fundamentalsol_A}). 
The map $\cJ_\tau$ is given by the explicit 
hypergeometric function 
$J(q,z)$ \cite{givental-mirrorthm-toric}:  
\begin{align*}
&\cJ_\tau = 
\begin{pmatrix}
\vert & \vert \\ 
J(e^{t^1},z) & z \partial_1 J(e^{t^1},z) \\ 
\vert & \vert  
\end{pmatrix} 
= e^{t^1 \omega/z} \circ Q, \quad 
Q:= 
\begin{pmatrix}
J_0 & z\partial_1 J_0  \\
J_1/z &  J_0 + \partial_1 J_1 
\end{pmatrix}, \\
&J(q,z) :=  e^{t^1 \omega/z} 
\sum_{k=0}^\infty 
\frac{q^k \unit}{(\omega + z)^2 \cdots (\omega+kz)^2}
= e^{t^1\omega/z} (J_0(q,z) \unit + J_1(q,z) \frac{\omega}{z}), 
\end{align*}  
where $\partial_1 = (\partial/\partial t^1)$. 
By Definition-Proposition \ref{def-prop:A-model_int}, 
an integral basis of $\cV^{\Proj^1}=H^*(\Proj^1)$ is given by 
\[
\Psi(\cO_{\Proj^1})  = 
\frac{1}{\sqrt{2\pi}}(\unit -2\gamma \omega),  \quad 
\Psi(\cO_{\pt})  = \sqrt{2\pi} \iu \omega,   
\] 
where $\gamma$ is the Euler constant. 
Hence the real involutions on $\cV^{\Proj^1}$ and 
$\cH^{\Proj^1}$ are given respectively by 
(see (\ref{eq:relation_kappa_H_V})):  
\[
\kappa_\cV =
\begin{bmatrix}
1 & 0 \\ 
-4\gamma & -1 
\end{bmatrix}\circ \ov{\phantom{A}}, \quad 
\kappa_\cH = 
\begin{bmatrix}
z & 0 \\ 
-4\gamma & -z^{-1} 
\end{bmatrix}\circ \ov{\phantom{A}}.    
\]
where $\ov{\phantom{A}}$ is the usual complex conjugation 
(when $z$ is on $S^1 =\{|z|=1\}$). 

To obtain the Cecotti-Vafa structure, we need to find 
a basis of $\F_\tau \cap \kappa_\cH(\F_\tau)$. 
The procedure below follows the proof of 
Theorem \ref{thm:pure_polarized} 
in Section \ref{subsec:pure_polarized}. 
Put $\F_\tau ' := e^{-t^1\omega/z}\F_\tau$ 
and $\kappa^\tau_\cH := e^{-(t^1+\ov{t^1})\omega/z}\kappa_\cH$. 
By 
\[ 
\F_\tau \cap \kappa_\cH(\F_\tau) 
= e^{t^1\omega/z} (\F_\tau' \cap \kappa_\cH^\tau (\F'_\tau)), 
\] 
it suffices to calculate a basis of 
$\F_\tau' \cap \kappa_{\cH}^\tau (\F'_\tau)$. 
First we approximate $\F_\tau'$ by 
$\Flim := H^*(\Proj^1)\otimes \C\{z\}$ 
and solve for a basis of $\Flim \cap \kappa^\tau_{\cH}(\Flim)$. 
By elementary linear algebra, we find 
the following Birkhoff factorization of 
$[\kappa^\tau_\cH(\unit), \kappa^\tau_\cH(\omega)]$: 
\[
[\kappa^\tau_\cH(\unit),\kappa^\tau_\cH(\omega)] = BC, \quad 
B:= 
\begin{bmatrix}
1 & z/a_\tau \\ 
0 & 1 
\end{bmatrix}, \quad 
C:= 
\begin{bmatrix}
0 & 1/a_\tau \\
a_\tau & -1/z 
\end{bmatrix}, 
\]
where $a_\tau := -t^1-\ov{t^1} - 4\gamma$. 
Then the column vectors of $B$ give a basis of 
$\Flim \cap \kappa^\tau_{\cH}(\Flim)$ (\emph{c.f.} (\ref{eq:A=BC})). 
Note that the column vectors of $Q$ above 
form a basis of $\F'_\tau$. 
Thus it suffices to calculate 
the Birkhoff factorization of $Q^{-1}\kappa_\cH^\tau(Q)$ 
to solve for a basis of $\F'_\tau \cap \kappa_\cH^\tau(\F'_\tau)$. 
Define a matrix $S$ by 
\[
\kappa_\cH^\tau (Q) = Q B S C. 
\]
Using that $Q^{-1}$ is given by the adjoint of $Q(-z)$ 
(by Proposition \ref{prop:fundamentalsol_A}), 
we have 
\[
S = 
\large
\begin{bmatrix}
\substack{
2\Re(J_0\ov{J_1})a_\tau^{-1} 
+ |J_0|^2 + 2\Re(\partial_1J_0\ov{J_1}
+J_0\ov{\partial_1J_1}) \\
+ 2  \Re(\partial_1J_0\ov{\partial_1 J_1})a_\tau 
-  |\partial_1J_0|^2 a_\tau^2 
}
&  
\substack{
( 2\Re(J_0\ov{J_1}) a_\tau^{-2} \\    
+ (\partial_1 J_0\ov{J_1}   
+ \ov{J_0}\partial_1 J_1) a_\tau^{-1}   
- \partial_1 J_0\ov{J_0}) z    
}
\\ 
\substack{
( -2\Re(J_0\ov{J_1}) 
- (\ov{\partial_1 J_0}J_1 + J_0
   \ov{\partial_1 J_1}) a_\tau \\
 + J_0\ov{\partial_1 J_0} a_\tau^2) z^{-1} 
}     
& 
\substack{
-2\Re(J_1\ov{J_0}) a_\tau^{-1} + |J_0|^2
} 
\end{bmatrix}, 
\]
where we restrict $z$ to lie on $S^1=\{|z|=1\}$. 
Because $S = \unit + O(|q|^{1-\epsilon})$, $\epsilon>0$ as 
$|q|\to 0$, 
this admits the Birkhoff factorization 
$S= \tilde{B}\tilde{C}$ for $|q|\ll 1$, where 
$\tilde{B}\colon\D_0 \to GL_2(\C)$, 
$\tilde{C}\colon \D_\infty \to GL_2(\C)$ 
such that $\tilde{B}(0) =\unit$. 
Then the column vectors of $QB\tilde{B}=
\kappa_{\cH}^\tau(Q) C^{-1}\tilde{C}^{-1}$ give a 
basis of $\F_\tau'\cap \kappa_\cH^\tau(\F_\tau')$. 
We perform the Birkhoff factorization in the following way. 
Note that $S$ is expanded in a power series in $q$ and $\ov{q}$ 
with coefficients in Laurent polynomials in $a_\tau$ and 
$z$: 
\[
S = \sum_{n,m\ge 0} 
S_{n,m} q^n \ov{q}^m, \quad 
S_{n,m} \in 
\End(\C^2)[z,z^{-1},a_\tau,a_\tau^{-1}]. 
\]
We put $\tilde{B} = \sum_{n,m\ge 0}
\tilde{B}_{n,m}q^n\ov{q}^m$, 
$\tilde{C} = \sum_{n,m\ge 0} \tilde{C}_{n,m} q^n \ov{q}^m$. 
Since $S_{0,0}=\tilde{B}_{0,0}=\tilde{C}_{0,0}=\id$, 
we can recursively solve for 
$\tilde{B}_{n,m}$ and $\tilde{C}_{n,m}$ 
by decomposing 
\[
\tilde{B}_{n,m} + \tilde{C}_{n,m} = S_{n,m} - 
\sum_{(i,j)\neq 0, (n-i,m-j)\neq 0} 
\tilde{B}_{i,j} \tilde{C}_{n-i,m-j} 
\] 
into strictly positive power series 
$\tilde{B}_{n,m}$ and 
non-positive power series $\tilde{C}_{n,m}$ in $z$.   
The first six terms of $B\tilde{B}$ are given by 
\begin{align*}
B\tilde{B} &= 
\begin{bmatrix}
1 & \frac{z}{a_\tau} \\ 
0 & 1 
\end{bmatrix} 
+ \ov{q}
\begin{bmatrix} 
(1+a_\tau)z^2 & \frac{z^3}{a_\tau} \\ 
(2+2a_\tau + a_\tau^2) z & \frac{(2+a_\tau)z^2}{a_\tau}  
\end{bmatrix} 
+ q\ov{q}
\begin{bmatrix}
0 & -\frac{(8+8a_\tau+2 a_\tau^2)z}{a_\tau^2} \\
0 & 0 
\end{bmatrix} \\
+ & \ov{q}^2 
\begin{bmatrix}
\frac{(1+2a_\tau)z^4}{4} & \frac{z^5}{4a_\tau} \\
\frac{(3+6a_\tau+2a_\tau^2)z^3}{4} & 
\frac{(3+a_\tau)z^4}{4a_\tau}  
\end{bmatrix} 
+ q \ov{q}^2 
\begin{bmatrix}
\frac{(33+34a_\tau+18a_\tau^2+4a_\tau^3)z^2}{4} & 
-\frac{(32+31a_\tau+12a_\tau^2 +2a_\tau^3)z^3}{4a_\tau^2} \\
\frac{(25+50 a_\tau +34a_\tau^2 +12 a_\tau^3 +2 a_\tau^4)z}{2} &
-\frac{(64+78a_\tau+ 45 a_\tau^2+14 a_\tau^3 +2 a_\tau^4)z^2}
{4a_\tau^2} 
\end{bmatrix}
\\ 
+ &\ov{q}^3 
\begin{bmatrix}
\frac{(1+3a_\tau)z^6}{36} & \frac{z^7}{36a_\tau} \\
\frac{(11+33a_\tau+9a_\tau^2)z^5}{108} & 
\frac{(11+3a_\tau)z^6}{108a_\tau} 
\end{bmatrix}  
+ O((\log|q|)^5 |q|^4) 
\end{align*} 
Let $\Phi_\tau$ denote the inverse 
to the natural projection 
$\F_\tau \cap \kappa_{\cH}(\F_\tau) \to \F_\tau/z\F_\tau 
= H^*(\Proj^1)$. 
Because $B\tilde{B}=\unit+O(z)$, 
we have 
$[\Phi_\tau(\unit),\Phi_\tau(\omega)] 
= e^{t^1\omega/z} QB\tilde{B}$: 
\[
\begin{CD}
\Phi_\tau \colon H^*(\Proj^1) = 
\F_\tau'/z\F_\tau' @>{QB\tilde{B}}>>
\F'_\tau \cap \kappa_{\cH}^\tau(\F_\tau') 
@>{e^{t^1\omega/z}}>> 
\F_\tau \cap \kappa_{\cH}(\F_\tau).    
\end{CD} 
\] 
The Cecotti-Vafa structure for $\Proj^1$ 
is defined on the trivial vector bundle 
$K := H^*(\Proj^1)\times H^*(\Proj^1) \to H^*(\Proj^1)$. 
Recall that the Hermitian metric $h$ 
on $K_\tau$ is the pull-back of the Hermitian metric 
$(\alpha, \beta) \mapsto 
(\kappa_{\cH}(\alpha),\beta)_{\cH}$ 
on $\F_\tau \cap \kappa_\cH(\F_\tau)$ 
through $\Phi_\tau \colon K_\tau 
\cong \F_\tau \cap \kappa_\cH(\F_\tau)$. 
The Hermitian metric $h$ is of the form: 
\[
h = 
\begin{bmatrix}
h_{\ov{0}0} & 0 \\
0 & h_{\ov{0}0}^{-1}  
\end{bmatrix}, \quad 
h_{\ov{0}0}:= \int_{\Proj^1} 
\kappa_{\cH}(\Phi_\tau(\unit))\Big|_{z\mapsto -z} 
\cup \Phi_\tau(\unit). 
\] 
The first seven terms of the expansion of $h_{\ov{0}0}$ 
are (with $a_\tau = -t^1-\ov{t^1}-4\gamma$, $q=e^{t^1}$) 
\scriptsize 
\begin{align*} 
h_{\ov{0}0} = & a_\tau+ 
|q|^{2} 
\left( {a_\tau^3}+4{a_\tau^2}+8 a_\tau+8 \right) + 
|q| ^{4} 
\left( {a_\tau^5}+ 8{a_\tau^4}+{\frac{121}{4}}{a_\tau^3}+
{\frac {129}{2}}{a_\tau^2}+{\frac 
{145}{2}} a_\tau+{\frac {145}{4}} \right) \\ 
&+ |q|^{6} 
\left( a_\tau^{7}+12 a_\tau^{6}+{\frac {275}{4}}a_\tau^{5}+{
\frac {477}{2}}a_\tau^{4}+{\frac {9539}{18}}a_\tau^{3}  
+ {\frac {81001}{108}} 
a_\tau^{2}+{\frac {50342}{81}}a_\tau 
+{\frac {55526}{243}} \right) \\ 
& +|q|^{8} 
\left( a_\tau^{9}+16a_\tau^{8}+{
\frac {493}{4}}a_\tau^{7}+{\frac {1185}{2}}a_\tau^{6}+{\frac {31001}{16}
}a_\tau^{5}+{\frac {79939}{18}}a_\tau^{4}+{\frac {49077907}{6912}}a_\tau^{3}
\right. \\ 
& \qquad \quad \left. 
+{\frac {52563371}{6912}}a_\tau^{2}+{\frac {614694323}{124416}}a_\tau+{
\frac {736622003}{497664}} \right)  \\
& + |q|^{10}
\left( a_\tau^{11}+20a_\tau^{10}+{\frac {775}{4}}a_\tau^{9}+
{\frac {2381}{2}}a_\tau^{8}+{\frac {368599}{72}}a_\tau^{7}+{\frac {
1738481}{108}}a_\tau^{6}+{\frac {780126811}{20736}}a_\tau^{5} 
\right. \\ 
& \qquad \quad +\left. {\frac {
4053627445}{62208}}a_\tau^{4}+{\frac {254355946241}{3110400}}a_\tau^{3}+
{\frac {1465574917127}{20736000}}a_\tau^{2}+{\frac {163291639271}{
4320000}}a_\tau +{\frac {1840366543439}{194400000}} \right) \\
&+|q| ^{12} 
\left( a_\tau^{13}+24a_\tau^{12}+{\frac {
1121}{4}}a_\tau^{11}+{\frac {4193}{2}}a_\tau^{10}+{\frac {1606399}{144}}
a_\tau^{9}+{\frac {2398517}{54}}a_\tau^{8}+{\frac {2814667745}{20736}}
a_\tau^{7}+{\frac {20004983519}{62208}}a_\tau^{6} \right. \\ 
& \qquad \quad +{\frac {407437321759}{
691200}}a_\tau^{5}+{\frac {51278023471273}{62208000}}a_\tau^{4}+{\frac {
796478452045403}{933120000}}a_\tau^{3}+{\frac {11553263487112967}{
18662400000}}a_\tau^{2} \\ 
& \qquad \quad \left. +{\frac {11823418405646927}{41990400000}}a_\tau
+{\frac{15268380040196927}{251942400000}} \right)
+ \cdots. 
\end{align*} 
\normalsize 
The other data 
$(\kappa,g, C,\tilde{C},D,\cU,\ov{\cU},\cQ)$ 
of the Cecotti-Vafa structure 
are given in terms of $h_{\ov{0}0}$. 
In fact, we have $C_0= \tC_{\ov{0}} =\id$, 
$D_0=\partial/\partial t^0$, 
$D_{\ov{0}} = \partial/\partial \ov{t^0}$ and 
\begin{gather*}
g = 
\begin{bmatrix}
0 & 1 \\
1 & 0
\end{bmatrix},  
\quad 
\kappa = 
\begin{bmatrix}
0 & h_{\ov{0}0}^{-1}  \\
h_{\ov{0}0} & 0
\end{bmatrix} \circ \ov{\phantom{A}},  
\quad 
D_1 = \partial_1 + 
\begin{bmatrix}
\partial_1 \log h_{\ov{0}0} & 0 \\
0 & -\partial_1 \log h_{\ov{0}0} 
\end{bmatrix}, \\  
D_{\ov{1}} = \ov{\partial_1}, \quad  
C_1 = \frac{1}{2} \cU = 
\begin{bmatrix}
0 & e^{t^1} \\
1 & 0  
\end{bmatrix}, \quad 
\tC_{\ov{1}} = 
\frac{1}{2} \ov\cU =
\begin{bmatrix}
0 &  h_{\ov{0}0}^{-2} \\
e^{\ov{t^1}} h_{\ov{0}0}^2 & 0  
\end{bmatrix},  \\ 
\cQ =  \partial_E+ \mu - D_E   
= 
\begin{bmatrix}
-\frac{1}{2} - 2 \partial_1 \log h_{\ov{0}0} & 0 \\
0 & \frac{1}{2} +2 \partial_1 \log h_{\ov{0}0} 
\end{bmatrix},    
\end{gather*} 
where $\partial, \ov{\partial}$ are 
the connections given by the given trivialization 
of $K$. 

\begin{remark}
\label{rem:P1tt*} 
(i) The fact that the Hermitian metric $h$ is 
represented by a diagonal matrix with determinant $1$ 
follows from an elementary argument. 
See \cite[Lemma 2.1]{takahashi-tt*}. 

(ii) From the general theory 
of (trTERP)$+$(trTLEP) structure on the tangent bundle, 
it follows that $h,C,\tC,\cU,\ov{\cU},\cQ$ are invariant 
under the flow of the unit vector field 
$(\partial/\partial t^0), (\partial/\partial \ov{t^0})$. 
Therefore, the calculation here 
determines the Cecotti-Vafa structure on the big 
quantum cohomology. Moreover we have 
$D_E + \cQ = \partial_E + \mu$ and $\Lie_{E-\ov{E}} h =0$. 
In the case of $\Proj^1$, this means that $h_{\ov{0}0}$  
depends only on $|q|$. See \cite{hertling-tt*}.  

(iii) We can show that our procedure for 
the Birkhoff factorization gives convergent 
series for sufficiently small values of $|q|$. 
In particular, the expansion for 
$h_{\ov{0}0}$ converges for small $|q|$. 
\end{remark} 

Our calculation of the Hermitian metric 
$h_{\ov{0}0}$ matches with 
Cecotti-Vafa's result 
\cite{cecotti-vafa-exactsigma} 
for the sigma model of $\Proj^1$. 
The $tt^*$-equation 
$[D_1,D_{\ov{1}}] + [C_1,\tC_{\ov{1}}] =0$ 
gives the following differential equation 
for $h_{\ov{0}0}$:  
\begin{equation}
\label{eq:diffeq_h00}
\partial_1 \ov{\partial_{1}} 
\log h_{\ov{0}0} = - 
h_{\ov{0}0}^{-2} + |q|^2 h_{\ov{0}0}^2. 
\end{equation} 
In \cite{cecotti-vafa-exactsigma}, $h_{\ov{0}0}$ 
was identified with a unique solution to (\ref{eq:diffeq_h00}) 
expanded in the form 
\[
h_{\ov{0}0} = \sum_{n=0}^\infty F_n|q|^{2n}, \quad 
F_0 = a_\tau, \quad F_n\in \C[a_\tau,a_\tau^{-1}], 
\quad a_\tau = -2 \log |q| -4\gamma. 
\]  
The equation (\ref{eq:diffeq_h00}) gives an infinite 
set of recursive differential equations for $F_n$. 
It is easy to check that the differential equations 
determine the Laurent polynomial $F_n$ \emph{uniquely}; 
Moreover it turns out that 
$F_n\in \Q[a_\tau]$ and $\deg F_n = 2n+1$. 
The \emph{existence} of such a solution 
seems to be non-trivial, however 
our Birkhoff factorization method 
certainly gives such $h_{\ov{0}0}$. 
The differential equation (\ref{eq:diffeq_h00}) is   
equivalent to Painlev\'{e} III 
equation \cite{cecotti-vafa-top-antitop}: 
\[
\frac{d^2u}{dz^2} + \frac{1}{z} \frac{du}{dz} = 4 
\sinh(u), \quad h_{\ov{0}0} = e^{u/2} |e^{-t^1/2}|, 
\quad z = 4 |e^{t^1/2}|.  
\] 
that $h_{\ov{0}0}$ here is positive and 
regular on the positive real axis 
$|q|\in (0,\infty)$\footnote{
For this, the constant $\gamma$ in 
$a_\tau$ must be the very Euler constant.}.  
By physical arguments, Cecotti-Vafa 
\cite{cecotti-vafa-top-antitop, cecotti-vafa-classification, 
cecotti-vafa-exactsigma} 
showed that $h_{\ov{0}0}$ is positive and smooth 
on the positive real axis $0< |q| <\infty$. 
Since the Landau-Ginzburg mirror of $\Proj^1$ 
is defined by a cohomologically tame function, 
this fact also follows from Sabbah's result 
\cite{sabbah} in the singularity theory 
(see Remark \ref{rem:sabbah}). 
Therefore, the Cecotti-Vafa structure for $\Proj^1$ 
is well-defined and positive definite 
on the whole $H^*(\Proj^1)$. 
It seems that the same solution as $h_{\ov{0}0}$ 
has been obtained in the study of 
Painlev\'{e} III equation 
\cite{its-novokshenov, mccoy-tracy-wu-PIII} 
(the first few terms of the expansion are the same as ours).  
If this is the case, $h_{\ov{0}0}$ should have 
the asymptotics \cite{its-novokshenov, mccoy-tracy-wu-PIII} 
(also appearing in \cite{cecotti-vafa-exactsigma}):  
\[ 
h_{\ov{0}0} \sim \frac{1}{\sqrt{|q|}} 
(1-\frac{1}{2\sqrt{\pi}|q|^{1/4}}e^{-8|q|^{1/2}}) 
\] 
as $|q|\to \infty$. 
With respect to the metric $h_{\ov{1}1}=h_{\ov{0}0}^{-1}$ 
on the K\"{a}hler moduli space $H^2(\Proj^1)/2\pi\iu 
H^2(\Proj^1,\Z)$, a neighborhood of the 
large radius limit point $q=0$ has negative curvature, 
but does not have finite volume. 
The curvature 
$-\frac{2}{h_{\ov{0}0}}(1-|q|^2 h_{\ov{0}0}^4)$  
goes to zero as $|q|\to 0$ and $|q|\to \infty$ 
and the total curvature is $-\pi/4$. 
Much more examples including $\Proj^n$, $\Proj^1/\Z_n$ 
are calculated in physics literature. 
We refer the reader to    
\cite{cecotti-vafa-top-antitop,cecotti-vafa-exactsigma, 
cecotti-vafa-massiveorb}. 

\section{Integral periods and Ruan's conjecture} 
\label{sec:integralperiods} 
In mirror symmetry for Calabi-Yau manifolds
(see \emph{e.g.} \cite{CDGP,morrison,deligne}), 
flat co-ordinates (or mirror map) $\tau_i$ on the B-model 
in a neighborhood of a maximally unipotent monodromy point 
was given by periods over integral cycles  
$A_1,\dots, A_r$ of a holomorphic $n$-form $\Omega$ 
\[
\tau_i = \int_{A_i} \Omega,   
\]
where $\Omega$ is normalized by the condition: 
\[
\int_{A_0} \Omega = 1. 
\]
Here, $A_0$ is a monodromy-invariant cycle (unique up to sign)  
and $A_1,\dots, A_r$ are such that 
$\Z A_0 \oplus \sum_{i=1}^r \Z A_i$ is preserved 
under monodromy transformations. 
We should note that \emph{flat co-ordinates 
are constructed as integral periods}. 
In this section, by choosing an integral structure 
on the A-model, we define integral periods of 
quantum cohomology analogously. 
We study relationships between integral periods and flat 
co-ordinates in the conformal limit (\ref{eq:conformal_limit}). 
Using integral periods, we will speculate on 
specialization values of quantum parameters 
appearing in Ruan's crepant resolution conjecture. 
From this viewpoint, the specialization to roots of unity 
seems to be natural. 
Throughout this section, we assume that $\cX$ is a 
weak Fano (\emph{i.e.} $\rho=c_1(\cX)$ is nef) 
Gorenstein orbifold without generic stabilizer. 

\subsection{Integral periods} 
The integral periods for a general \seminf VHS can be defined in 
the following way. Recall that a choice of integral structures 
defines an integral lattice $\cV_\Z$ in the space $\cV$ of 
multi-valued flat sections of 
$(\sfH,\hatnabla_{z\partial z})$ on $\C^*$. 
Take a basis $\Gamma_1,\dots, \Gamma_N$ of $\cV_\Z^\vee$. 
Each $\Gamma_i$ defines a multi-valued section $\Gamma_i(\log z)$ 
of the dual bundle 
$\cH^\vee \to \C^*$. Take a $\C\{z\}$-basis 
$s_1(\tau,z),\dots, s_N(\tau,z)$ of the Hodge structure 
$\F_\tau \subset \cH$. We call the pairing 
\[
\tau \mapsto \pair{\Gamma_i(\log z)}{s_j(\tau,z)}
\]
an \emph{integral period}.  
For the Landau-Ginzburg mirror of toric orbifolds, 
these integral periods are given by 
oscillatory integrals over Lefschetz thimbles.  
These periods are themselves multi-valued functions in $z$  
and not easy to understand. 
They will become more tractable if we choose 
$\Gamma_i$ to be invariant under monodromy 
transformations around $z=\infty$. 

\subsection{A-model integral periods in the conformal limit} 
\label{subsec:periods_conformallimit} 
Consider the A-model \seminf VHS of $\cX$ 
with an integral structure. 
The assumption that $\cX$ is Gorenstein 
implies that the age $\iota_v$ is an integer 
for all $v\in \sfT$. 
Therefore, $H_{\rm orb}^*(\cX)$ is graded by even integers.  
Thus by (\ref{eq:z-monodromy_V}), 
the monodromy transformation $M_z \in \End(\cV^\cX)$ 
around $z=\infty$ is of the form: 
\[
M_z = (-1)^n e^{2\pi\iu\rho} 
\]
where $n=\dim_\C\cX$ and $\rho = c_1(\cX)$. 
We define $\cV^\cX_{\Z,1} \subset  
H^*_{\rm orb}(\cX) $ by 
\[
\cV^\cX_{\Z,1} := \Ker(\id - M_z^2) \cap \cV^\cX_\Z =  
\Ker(\rho) \cap \cV^\cX_\Z. 
\]
Under the map (\ref{eq:solutionmap_z}), 
an element of $\cV^\cX_{\Z,1}$ corresponds 
to a flat section of $(\sfH^\cX,\hatnabla_{z\partial_z})$ 
which is single-valued (when $n$ is even) or 
two-valued (when $n$ is odd). 
For convenience, we introduce the space $\hcH^\cX$ 
of possibly two-valued sections of $\sfH^\cX \to \C^*$:  
\[
\hcH^\cX = \cH^\cX\otimes_{\C\{z,z^{-1}\}} 
\C\{z^{1/2},z^{-1/2}\}.  
\]
The pairing on $\cH^\cX$ is extended on $\hcH^\cX$ 
by 
\[
(\alpha,\beta)_{\cH^\cX} = 
(\alpha(\iu z^{1/2}), \beta(z^{1/2}))_{\rm orb} 
\]
where we regard $\alpha,\beta\in \hcH^\cX$ 
as cohomology-valued functions in $z^{1/2}$. 
Under (\ref{eq:solutionmap_z}), $A \in \cV^\cX_{\Z,1}$ 
corresponds to $z^{-\mu}A \in \hcH^\cX$ and 
gives an integral period:   
\begin{equation}
\label{eq:monodromyfree_integralperiod}
\tau \mapsto (z^{-\mu}A, \cJ_\tau(\alpha))_{\cH^\cX}  
\in \C\{z^{1/2},z^{-1/2}\},   
\end{equation} 
where $\alpha\in H_{\rm orb}^*(\cX)\otimes \C\{z\}$ and 
$\cJ_\tau \colon H_{\rm orb}^*(\cX)\otimes \C\{z\} 
\to \cH^\cX$ is an embedding in (\ref{eq:fundamentalsol_A}). 
These integral periods behaves well 
in the following limit:  
\begin{equation}
\label{eq:conformal_limit}
\tau - s \rho, \quad \Re(s) \to \infty  
\end{equation} 
with a fixed $\tau \in H^2_{\rm orb}(\cX)$. 
We call such a sequence in $H^2_{\rm orb}(\cX)$ 
the \emph{conformal limit}. 
Recall that we assumed that $\rho = c_1(\cX)$ is nef. 
For the embedding $\cJ_\tau$ in (\ref{eq:fundamentalsol_A}), 
we define $\cJ^{\rm CY}_\tau\colon 
H^*_{\rm orb}(\cX)\otimes \C\{z\} \to \cH^\cX$ as 
\begin{align} 
\label{eq:CY_Jfunct}
\begin{split}
\cJ_{\tau}^{\rm CY}(\alpha) 
&:= \lim_{\Re(s) \to \infty} 
e^{s\rho/z} \cJ_{\tau-s\rho}(\alpha) \\ 
& = e^{\tau_{0,2}/z}\left( \alpha + 
\sum_{\substack{(d,l)\neq (0,0),\\ d\in \Ker(\rho)}} 
\sum_{i=1}^N \frac{1}{l!}
\corr{\alpha,\tau_{\rm tw},\dots,\tau_{\rm tw}, 
\frac{\phi_i}{z-\psi}}_{0,l+2,d}^\cX 
e^{\pair{\tau_{0,2}}{d}}\phi^i \right).  
\end{split} 
\end{align} 
Here we put $\tau = \tau_{0,2} + \tau_{\rm tw}$ with 
$\tau_{0,2} \in H^2(\cX)$ and 
$\tau_{\rm tw} \in \bigoplus_{\iota_v=1} H^0(\cX_v)$ 
and used that $\pair{\rho}{d}\ge 0$ for all $d \in \Eff_\cX$. 
When $\alpha \in H^{2k}_{\rm orb}(\cX)$, 
$\cJ^{\rm CY}_\tau(\alpha)$ is homogeneous of degree $2k$ 
if we set $\deg(z) =2$.  
From this calculation, the following definition makes sense. 

\begin{definition}
\label{def:CY_limit_VHS}
Assume that $\rho=c_1(\cX)$ is nef. 
Then we can define a new \seminf VHS 
$\F_\tau^{\rm CY}\subset \cH^\cX$ 
(in the moving subspace realization) by  
\[
\F_\tau^{\rm CY} := \lim_{\Re(s)\to \infty} 
e^{s\rho/z} \F_{\tau-s\rho}
= \cJ_\tau^{\rm CY}(H^*_{\rm orb}(\cX)\otimes \C\{z\}), 
\quad \tau 
\in H^2_{\rm orb}(\cX).  
\]
This satisfies $
\F^{\rm CY}_{\tau+ a\rho} = e^{a\rho/z} \F^{\rm CY}_{\tau}$ 
and is homogeneous 
$(z\partial_z +\mu) \F^{\rm CY}_\tau \subset 
\F^{\rm CY}_\tau$. 
\end{definition}  
\begin{remark}
We can also define the new \seminf VHS above
as a sheaf of $\cO_{H^2_{\rm orb}(\cX)}\{z\}$-modules, 
using Dubrovin connection associated with  
a new quantum product  
$\circ^{\rm CY}_\tau := \lim_{\Re(s)\to \infty} \circ_{\tau-s\rho}$.   
The conformal limit of quantum cohomology 
is closely related to Y.\ Ruan's quantum corrected ring 
\cite{ruan-crc}, which is defined by 
counting rational curves contained 
in the exceptional locus (in the case of 
crepant resolution). 
The conformal limit of a \seminf VHS  
appears in the work of 
Sabbah \cite[Part I]{sabbah-hypergeometric}   
as the associated graded of a free 
$\C[z]$-module $G_k$ (an algebraization of $z^{-k}\F_\tau$) 
with respect to the  Kashiwara-Malgrange 
$V$-filtration at $z=\infty$. 
See also Hertling and Sevenheck 
\cite[Section 7]{hertling-sevenheck} for a review. 
\end{remark}

Because of the homogeneity of $\F^{\rm CY}_{\tau}$, 
the \seminf VHS $\{\F^{\rm CY}_\tau\subset \cH^\cX\}$ 
reduces to a \emph{finite dimensional} VHS. 
Set $\hbF^{\rm CY}_\tau := \F_\tau^{\rm CY} \otimes_{\C\{z\}} 
\C\{z^{1/2}\} \subset \hcH^\cX$ 
and $H_0:= \Ker(z\partial_z +\mu) \subset \hcH^\cX$.  
By restriction, the pairing on $\hcH^\cX$ 
induces a $(-1)^n$-symmetric $\C$-valued pairing 
$(\cdot,\cdot)_{H_0}$ on $H_0$. 
By restricting the semi-infinite flag 
$\cdots \supset z^{-1} \hbF^{\rm CY}_\tau \supset 
\hbF^{\rm CY}_\tau \supset z 
\hbF^{\rm CY}_\tau \supset \cdots$ to $H_0$, 
we obtain a finite dimensional flag 
$H_0= F_\tau^0 \supset F_\tau^1 
\supset \cdots \supset F_\tau^n \supset 0$: 
\begin{align*} 
F^{p}_\tau &:= z^{p-n/2}\hbF^{\rm CY}_\tau \cap H_0 \\ 
& = \Span \left
\{ z^{p-n/2} \cJ^{\rm CY}_{\tau}(z^j \alpha)\;;\; \alpha 
\in H^{2n-2p-2j}_{\rm orb}(\cX),j\ge 0 
\right\}.   
\end{align*} 
One can check that $F^p_\tau$ satisfies 
the Griffiths transversality and 
Hodge-Riemann bilinear relation: 
\[
\parfrac{}{t^i} F_\tau^p \subset F_\tau^{p-1}, \quad 
(F_\tau^p, F_\tau^{n-p+1})_{H_0} =0. 
\]
The real involution $\kappa_\cH$ on $\cH^\cX$ induces  
those on $\hcH^\cX$ and $H_0$ since 
$z\partial_z + \mu$ is purely imaginary on $\cH^\cX$ 
by (\ref{eq:grading_imaginary_H}).   
Denote by $\kappa_{H_0}$ the real involution on $H_0$. 
When moreover $\F^{\rm CY}_\tau$ is pure and polarized 
(these properties hold near the large radius limit 
if the conditions of Theorem 
\ref{thm:pure_polarized} are satisfied), 
one can easily check that $F^p_{\tau}$ satisfies 
the Hodge decomposition and 
Hodge-Riemann bilinear inequality:  
\[
H_0 = F^p_\tau \oplus \kappa_{H_0}(F^{n-p+1}_\tau), 
\quad 
(-\iu)^{2p-n} (\phi, \kappa_{H_0}(\phi))_{H_0} >0 
\]
where $\phi \in F^{p}_\tau \cap \kappa_{H_0}(F^{n-p}_\tau) 
= z^{p-n/2} 
(\hbF^{\rm CY}_\tau \cap \kappa_{\cH}(\hbF^{\rm CY}_\tau)) 
\cap H_0$. 
Conversely, this finite dimensional VHS $F^\bullet_\tau$ 
recovers the \seminf VHS $\F^{\rm CY}_\tau$ by 
\[
\F_\tau^{\rm CY} = z^{-n/2} F^n_\tau \otimes \C\{z\} + 
z^{-n/2+1} F^{n-1}_\tau \otimes \C\{z\} + 
\dots  + z^{n/2} F^0_\tau \otimes \C\{z\}.  
\]

In contrast to the real structure, 
the integral structure on the A-model \seminf VHS 
does not induce a full integral lattice of $H_0$. 
One can see however that the lattice 
$\cV^\cX_{\Z,1}$ is naturally 
contained in $H_0$ by $A \mapsto  z^{-\mu} A$  
as a {\it partial} lattice.  
The pairing between $z^{-\mu}A$ with $A\in\cV^\cX_{\Z,1}$ 
and a section of $F^p_\tau$ gives an integral period 
for $F^p_\tau$. 
Let us consider the conformal limit of the A-model 
integral period (\ref{eq:monodromyfree_integralperiod}) 
with $\alpha \in H^{2p}_{\rm orb}(\cX)$: 
\begin{align*} 
\lim_{\Re(s)\to \infty} 
(z^{-\mu}A, \cJ_{\tau-s\rho} (\alpha))_{\cH^\cX} 
&= 
\lim_{\Re(s)\to \infty} 
(z^{-\mu}A, e^{s\rho} \cJ_{\tau - s\rho}(\alpha))_{\cH^\cX} \\ 
&= z^{p-n/2} 
(z^{-\mu}A, z^{n/2-p} \cJ_{\tau}^{\rm CY}(\alpha))_{H_0} 
\in  z^{p-n/2} \C.    
\end{align*} 
Note that the last line is a period of  
$z^{n/2-p} \cJ_{\tau}^{\rm CY}(\alpha) \in F^{n-p}_\tau$. 
Therefore, \emph{the A-model integral period 
(\ref{eq:monodromyfree_integralperiod}) 
approaches to a period of the finite dimensional VHS 
in the conformal limit}. 
Note that in this limit, the integral period depends only on 
$\tau \in H^2_{\rm orb}(\cX)/\C\rho$. 

Now we focus on integral periods for $F^n_\tau \subset H_0$. 
Note that $F^n_\tau = z^{n/2} \hbF^{\rm CY}\cap H_0 
=z^{n/2}\cJ^{\rm CY}_\tau(H^0_{\rm orb}(\cX))$ 
is one dimensional over $\C$. 
We use the Galois action (monodromy action) to choose
a good set of integral vectors in $\cV_{\Z,1}^\cX$. 
Let $L$ be a line bundle on $\cX$ which is a 
pull-back of an ample line bundle 
on the coarse moduli space $X$ of $\cX$. 
Then the Galois action $G^\cV([L])
\in \End(\cV^\cX)$ in (\ref{eq:Galois_V}) is unipotent 
since $f_v([L])=0$ for a pulled-back line bundle $L$. 
Take a weight filtration $W_k$ on $\cV^\cX$ 
defined by 
the logarithm $\Log(G^\cV([L])) = - 2\pi\iu c_1(L)$.  
See the proof of Proposition \ref{prop:basic_realstr} 
for the weight filtration. 
This is given by (independent of a choice of $L$)  
\begin{equation}
\label{eq:weightfiltr_ampleLoncoarse} 
W_k = \bigoplus_{v\in \sfT} H^{\ge n_v-k}(\cX_v).  
\end{equation} 
The weight filtration is defined over $\Q$. 
We will also use the subspace 
$\Ker(H^2(\cX)) = \{\alpha\in \cV^\cX \;;\; 
\tau_{0,2}\cdot \alpha=0, \forall 
\tau_{0,2}\in H^2(\cX)\}$. 
Since this consists of $\alpha\in \cV^\cX$ 
satisfying $G^\cV(\xi) \alpha = \alpha$ for every 
integral cohomology class $\xi\in H^2(X,\Z)$ 
on the coarse moduli space, 
this is also defined over $\Q$. 
These subspaces define 
the following filtration on $\cV_{\Z,1}^\cX$: 
\[
(W_{-n} \cap \cV_{\Z,1}^\cX) \subset 
(\Ker(H^2(\cX))\cap W_{-n+2} \cap \cV_{\Z,1}^\cX) \subset 
(W_{-n+2} \cap \cV_{\Z,1}^\cX)  
\]
which are full lattices of the vector spaces: 
\[
H^{2n}(\cX)\subset 
H^{2n}(\cX) \oplus \bigoplus_{n_v=n-2} H^{2n_v}(\cX_v) 
\subset 
(H^{\ge 2n-2}(\cX)\cap \Ker(\rho)) 
\oplus \bigoplus_{n_v=n-2} H^{2 n_v}(\cX_v). 
\]
Note that by the Gorenstein assumption, 
there is no $v\in \sfT$ satisfying $n_v=n-1$ and 
that $n_v=n-2$ implies $\iota_v=1$. 
Thus these subspaces are contained in 
$H^{\ge 2n-2}_{\rm orb}(\cX)$. 
We take integral vectors $A_0,A_1,\dots, A_\flat$, 
$A_{\flat+1},\dots,A_\sharp$ in $\cV_{\Z,1}^\cX$ 
compatible with this filtration: 
\begin{align*}
W_{-n} \cap \cV_{\Z,1}^\cX &= \Z A_0, \\
\Ker(H^2(\cX)) \cap W_{-n+2} \cap \cV_{\Z,1}^\cX 
& = \Z A_0 + \textstyle\sum_{i=1}^\flat \Z A_i, \\ 
W_{-n+2} \cap \cV_{\Z,1}^\cX 
&= \Z A_0 + \textstyle\sum_{i=1}^\flat \Z A_i
+ \sum_{i=\flat+1}^\sharp \Z A_i. 
\end{align*} 
The vector $A_0$ is unique up to sign 
and invariant under all Galois action. 
In analogy with the Calabi-Yau B-model,  
we normalize a generator $\Omega_\tau \in F^n_\tau$ 
by the condition 
\begin{equation}
\label{eq:normalization_Omega} 
(z^{-\mu}A_0, \Omega_\tau)_{H_0} = 1.  
\end{equation} 

\begin{proposition} 
\label{prop:integralperiods_basic} 
For $\tau \in H^2_{\rm orb}(\cX)$, 
we write $\tau = \tau_{0,2} + \tau_{\rm tw} = 
\tau_{0,2} + \tau_{\rm tw}' + \tau_{\rm tw}''$ 
with $\tau_{0,2} \in H^2(\cX)$, 
$\tau_{\rm tw} \in \bigoplus_{\iota_v=1}H^0(\cX_v)$, 
$\tau_{\rm tw}'\in \bigoplus_{n_v = n-2}H^0(\cX_v)$ 
and $\tau_{\rm tw}''\in 
\bigoplus_{n_v<n-2, \ \iota_v=1} H^0(\cX_v)$. 
Define $a_i:= (A_i,1)_{\rm orb}$. 
Under the normalization (\ref{eq:normalization_Omega}), 
we have $\Omega_\tau = 
\iu^n a_0^{-1} z^{n/2} \cJ^{\rm CY}_\tau(1)$ 
and the integral periods $(z^{-\mu}A_i,\Omega_\tau)_{H_0}$  
give an affine co-ordinate system 
on $(H^2(\cX)/\C\rho) \oplus \bigoplus_{n_v=n-2} H^0(\cX_v)$:   
\begin{align*} 
(z^{-\mu}A_i, \Omega_\tau)_{H_0} &= 
a_0^{-1}a_i - (a_0^{-1}A_i,\tau_{\rm tw}')_{\rm orb}, \quad 
1\le i\le \flat,  \\ 
(z^{-\mu}A_i, \Omega_\tau)_{H_0} & =
a_0^{-1}a_i  - (a_0^{-1}A_i,\tau_{\rm tw}')_{\rm orb} 
- \frac{1}{2\pi\iu} [C_i]\cap\tau_{0,2},      
\quad \flat+ 1\le i \le \sharp 
\end{align*} 
where $[C_i]\in H_2(\cX)$ is the Poincar\'{e} dual 
of the $H^{2n-2}(\cX)$-component of $2\pi\iu a_0^{-1} A_i$ 
and 
\begin{equation}
\label{eq:integer_A_H2} 
[C_i]\in H_2(X,\Z)\cap \Ker\rho, \quad 
\text{where $X$ is the coarse moduli space of $\cX$}.  
\end{equation} 
$[C_{\flat+1}],\dots,[C_{\sharp}]$ form a 
$\Q$-basis of $H_2(X,\Q)\cap \Ker\rho$. 
The period for 
$B \in \Ker(H^2(\cX)) \cap \cV^\cX_{\Z,1}$ is possibly non-linear 
and has the asymptotic 
\[
(z^{-\mu}B, \Omega_\tau)_{H_0} \sim 
a_0^{-1}b - (a_0^{-1} B,\tau_{\rm tw})_{\rm orb}, \quad 
b := (B,1)_{\rm orb}  
\]
as $\tau$ goes to the large radius limit point:  
\[ 
\Re(\pair{\tau_{0,2}}{d}) \to -\infty, \quad 
\forall d\in \Eff_{\cX}\setminus \{0\}, 
\quad \tau_{\rm tw} \to 0.  
\] 
For the constant terms $a_0^{-1} a_i$, $a_0^{-1} b$ of 
integral periods, we have the following: 
\begin{itemize} 
\item[(i)] 
If the following condition holds, 
\begin{align} 
\label{eq:sep_codim2_trivial}
&\forall v \in \sfT \ 
(n_v = n-2 \ \Longrightarrow \ 
\exists \xi\in H^2(\cX,\Z) \text{ such that }  f_v(\xi)>0),  
\end{align} 
we have  $a_0^{-1} a_i  \in \Q$ for $1\le i\le \flat$.  
\item[(ii)] 
If moreover $H^*(\cX)$ is generated by $H^2(\cX)$ as a ring 
and the following holds,  
\begin{equation}
\label{eq:sep_trivial}
\forall v \in \sfT \ 
(v \neq 0 \ \Longrightarrow \ \exists 
\xi \in H^2(\cX,\Z) \text{ such that } f_v(\xi)>0), 
\end{equation} 
we have $a_0^{-1} b\in \Q$ for $b=(B,1)_{\rm orb}$ 
and $B\in \Ker(H^2(\cX))\cap \cV_{\Z,1}^\cX$. 
\item[(iii)] 
If the following holds for the integral structure,   
\begin{equation} 
\label{eq:curveclass_overQ}
(H^{2n-2}_{\rm orb}(\cX)\cap W_{-n+2} 
\cap \Ker\rho)  
\subset \cV^\cX   
\text{ is defined over $\Q$,} 
\end{equation} 
we have $a_0^{-1} a_i \in \Q$ for 
$1\le i \le\sharp$.  
\end{itemize}  

\end{proposition} 
\begin{proof} 
By (\ref{eq:CY_Jfunct}) and the string equation (see \cite{AGV}), 
$\cJ^{\rm CY}_{\tau}(1)$ has the following expansions: 
\begin{align*} 
\cJ^{\rm CY}_\tau(1) & = e^{\tau_{0,2}/z} 
\left (1 + \frac{\tau_{\rm tw}}{z} + 
\sum_{\substack{d\in \Eff_\cX\cap \Ker(\rho),  
\\ l\ge 0,\\ 
d=0 \Rightarrow l\ge 2.}}
\sum_{i=1}^N 
\corr{\tau_{\rm tw},\dots,\tau_{\rm tw},  
\frac{\phi_i}{z(z-\psi)} }_{0,l+1d}^\cX 
e^{\pair{\tau_{0,2}}{d}} \phi^i \right ) \\ 
& = 1 + \frac{\tau}{z} + 
z^{-2} H^{\ge 4}_{\rm orb}(\cX)\otimes \C\{z^{-1}\}.  
\end{align*}  
The forms of $\Omega_\tau$, 
$(z^{-\mu}A_i,\Omega_\tau)_{H_0}$ and $(z^{-\mu}B,\Omega_\tau)$ 
easily follow from these expansions. 
If $\xi \in H^2(X,\Z)$ 
is an integral class on the coarse moduli space, 
we have $G^{\cV}(\xi) = e^{-2\pi\iu\xi}$ by (\ref{eq:Galois_V}).  
Because the Galois action preserves the integral structure, 
$e^{-2\pi\iu\xi} A_i = A_i - m_i A_0$ 
for some integer $m_i$. Here, 
$2\pi\iu\xi A_i = m_i A_0$. 
Hence, 
$[C_i]\cap\xi = (2\pi\iu a_0^{-1} A_i, \xi)_{\rm orb} = 
a_0^{-1} (2\pi\iu\xi A_i, 1)_{\rm orb} = m_i\in \Z$. 
This shows (\ref{eq:integer_A_H2}). 
We set $V:= H^{2n}(\cX) \oplus 
\bigoplus_{n_v=n-2} H^{2n_v}(\cX_v) \subset \cV^\cX$. 
The Galois action preserves the full-lattice 
$\Z A_0 + \sum_{i=1}^\flat \Z A_i$ of $V$. 
The Galois action on $V$ is simultaneously diagonalizable. 
Under the condition (\ref{eq:sep_codim2_trivial}), 
$\C A_0$ gives the simultaneous eigenspace of eigenvalue $1$ 
and $V'= \bigoplus_{n_v=n-2} H^{2n_v}(\cX_v)$ 
gives the sum of simultaneous eigenspaces other than $\C A_0$. 
Then the direct sum decomposition 
\[
V = \C A_0 \oplus V' 
\]
is actually defined over $\Q$ since 
this is invariant under the Galois group over $\Q$. 
Therefore, there exists a rational number 
$c_i$ and $A_i'\in V'$ 
such that $A_i = c_i A_0 + A_i'$ for $1\le i\le \flat$. 
Hence $a_i=(A_i,1)_{\rm orb} = c_i (A_0,1)_{\rm orb}
= c_i a_0$. 
This shows (i). 
On the other hand, the Galois action on 
$\Ker(H^2(\cX))$ is again simultaneously diagonalizable 
and preserves its full lattice 
$\Ker(H^2(\cX))\cap \cV^\cX_{\Z,1}$. 
When $H^*(\cX)$ is generated by $H^2(\cX)$, we have 
$\Ker(H^2(\cX)) \cap H^*(\cX) = H^{2n}(\cX)$ 
by Poincar\'{e} duality. 
Therefore, if moreover (\ref{eq:sep_trivial}) holds, 
$H^{2n}(\cX)= \C A_0$ 
is the simultaneous eigenspace of eigenvalue $1$ 
of the Galois action on $\Ker(H^2(\cX))$. 
For the same reason as above, the decomposition 
$\Ker(H^2(\cX)) = \C A_0 \oplus 
(\bigoplus_{v\in \sfT'} H^*(\cX_v) \cap  \Ker(H^2(\cX)))$ 
is defined over $\Q$. 
(ii) follows from this. 
The condition (\ref{eq:curveclass_overQ}) 
implies the decomposition over $\Q$: 
$W_{-n+2}\cap \Ker\rho = H^{2n}(\cX) \oplus 
(H_{\rm orb}^{2n-2}(\cX)\cap W_{-n+2} \cap \Ker\rho)$.  
(iii) follows from this similarly. 
\end{proof} 

\begin{remark}
The conditions (\ref{eq:sep_codim2_trivial}), 
(\ref{eq:sep_trivial}) are weaker versions of 
(\ref{eq:separation}). 
The condition (\ref{eq:curveclass_overQ}) 
does not seem to follow from monodromy consideration.  
But this happens for the $\hGamma$-integral structures. 
See Example \ref{ex:hGamma} below. 
\end{remark} 

\begin{example}
\label{ex:hGamma} 
Taking the $\hGamma$-integral structure 
in Definition-Proposition \ref{def-prop:A-model_int}, 
we give explicit examples of A-model integral periods.  
By a natural map from the $K$-group of coherent sheaves 
to the $K$-group of topological orbifold vector bundles, 
we can regard a coherent sheaf as an element of $K(\cX)$. 
The integral vector $A_0 \in W_{-n}\cap \cV^\cX_{\Z,1}$ 
is given by the image of the structure sheaf 
$\cO_x$ of a non-stacky point $x\in \cX$: 
\[
A_0 = \Psi([\cO_x]) = 
\frac{(2\pi\iu)^n}{(2\pi)^{n/2}}
[\pt]. 
\] 
Here, we used the Poincar\'{e} duality to identify 
$[\pt]\in H_0(\cX)$ with an element in $H^{2n}(\cX)$.  
Hence we have 
$\Omega_\tau = (2\pi)^{-n/2} z^{n/2} 
\cJ^{\rm CY}_\tau(1)$. 

(i) Let $\cX=X$ be a manifold and $C\subset X$ be a 
smooth curve of genus $g$ such that $[C]\cap \rho =0$.  
Then $[\cO_C(g-1)]$ gives an integral vector 
$A_C\in W_{-n+2} \cap \cV_{\Z,1}^\cX$ 
\[
A_C: = \Psi([\cO_C(g-1)]) = 
\frac{(2\pi\iu)^{n-1}}{(2\pi)^{n/2}} [C] 
\]
and an integral period 
\[
(z^{-\mu}A_C, \Omega_\tau)_{H_0} 
= - \frac{1}{2\pi\iu} [C]\cap \tau. 
\]

(ii) Let $\Psi([V])$ be any integral vector 
in $W_{-n+2}\cap \cV_{\Z,1}^\cX$.  
Since $\hGamma_\cX$ on the untwisted sector 
is of the form $1 - \gamma \rho + \text{higher degree}$  
($\gamma$ is the Euler constant), it follows that  
the $H^{2n}(\cX)$-component of $\Psi([V])$ belongs to 
$(2\pi)^{-n/2}(2\pi\iu)^n H^{2n}(\cX,\Q)=\Q A_0$. 
This implies that the component projection 
$W_{-n+2}\cap \cV_{\Z,1}^\cX
\to H^{2n}(\cX) = \C A_0$ is defined over $\Q$. 
Therefore, the condition (\ref{eq:curveclass_overQ}) 
holds for the $\hGamma$-integral structure. 
We have 
\[
(z^{-\mu}\Psi([V]), \Omega_{\tau})_{H_0} 
= \int_{\cX} \ch(V) - 
(a_0^{-1}\Psi([V]), \tau'_{\rm tw})_{\rm orb}
- \frac{1}{2\pi\iu}[C]\cap\tau_{0,2}.  
\]
for some $[C]\in H_2(X,\Z)\cap \Ker\rho$ 
and $a_0 = (2\pi)^{-n/2}(2\pi\iu)^n$. 

(iii) Let $y \in \cX$ be a possibly stacky point. 
Let $\clubsuit\colon \Aut(y) \to \End(V)$ be a 
finite dimensional representation of 
the automorphism group of $y$. 
This defines a coherent sheaf $\cO_y\otimes V$ 
supported on $y$ and an integral vector 
$A_{(y,V)} 
:= \Psi([\cO_y\otimes V]) \in 
\Ker(H^2(\cX)) \cap \cV_{\Z,1}^\cX$.  
Using Toen's Riemann-Roch formula \cite{toen}, 
one calculates 
\[
A_{(y,V)}  
= \frac{(2\pi\iu)^n}{(2\pi)^{n/2}}  
\sum_{(g)\subset \Aut(y)}  
\frac{(-1)^{n+n_{v(g)}+\iota_{v(g)}}\Tr(\clubsuit(g^{-1}))}
{|C(g)| \prod_{j=1}^{n-n_{v(g)}} \Gamma(f_{g,j})} 
[\pt]_{v(g)}, 
\]
where the sum is over all conjugacy classes 
$(g)$ of $g\in \Aut(y)$, 
$C(g)$ is the centralizer of $g$, 
$v(g)\in \sfT$ is the inertia component 
containing $(y,g) \in I\cX$, 
$[\pt]_{v(g)}$ is the homology class of a point 
on $\cX_{v(g)}$ (represented by a map $\pt \to \cX_v$ of stacks), 
$f_{g,1},\dots f_{g, n-n_{v(g)}}$ are 
rational numbers in $(0,1)$ such that 
$\{e^{2\pi\iu f_{g,j}}\}_j$ is a multi-set of 
the eigenvalues $\neq 1$ of 
the $g$ action on $T_y\cX$. 
The corresponding integral period has the asymptotic 
\begin{align*} 
(z^{-\mu}A_{(y,V)}, \Omega_\tau)_{H_0} 
 \sim  \frac{\dim(V)}{|\Aut(y)|} + 
\sum_{\substack{(g)\subset \Aut(y) \\ \iota_{v(g)}=1}} 
\frac{\Tr(\clubsuit(g))}
{|C(g)| \prod_{j=1}^{n-n_{v(g)}} \Gamma(1-f_{g,j})} 
[\pt]_{v(g)} \cap \tau_{\rm tw}  
\end{align*}  
in the large radius limit. This asymptotic is exact 
if $y\notin \cX_v$ for all $v$ with 
$\codim \cX_v = n-n_v \ge 3$ 
or equivalently, $A_{(y,V)} \in 
\Ker(H^2(\cX))\cap W_{-n+2}\cap \cV_{\Z,1}^\cX$. 
\end{example}

\subsection{Ruan's conjecture with integral structure} 
Yongbin Ruan's crepant resolution conjecture states that 
when $Y$ is a crepant resolution of  
the coarse moduli space $X$ of a Gorenstein orbifold $\cX$, 
\[
\pi\colon Y \to X, \quad \pi^*(K_X) = K_Y,  
\]
the quantum cohomology for $Y$ and the orbifold quantum 
cohomology for $\cX$ are related by analytic continuation
in the quantum parameters. 
See \cite{ruan-crc,bryan-graber,coates-ruan} for references. 
In the joint work \cite{CIT:I} 
with Coates and Tseng, 
in some examples of toric wall-crossings, 
we found the picture\footnote{
The symplectic transformation here also appeared 
in the work of Aganagic-Bouchard-Klemm \cite{ABK} 
and was also conceived by Ruan himself.} that  
the A-model \seminf VHS's of $Y$ and $\cX$ are 
connected by analytic continuation 
and that the two \seminf VHS's will match under 
a certain linear symplectic transformation 
$\U\colon \cH^\cX \to \cH^Y$.  
This symplectic transformation $\U$ encodes 
all the information on relationships between 
the genus zero Gromov-Witten theories of $\cX$ and $Y$. 
We refer the reader to \cite{coates-ruan} 
for a detailed discussion on 
the symplectic transformation and relationships 
to other versions of Ruan's conjecture. 
In this section, we incorporate integral structures 
into this picture and propose a possible relationship between 
the classical McKay correspondence and Ruan's conjecture. 

Mirror symmetry and the integral structure calculation 
in Section \ref{sec:intstr_via_toricmirrors} 
suggest the following refined picture 
involving $K$-groups: 
\begin{itemize}
\item[(a)] There exist ``natural" integral structures 
on the (algebraic) A-model \seminf VHS of $\cX$ and $Y$. 
The corresponding integral lattices in $\cV^\cX$ and $\cV^Y$ 
are given by the images of 
the $K$-groups of topological 
(resp. algebraic) orbifold vector bundles: 
\[
\Psi^\cX \colon K(\cX) \to \cV^\cX, \quad 
\Psi^Y \colon K(Y) \to \cV^Y.  
\]
In the discussion below, we do not need to assume that 
$\Psi^\cX$ and $\Psi^Y$ are defined by the 
same formula as $\Psi$ in the $\hGamma$-integral structure,  
but we assume that they satisfy the same conditions (i), (ii), (iii) 
in Definition-Proposition \ref{def-prop:A-model_int} 
as the $\hGamma$-integral structure satisfies.  

\item[(b)] 
There exists an isomorphism of $K$-groups 
\[
\U_K \colon K(\cX) \cong K(Y) 
\]
which preserves the Mukai pairing 
(as given in Definition-Proposition 
\ref{def-prop:A-model_int}) 
and commutes with the tensor by 
a topological (resp. algebraic) line bundle $L$ 
on the coarse moduli space of $\cX$, 
$\U_K(L \otimes \cdot) = \pi^*(L)\otimes \U_K(\cdot)$. 

\item[(c)] 
Via $\Psi^\cX$ and $\Psi^Y$, 
$\U_K$ induces an isomorphism 
$\U_\cV\colon \cV^\cX\cong \cV^Y$ preserving the pairing. 
By (\ref{eq:solutionmap_z}), 
$\cV^\cX$ (resp. $\cV^Y$) is identified with  
the space of multi-valued flat sections of a flat bundle 
$(\sfH^\cX,\hatnabla_{z\partial_z})$ 
(resp. $(\sfH^Y,\hatnabla_{z\partial_z})$) over $\C^*$. 
Because $\U_\cV$ commutes with the monodromy transformation 
in $z$, $\U_\cV$ induces a map of flat bundles 
$\U\colon (\sfH^\cX,\hatnabla_{z\partial_z}) 
\to (\sfH^Y,\hatnabla_{z\partial_z})$. 
This is considered as a $\C\{z,z^{-1}\}$-linear 
symplectic isomorphism $\U \colon \cH^\cX\to \cH^Y$ 
(with respect to the symplectic form (\ref{eq:pairing_H})) 
making the following diagram commute: 
\[
\begin{CD}
K(\cX) @>{\U_K}>> K(Y) \\ 
@V{z^{-\mu}z^{\rho}\Psi^\cX}VV  @VV{z^{-\mu}z^\rho \Psi^Y}V  \\
\Gamma(\widetilde{\C^*},\sfH^\cX) @>{\U}>> 
\Gamma(\widetilde{\C^*},\sfH^Y). 
\end{CD} 
\]
The isomorphism $\U$ so defined 
sends the A-model \seminf VHS 
$\F^\cX_\tau\subset \cH^\cX$ of $\cX$ to 
that $\F^Y_\tau \subset \cH^Y$ of $Y$, \emph{i.e.} 
\[
\U (\F^\cX_\tau) = \F^Y_{\Upsilon(\tau)} 
\]
where $\Upsilon$ is a map 
from a subdomain of $H_{\rm orb}^*(\cX)$ 
to a subdomain of $H^*(Y)$ where the quantum cohomology 
of $\cX$ and $Y$ can be analytically continued respectively.  
\end{itemize} 

\begin{remark}
The isomorphism of $K$-groups (or even the equivalence 
of derived categories of coherent sheaves) 
are studied in the context of McKay correspondence 
and usually given by a Fourier-Mukai transformation. 
We expect that the isomorphism $\U_K$ in (b) will be 
given as a Fourier-Mukai transformation. 
In fact, 
Borisov-Horja \cite{borisov-horja-FM} showed  
that an analytic continuation of solutions to 
the GKZ-system corresponds to a Fourier-Mukai 
transformation between $K$-groups of 
toric Calabi-Yau orbifolds. 
We can also ask if the integral structures 
have the same ``functoriality" as the $K$-theory has.  
In this viewpoint, the map $\Psi^\cX$ will play a role of 
``natural transformation" from 
$K$-theory integral structures 
to quantum cohomology. 
\end{remark} 

We discuss what follows from this picture, 
assuming $\cX$ is weak Fano, \emph{i.e.} $c_1(\cX)$ is nef. 
Since this picture contains 
the suggestions we made in \cite{CIT:I},  
it in particular implies that 
quantum cohomology of $\cX$ and $Y$ 
are identified via $\Upsilon$ and $\U$ 
\emph{as a family of algebras} 
(\emph{not necessarily as Frobenius manifolds}). 
However, the large radius limit points for 
$\cX$ and $Y$ are not identified under $\Upsilon$, 
so we need analytic continuations indeed. 
We will not repeat the argument in \cite{CIT:I, coates-ruan} 
on the isomorphism of quantum cohomology algebras here. 
Let us first observe that integral periods of $\cX$ and $Y$ 
in the conformal limit match under $\Upsilon$ and $\U$ 
(see (\ref{eq:matching_integralperiods}) below). 
Because $\U_K$ commutes with the tensor 
by a line bundle pulled back from $X$, 
it follows that $\U$ must commute with $H^2(\cX)$ 
((b), Conjecture 4.1 in \cite{coates-ruan};  
(b), Section 5 in \cite{CIT:I}), \emph{i.e.} 
\begin{equation}
\label{eq:U_commutes_H2}
\U (\alpha\cup \cdot) = \pi^*(\alpha) \cup \U(\cdot), \quad 
\alpha \in H^2(\cX).   
\end{equation} 
By definition, $\U$ commutes with 
$\hatnabla_{z\partial_z}$-action 
on $\cH^\cX$ and $\cH^Y$. 
Hence by (\ref{eq:U_commutes_H2}) and (\ref{eq:conn_z-direction}),   
\[
\U \circ (z\partial_z + \mu^\cX) = (z\partial_z +\mu^Y) \circ \U 
\]
\emph{i.e.} $\U$ is degree-preserving. 
Since $\cX$ is weak Fano, 
by the discussion leading to 
Theorem 8.2 in \cite{coates-ruan} 
(essentially using Lemma 5.1 \emph{ibid.}),  
we know that $\Upsilon$ should map $H^2_{\rm orb}(\cX)$ 
to $H^2(Y)$:  
\[
\Upsilon(H^2_{\rm orb}(\cX)) \subset H^2(Y). 
\]
The conformal limit $\tau \to \tau -s \rho$, $\Re(s)\to \infty$ 
on $H^2_{\rm orb}(\cX)$ should also be mapped to 
the conformal limit on $H^2(Y)$ under $\Upsilon$ 
because this flow is generated by the Euler vector field 
and the two Euler vector fields should match under $\Upsilon$  
(the Euler vector field is a part of the data of \seminf VHS).  
Therefore, by (\ref{eq:U_commutes_H2}) and 
$\pi^*c_1(\cX) = c_1(Y)$, 
the \seminf VHSs appearing in the conformal limit 
(see Definition \ref{def:CY_limit_VHS}) 
also match under $\U$: 
\[
\U(\F_{\tau}^{\cX, {\rm CY}}) = \F_{\Upsilon(\tau)}^{Y, {\rm CY}}.  
\] 
In particular, 
the finite dimensional VHS's 
$(F^{\cX,\bullet}_\tau\subset H_0^\cX)$, 
$(F^{Y,\bullet}_\tau \subset H_0^Y)$ 
associated with these also match:  
\[
\U(F^{\cX,\bullet}_{\tau}) = F^{Y,\bullet}_{\Upsilon(\tau)}, \quad 
\U \colon \cH^\cX \supset H_0^{\cX} 
\to H_0^{Y} \subset \cH^Y.  
\] 
We used that $\U$ induces a map from 
$H_0^{\cX} = \Ker(z\partial_z + \mu^{\cX})$ 
to $H_0^Y = \Ker(z\partial_z + \mu^Y)$.  
Let $L$ be an ample line bundle on $X$.  
Consider the weight filtration $W_k^\cX$ 
(\ref{eq:weightfiltr_ampleLoncoarse}) 
on $\cV^\cX$ defined by 
the Galois action logarithm 
$-2\pi\iu c_1(L)$. 
The first term $W_{-n}^\cX$ of the weight filtration 
is given by $\Image(c_1(L)^n)$. 
Thus $\U_\cV(W_{-n}^\cX) = \Image(\pi^*(c_1(L))^n) 
= H^{2n}(Y)$. Note that $\pi^*(c_1(L))^n$ is non-trivial   
since $\pi \colon Y\to X$ is birational. 
Therefore, for the weight filtration $W_k^Y$ on $\cV^Y$ 
(defined similarly by the Galois action logarithm 
corresponding to an ample line bundle on $Y$), 
we have 
\[
\U_{\cV}(W_{-n}^\cX) = W_{-n}^Y. 
\]
As we did before, we use an integral vector 
$A_0^\cX$ (unique up to sign)  
in $W_{-n}^\cX\cap \cV_{\Z,1}^\cX$ 
to normalize a generator $\Omega^{\cX}_\tau 
\in F^{\cX,n}_\tau$ and then 
use $A_0^Y:=\U_\cV(A_0^\cX)\in W_{-n}^Y \cap \cV_{\Z,1}^Y$ 
to normalize $\Omega^Y_\tau \in F^{Y,n}_\tau$  
(see (\ref{eq:normalization_Omega})).  
Because the $\U$ preserves the pairing, we have
\[
\U(\Omega_\tau^\cX) = \Omega_{\Upsilon(\tau)}^Y.  
\]
When $A^\cX \in \cV_{\Z,1}^\cX = \cV_{\Z}^\cX\cap \Ker(c_1(\cX))$, 
the corresponding vector $A^Y =\U_\cV(A^\cX)$ 
belongs to $\cV_{\Z}^Y \cap \Ker(\pi^*(c_1(\cX))) 
= \cV_{\Z,1}^Y$ and the integral periods match 
\begin{equation}
\label{eq:matching_integralperiods}
(z^{-\mu}A^\cX, \Omega_\tau^\cX)_{H_0^\cX} = 
(z^{-\mu}A^Y, \Omega_{\Upsilon(\tau)}^Y)_{H_0^Y}.  
\end{equation} 

Now we can make predictions   
on the specialization values of quantum parameters. 
$\Ker(\pi^*H^2(\cX))\subset \cV^Y$ 
is defined over $\Q$ since this is the intersection of 
$\Ker(\id-G^\cV(\pi^*\xi))$ over integral class $\xi\in H^2(X,\Z)$. 
Take a basis $A_0^Y,A_1^Y,\dots,A_\natural^Y$ of 
$\Ker(\pi^*H^2(\cX))\cap W_{-n+2}^Y \cap \cV^Y_{\Z,1}$.  
These generate a full lattice in $H^{2n}(Y) \oplus 
(H^{2n-2}(Y) \cap \Ker \pi_*)$ over $\C$.  
By Proposition \ref{prop:integralperiods_basic}, 
the integral periods for 
$A_1^Y,\dots,A_\natural^Y$ are of the form:
\begin{align}
\label{eq:exccurve_intperiod}
(z^{-\mu}A_i^Y, \Omega_{\tau}^Y)_{H_0^Y} = 
a_0^{-1} a_i - \frac{1}{2\pi\iu} [C_i]\cap \tau,
\quad  a_i :=(A_i^Y, 1),  
\end{align} 
for some $[C_i]\in H_2(Y,\Z)\cap \Ker\pi_*$. 
$[C_1],\dots,[C_\natural]$ are a $\Q$-basis of 
$H_2(Y,\Q)\cap \Ker\pi_*$, so these form 
an affine co-ordinate system on $H^2(Y)/\Image\pi^*$. 
The integral vector $A_i^\cX$ corresponding to $A_i^Y$ 
belongs to $\Ker(H^2(\cX)) \cap \cV_{\Z,1}^\cX$. 
From (\ref{eq:matching_integralperiods}),  
Proposition \ref{prop:integralperiods_basic} and 
Example \ref{ex:hGamma}, 
our picture leads to the following prediction: 

\begin{itemize}
\item[(i)] 
Assume that $H^*(\cX)$ is generated by $H^2(\cX)$ 
and that the condition (\ref{eq:sep_trivial}) is satisfied.  
Then the integral periods of $Y$ of the form 
(\ref{eq:exccurve_intperiod}) take rational values 
at the large radius limit point of $\cX$. 

\item[(ii)] 
Assume in addition to (i) that 
the condition (\ref{eq:curveclass_overQ}) holds 
for $Y$. Then $a_0^{-1} a_i$ above is rational, so 
the ``quantum parameter" $q_C := \exp([C]\cap \tau)$ with 
$[C]\in H_2(Y,\Z)\cap \Ker\pi_*$ for $Y$ 
specializes to a root of unity at the 
large radius limit point of $\cX$. 


\item[(iii)] 
Let $C\subset Y$ be a smooth rational curve in the 
exceptional set. 
Assume in addition to (iii) that 
$\U_K^{-1}$ sends $[\cO_C(-1)]\in K(Y)$ to 
$[\cO_x\otimes V]\in K(\cX)$ for $x = \pi(C)$ and 
some representation $V$ of $\Aut(x)$. 
Then the quantum parameter $q_C$ 
specializes to 
$\exp(-2\pi\iu \dim V/|\Aut(x)|)$ 
at the large radius limit point of $\cX$. 
\end{itemize} 

For the $A_n$ singularity resolution, 
each irreducible curve in the exceptional set 
corresponds to a one-dimensional 
irreducible representation of $\Z/(n+1)\Z$ 
under McKay correspondence. 
If we use this McKay correspondence as $\U_K$, 
the prediction of specialization values 
made in (iii) is true \cite{CCIT:An}.  
Also, under the McKay correspondence, (iii) gives 
the same prediction (up to complex conjugation)  
made by Bryan-Graber \cite{bryan-graber} and 
Bryan-Gholampour \cite{bryan-gholampour}
for the ADE surface singularities   
and $\C^3/G$ with a finite subgroup $G\subset SO(3)$. 

\section{Appendix} 
\subsection{Proof of (\ref{eq:Birkhoff_withestimate})} 
\label{subsec:orderestimate_Birkhoff} 
Birkhoff's theorem implies that 
there exists an open dense neighborhood of $\unit$ 
in the loop group $LGL_N(\C)$ 
which is diffeomorphic to the product of subgroups
$L^+_1GL_N(\C)\times L^-GL_N(\C)$ \cite{pressley-segal}.  
We use the inverse function theorem for Hilbert manifolds 
to explain the order estimate in 
(\ref{eq:Birkhoff_withestimate}). 
Consider the space $LGL_N(\C)^{1,2}$ of Sobolev loops 
which consists of maps $\lambda\colon S^1 \to GL_N(\C)$ such 
that $\lambda$ and its weak derivative $\lambda'$ are 
square integrable. 
Note that this is a subgroup of 
the group of continuous loops by 
Sobolev embedding theorem 
$W^{1,2}(S^1) \subset C^0(S^1)$ and 
the multiplication theorem 
$W^{1,2}(S^1)\times W^{1,2}(S^1) \to W^{1,2}(S^1)$. 
$LGL_N(\C)^{1,2}$ is a Hilbert manifold 
modeled on the Hilbert space $W^{1,2}(S^1,\gl_N(\C))$. 
A co-ordinate chart of a neighborhood of $\unit$ is 
given by the exponential map $A(z) \mapsto e^{A(z)}$. 
Let $L^+_1GL_N(\C)^{1,2}$ be the subgroup of 
$LGL_N(\C)^{1,2}$ consisting of the boundary values 
of holomorphic maps $\lambda_+ \colon 
\{|z|<1\} \to GL_N(\C)$ satisfying $\lambda_+(0)=\unit$. 
Let $L^-GL_N(\C)^{1,2}$ be the subgroup of 
$LGL_N(\C)^{1,2}$ consisting of the boundary values 
of holomorphic maps $\lambda_- \colon 
\{|z|>1\}\cup\{\infty\} \to GL_N(\C)$. 
Notice that $W^{1,2} := W^{1,2}(S^1,\gl_N(\C))$ has 
the direct sum decomposition: 
\begin{equation}
\label{eq:Liealgdecomp_pos_neg}
W^{1,2} = W^{1,2}_+ \oplus W^{1,2}_-, \quad  
\end{equation} 
where $W^{1,2}_+$ ($W^{1,2}_-$) 
is the closed subspace of 
$W^{1,2}(S^1,\gl_N(\C))$ consisting of 
strictly positive Fourier series $\sum_{n>0} a_n z^n$ 
(non-positive Fourier series $\sum_{n\le 0} a_nz^n$ resp.) 
with $a_n\in \gl_N(\C)$. 
The subgroups $L^+_1GL_N(\C)^{1,2}$ and $L^-GL_N(\C)^{1,2}$ are 
modeled on the Hilbert spaces $W^{1,2}_+$ and 
$W^{1,2}_-$ respectively. 
Consider the multiplication map 
$L^+_1GL_N(\C)^{1,2} \times L^-GL_N(\C)^{1,2}
\to LGL_N(\C)^{1,2}$. 
The differential of this map at the identity is given by 
the sum $W^{1,2}_+ \times W^{1,2}_- \to W^{1,2}$ 
and is clearly an isomorphism. 
By the inverse function theorem for Hilbert manifolds, 
there exists a differentiable inverse map 
on a neighborhood of $\unit$.  
In the case at hand, we have  
$\|(B_t^{-1}Q_t B_t) (C_t \ov{Q}_t C_t^{-1})-\unit\|_{W^{1,2}}
= O(e^{-\epsilon t})$ as $t\to \infty$.  
Therefore, this admits the Birkhoff factorization 
(\ref{eq:Birkhoff_withestimate}) for $t\gg 0$ with 
$\|\tilde{B}_t-\unit\|_{W^{1,2}} = O(e^{-\epsilon t})$ and 
$\|\tilde{C}_t - \unit\|_{W^{1,2}} = O(e^{-\epsilon t})$.  
By Sobolev embedding, the order estimates hold also for 
the $C^0$-norm. 
(The method here does not work directly 
for the Banach manifold of continuous loops, 
since the decomposition (\ref{eq:Liealgdecomp_pos_neg}) 
is not true in this case.)

\subsection{Proof of Lemma \ref{lem:PScond}} 
\label{subsec:proof_PS} 
Let $B \subset \cMo\times \C^*$ be a compact set. 
We need to show that $B' = \{(q,z,y)\;;\; (q,z)\in B, \ y\in Y_q, \ 
\|df_{q,z}(y)\|\le \epsilon \}$ is compact. 
Assume that there exists a divergent 
sequence $\{(q_{(k)},z_{(k)}, y_{(k)})\}_{k=0}^\infty$ in $B'$, 
\emph{i.e.} any subsequence of it does not converge. 
Take an arbitrary Hermitian norm $\|\cdot\|$ on $\bN\otimes \C$. 
Note that we have 
\[
\|df_{q,z}(y)\| = \frac{1}{|z|} 
\|\sum_{i=1}^m q^{\ell_i} y^{b_i}b_i\|.   
\]
By passing to a subsequence and renumbering $b_1,\dots,b_m$, 
we can assume that $q_{(k)}$ and $z_{(k)}$ converge and that 
$|y_{(k)}^{b_1}| \ge |y_{(k)}^{b_2}| \ge \cdots \ge |y_{(k)}^{b_m}|$
for all $k$. 
Since $0$ is in the interior of $\hat{S}$, 
there exist $c_i>0$ such that $\sum_{i=1}^m c_i b_i=0$. 
Hence $\prod_{i=1}^m |y_{(k)}^{b_i}|^{c_i} = 1$. 
Because $y_{(k)}$ diverges, 
we must have $\lim_{k\to\infty} |y_{(k)}^{b_1}| = \infty$. 
Since $\|df_{q_{(k)},z_{(k)}}(y_{(k)})\|$ is bounded, we have 
\[
0 = \lim_{k\to \infty} \frac{|z_{(k)}|}{|y_{(k)}^{b_1}|} 
\|df_{q_{(k)},z_{(k)}}(y_{(k)})\|
= \lim_{k\to \infty} \| \sum_{i=1}^m q_{(k)}^{\ell_i} y_{(k)}^{b_i-b_1} b_i \|.  
\]
Because $|y_{(k)}^{b_i-b_1}|\le 1$, 
by passing to a subsequence again, 
we can assume that $y_{(k)}^{b_i-b_1}$ converges to $\alpha_i\neq 0$ 
for all $1\le i\le l$ and $y_{(k)}^{b_i-b_1}$ goes to $0$ for 
$i> l$. Then we have 
\[
0= \sum_{i=1}^l \tilde{q}^{\ell_i} \alpha_i b_i, \quad 
\tilde{q}= \lim_{k\to\infty} q_{(k)} \in \cMo.  
\]
Put $\xi_{(k),i}:=\log y_{(k),i}$. 
By choosing a suitable branch of the logarithm, 
we can assume that 
$\lim_{k\to \infty} \pair{\xi_{(k)}}{b_i-b_1} = \log \alpha_i$ 
for $1\le i\le l$ 
and $\lim_{k\to \infty} \pair{\Re(\xi_{(k)})}{b_i-b_1} = -\infty$ 
for $i> l$. 
Let $V$ be the $\C$ subspace of $\bN\otimes \C$ 
spanned by $b_i-b_1$ with $1\le i\le l$. 
Take the orthogonal decomposition $\bN\otimes \C \cong V\oplus V^\perp$
and write $\xi_{(k)}=\xi_{(k)}'+\xi_{(k)}''$, 
where $\xi'_{(k)}\in V$ and $\xi''_{(k)}\in V^\perp$. 
Then $\xi_{(k)}'$ converges to some $\xi'\in V$. 
Putting $\tilde{y}_i = \exp(\xi'_i)$, we have 
$\tilde{y}^{b_i-b_1}=\alpha_i$ for $1\le i\le l$ and so 
\begin{equation}
\label{eq:crit_W_Delta}
\sum_{i=1}^l \tilde{q}^{\ell_i} \tilde{y}^{b_i} b_i  
= \tilde{y}^{b_1}
(\sum_{i=1}^l \tilde{q}^{\ell_i} \tilde{y}^{b_i-b_1} b_i) 
=0.    
\end{equation} 
On the other hand, for a sufficiently big $k$, 
$\pair{\Re(\xi''_{(k)})}{b_i-b_1}=0$ for $1\le i\le l$ 
and $\pair{\Re(\xi''_{(k)})}{b_i-b_1}<0$ for $i>l$. 
This means that $b_1,\dots,b_l$ are on some face $\Delta$ of $\hat{S}$. 
But the equation (\ref{eq:crit_W_Delta}) shows that 
$\tilde{y}$ is a critical point of $W_{\tilde{q},\Delta}$.  
This contradicts to the assumption that $W_{\tilde{q}}$ 
is non-degenerate at infinity.

\bibliographystyle{amsplain}

\end{document}